\documentclass{amsart}
\usepackage{amsmath,amssymb,amsthm,graphicx,amsfonts,tikz,mathrsfs,url}
\usepackage{color}
\usepackage[unicode,colorlinks=true,linkcolor=red,citecolor=teal]{hyperref}
\usepackage{amsxtra}
\usepackage[all,2cell,ps]{xy}
\usepackage{thmtools} 
\usepackage{thm-restate}

\theoremstyle{plain}
\newtheorem{thm}{Theorem}[section]

\newtheorem{lem}[thm]{Lemma}
\newtheorem{prop}[thm]{Proposition}

\newtheorem{cor}[thm]{Corollary}

\newtheorem{conjec}[thm]{Conjecture}

\newtheorem{qtn}[thm]{Question}

\theoremstyle{definition}
\newtheorem{defn}[thm]{Definition}

\newtheorem{ex}[thm]{Example}

\newtheorem{rem}[thm]{Remark}

\theoremstyle{remark}

\newcommand{\bbC}{\mathbb{C}}

\newcommand{\bbF}{\mathbb{F}}

\newcommand{\bbH}{\mathbb{H}}

\newcommand{\bbK}{\mathbb{K}}

\newcommand{\bbN}{\mathbb{N}}

\newcommand{\bbQ}{\mathbb{Q}}

\newcommand{\bbZ}{\mathbb{Z}}
\DeclareMathOperator{\Spec}{Spec}

\newcommand{\al}{\alpha}

\newcommand{\de}{\delta}

\newcommand{\ep}{\epsilon}

\DeclareMathOperator{\Un}{U}

\DeclareMathOperator{\SL}{SL}
\DeclareMathOperator{\PSL}{PSL}

\DeclareMathOperator{\Id}{Id}
\DeclareMathOperator{\id}{id}

\DeclareMathOperator{\Ad}{Ad}
\DeclareMathOperator{\Tr}{Tr}

\DeclareMathOperator{\Aut}{Aut}

\DeclareMathOperator{\Hom}{Hom}
\DeclareMathOperator{\End}{End}
\DeclareMathOperator{\Mod}{Mod}

\DeclareMathOperator{\Br}{Br}

\DeclareMathOperator{\im}{Im}

\DeclareMathOperator{\ord}{ord}

\newenvironment{pf}{\begin{proof}}{\end{proof}}

\usetikzlibrary{arrows}

\title{Azumaya algebras and once-punctured torus bundles}

\author{Nicholas Miller}
\address{Department of Mathematics\\University of Oklahoma\\Norman, OK 73019}
\email{nickmbmiller@ou.edu}

\begin{document}

\begin{abstract}
Recent work of Chinburg, Reid, and Stover has shown that certain arithmetic and algebro-geometric properties of the character variety of a hyperbolic knot complement in the $3$-sphere $M=S^3\setminus K$ yields topological and number theoretic information about Dehn fillings of $M$.
Specifically, they show how the study of a certain extension problem for quaternion Azumaya algebras is related to topological invariants associated to these fillings.

In this paper, we extend their work to the setting of hyperbolic once punctured torus bundles.
Along the way, we exhibit new phenomenon in the relevant extension problem not visible in the case of a hyperbolic knot complement, which is related to the more complicated non-abelian reducible representation theory of hyperbolic once punctured torus bundles.
We then apply these results to a series of examples from the literature and list some remaining questions from both works.
\end{abstract}

\maketitle

%~~~~~~~~~~~~~~~~~~~~~~~~~~~~~~~~~~~~~~~~~~~~~~~~~~~~~~~~~~~
%~~~~~~~~~~~~~~~~~~~~~~~~~~~~~~~~~~~~~~~~~~~~~~~~~~~~~~~~~~~
%~~~~~~~~~~~~~~~~~~~~~~~~~~~~~~~~~~~~~~~~~~~~~~~~~~~~~~~~~~~

\section{Introduction}
Over the past several decades there have been a plethora of results exploring the rich connection between the geometry of a finite-volume hyperbolic 3-manifold $M$ and properties of its space of representations into $\SL_2(\bbC)$ modulo conjugation, which is frequently referred to as the $\SL_2(\bbC)$--character variety.
For instance, Thurston \cite{Thurston} showed that any canonical component of the $\SL_2(\bbC)$--character variety, that is any component containing (a lift of) the discrete faithful representation, has complex dimension equal to the number of cusps of $M$ and Culler--Shalen \cite{CS1} have shown that the character variety can be used to find embedded essential surfaces in $M$.

More recently, Chinburg--Reid--Stover \cite{CRS} have explored a connection between arithmetic properties of the $\SL_2(\bbC)$--character variety and certain topological invariants associated to Dehn fillings of knot complements in the $3$-sphere.
In this paper, we extend this connection to the setting of hyperbolic once punctured torus bundles.
To motivate our theorems, we begin by recalling some results from \cite{CRS}.

Let $M$ be a finite volume hyperbolic manifold and $\rho:\pi_1(M)\to\SL_2(\bbC)$ be any representation.
Then we may associate to $\rho$ the following invariants
$$k_\rho=\bbQ(\{\Tr(\gamma)\mid\gamma\in\pi_1(M)\}),$$
and 
\begin{equation}\label{eqn:quatalg}
\mathcal{A}_\rho=\{\sum a_i\rho(\gamma_i)\mid a_i\in k_\rho,\gamma_i\in\pi_1(M)\},
\end{equation}
where the sum in Equation \eqref{eqn:quatalg} is assumed to be finite.
When $\rho$ is absolutely irreducible, $\mathcal{A}_\rho$ defines a quaternion algebra over the field $k_\rho$ and, moreover, when $\rho$ is a representation of geometric significance, the invariants $\mathcal{A}_\rho$, $k_\rho$ tend to be important topological invariants.
For instance, when $\rho$ is (a lift of) the discrete faithful representation of $\pi_1(M)$, Mostow rigidity implies that $\mathcal{A}_\rho$, $k_\rho$  are topological invariants which are intimately related to lengths of closed geodesic on $M$ and, in many natural situations, are invariants of the commensurability class of $M$.
Similarly, if $M$ is a one cusped hyperbolic manifold and $\rho$ denotes the composition $\rho_{p/q}:\pi_1(M)\twoheadrightarrow\pi_1(M_{p/q})\to\SL_2(\bbC)$, where $M_{p/q}$ denotes its $p/q$-Dehn filling, then $A_{\rho_{p/q}}$, $k_{\rho_{p/q}}$ are topological invariants of the Dehn filled orbifold $M_{p/q}$.

When $M$ is the figure eight knot complement, it was observed by Chinburg, Reid, and Stover that the invariants $\mathcal{A}_{\rho_{p/q}}$ behave in a very structured way as $p/q$ varies through all Dehn fillings of $M$.
Precisely, recall that for a $p/q$-Dehn filling $k_{\rho_{p/q}}$ is a number field and therefore $\mathcal{A}_{\rho_{p/q}}$ is completely characterized by the places of $k_{\rho_{p/q}}$ (both finite and infinite) at which $\mathcal{A}_{\rho_{p/q}}$ ramifies. 
If we then define
$$S_M=\{\ell\in\bbN\mid\ell\text{ is prime and }\exists~ p/q\text{ such that }\mathcal{A}_{\rho_{p/q}}\text{ ramifies above }\ell\},$$
then it was observed in \cite[Thm 1.7]{CRS} that $S_M=\{2\}$.
In particular, strikingly the associated quaternion algebras $\mathcal{A}_{p/q}$ exhibit structure in their ramification sets as $p/q$ vary that a priori there is no reason to expect.

Motivated by understanding this phenomenon in the setting of hyperbolic knot complements in the $3$-sphere, $M=S^3\setminus K$, Chinburg, Reid, and Stover show that finiteness of the above set $S_M$ is related to an extension problem for a certain Azumaya algebra which one can naturally associate to a canonical component.
They then study this extension problem and show that one can give a sufficient condition for the extension problem to be solved coming from properties of the Alexander polynomial $\Delta_{K}(t)$. 
To discuss this and extensions therein, we briefly recall a few facts about curves of characters of representations, where we opt for informality in the introduction and postpone the formal definitions until Section \ref{sec:background}. 
In this paper, we use the language of varieties and schemes over number fields as it is the natural setting to discuss Azumaya algebras.

Let $\frak{C}$ denote a fixed geometrically integral curve component of the character scheme $\frak{X}(M)_k$, defined over a field $k$, which contains the character of an irreducible representation.
Then Culler--Shalen \cite{CS1} associate to $\frak{C}$ a certain representation $\rho_\frak{C}:\pi_1(M)\to \SL_2(F)$ called the tautological representation, where $F$ is an extension of the function field $k(\frak{C})$.
From this, one can form a quaternion algebra $\mathcal{A}_{k(\frak{C})}$ using precisely the same definition in Equation \eqref{eqn:quatalg}, with $k_\rho$ replaced by $k(\frak{C})$.
As $\frak{C}$ is a curve, it admits a smooth projective completion $\widetilde{C}$ with the property that $k(\widetilde{C})\cong k(\frak{C})$ and therefore, under this isomorphism, we have a quaternion algebra $\mathcal{A}_{k(\widetilde{C})}$ over $k(\widetilde{C})$ which we call the \emph{tautological Azumaya algebra}.
The aforementioned extension problem is then the question of whether there exists an Azumaya algebra $\mathcal{A}_{\widetilde{C}}$ over the entirety of $\widetilde{C}$ such that its specialization to the generic point $\xi$ of $\widetilde{C}$ is isomorphic to the tautological Azumaya algebra $\mathcal{A}_{k(\widetilde{C})}$.
If this is indeed the case, we say that $\mathcal{A}_{k(\widetilde{C})}$ extends over $\widetilde{C}$.
This formulation will be made more precise in Section \ref{sec:azalg}.

When $M=S^3\setminus K$ is a hyperbolic knot complement in $S^3$, Chinburg--Reid--Stover show that the tautological Azumaya algebra extends over $\widetilde{C}$ provided that the Alexander polynomial $\Delta_{K}(t)$ satisfies the following condition:

\medskip
\begin{enumerate}
\item[$(\star)$] For any root $z$ of $\Delta_{K}(t)$ in $\overline{\bbQ}$, if $w$ is such that $w^2=z$ then $\bbQ(w)=\bbQ(w+1/w).$ \medskip
\end{enumerate}

\noindent The key tool in analyzing the extension problem is a cohomological invariant associated to a point $p\in\widetilde{C}$ called the tame symbol, which controls precisely when the tautological Azumaya algebra can be extended over a point $p$.
In \cite{CRS}, the authors are able to show that this Azumaya algebra always extends over any point in $\widetilde{C}$ except at those points given by characters of non-abelian reducible representations, from which Condition $(\star)$ arises naturally by analyzing the tame symbol at such representations.

The main thrust of this paper is to generalize this theorem to the context of hyperbolic once punctured torus bundles $M$. 
We let $x$, $y$ denote generators of the once punctured torus fiber of $M$ and $t$ denote the stable letter.
See Section \ref{sec:1ptb} for a presentation of $\pi_1(M)$ in these three generators.
In contrast to the knot complement case, where meridians always have the same eigenvalues under any non-abelian reducible representation, there are three distinct types of non-abelian reducible representations of $\pi_1(M)$ based on whether the stable letter lies in the center, is quasi-unipotent and infinite order, or neither. We call the latter \emph{Type A representations}, the second \emph{Type B representations}, and the first \emph{Type C representations} (see Section \ref{sec:nonabclassification}).
If $\Phi\in\SL_2(\bbZ)$ denotes the class of the monodromy matrix of $M$ in the mapping class group of the once punctured torus and $k$ is a field extension of $\bbQ$, we define the following two conditions:\medskip
\begin{enumerate}
\item[$(\star_1)$] Either $\Tr(\Phi)=a^2+2$ for some $a\in\bbN$ or $w-1/w\in k$, where $w$ is such that $w^2$ is a root of the characteristic polynomial for $\Phi$,\medskip
\item[$(\star_2)$] If $n$ is the least common multiple of the orders of $x$, $y$ in $H_1(M,\bbZ)_{tors}$, then $\eta_n-\eta_n^{-1}\in k$.\medskip
\end{enumerate}

We remark that Condition $(\star_1)$ should be seen as the natural interpretation of Condition $(\star)$ in the context of once punctured torus bundles.
That is, when $M=S^3\setminus K$, the Alexander polynomial $\Delta_{K}(t)$ is the characteristic polynomial of the action of the generator of $H_1(M,\bbZ)\cong \bbZ$ on $H_1(M^{cyc},\bbZ)$, where $M^{cyc}$ is the infinite cyclic cover of $M$ associated to the kernel of the map to homology.
Similarly, for a once punctured torus bundle we have that $H_1(M,\bbZ)\cong\bbZ\oplus F$ for some finite group $F$ and one can analyze the action of the generator of the first factor (given by the stable letter) on $H_1(M^{cyc},\bbZ)$, where $M^{cyc}$ is associated to the kernel of projection to the first factor.
The characteristic polynomial of this is given by the characteristic polynomial of the monodromy matrix $\Phi$ and analyzing the identical condition that $\bbQ(w)=\bbQ(w+1/w)$ naturally leads to the trace condition (see Proposition \ref{prop:galois}).
Condition $(\star_2)$ is a genuinely new phenomenon which does not occur in the hyperbolic knot complement case, as it is related to the presence of torsion in $H_1(M,\bbZ)$ from the second factor, and ends up making the non-abelian reducible representation theory for $\pi_1(M)$ richer for hyperbolic once punctured torus bundles.

Using these two conditions, we then have the analogous theorem to that of \cite[Thm 1.2]{CRS}.

\begin{restatable}{thm}{mainone}\label{thm:main1}
Let $M$ be a hyperbolic once punctured torus bundle and let $\frak{C}$ be a geometrically irreducible component of $\frak{X}(M)_k$ containing the character of an irreducible representation.
Suppose that if there is a character of a Type A representation on $\frak{C}$, Condition $(\star_1)$ holds, and if there is a character of a Type B or C representation on $\frak{C}$, Condition $(\star_2)$ holds.
Then the tautological Azumaya algebra $\mathcal{A}_{k(\widetilde{C})}$ extends over $\widetilde{C}$. 
\end{restatable}

As a consequence of Theorem \ref{thm:main1}, we show the following corollary, which is the analogue of \cite[Thm 1.4]{CRS}, where we recall that a canonical component of $\frak{X}(M)_k$ is any component containing the character of a lift of the discrete faithful representation.

\begin{cor}\label{cor:main1}
Suppose that $M$ be a hyperbolic once punctured torus bundle and suppose that $\frak{C}$ is a canonical component which fits the hypothesis of Theorem \ref{thm:main1}.
Then there exists a finite set of rational primes $S_M$ such that each finite place of $k_{p/q}$ which ramifies in $A_{M_{p/q}}$ lies over a prime of $S_M$. 
\end{cor}

As in \cite{CRS}, we will then go on to prove the converse to Theorem \ref{thm:main1}.

\begin{restatable}{thm}{maintwo}\label{thm:main2}
Let $\frak{C}$ be an irreducible component of $\frak{X}(M)_k$ containing the character of an irreducible representation $\chi_\rho$ for which $\rho\vert_{\langle x,y\rangle}$, $\rho\vert_{\langle x,t\rangle}$, and $\rho\vert_{\langle y,t\rangle}$ are irreducible.
Suppose that either there is a character of a Type A representation on $\frak{C}$ and Condition $(\star_1)$ does not hold or there is a character of a Type B or C representation on $\frak{C}$ and Condition $(\star_2)$ does not hold.
Then $\mathcal{A}_{k(\widetilde{C})}$ does not extend over $\widetilde{C}$.
\end{restatable}

Note that in the above theorem we require a certain irreducibility condition on representations $\rho$ with a character on $\frak{C}$ which we did not need in Theorem \ref{thm:main1}.
This owes to the fact that we need more delicate control over a potential Hilbert symbol for the tautological Azumaya algebra $\mathcal{A}_{k(\widetilde{C})}$.
However, as we will see in Corollary \ref{cor:canoncomp}, the condition on $\rho$ from Theorem \ref{thm:main2} holds for any (lift of a) discrete faithful representation and so we have the following consequence.

\begin{restatable}{cor}{maintwoconv}\label{cor:maintwocon}
Let $\frak{C}$ be any canonical component of $\frak{X}(M)_k$.
Suppose that either there is a character of a Type A representation on $\frak{C}$ and Condition $(\star_1)$ does not hold or there is a character of a Type B or C representation on $\frak{C}$ and Condition $(\star_2)$ does not hold.
Then $\mathcal{A}_{k(\widetilde{C})}$ does not extend over $\widetilde{C}$.
\end{restatable}

As we will explain at the beginning of Section \ref{sec:convthm}, the key difficulty in proving the converse theorem is that triviality of the tame symbol is easier to certify than non-triviality.
That is to say, after the preliminary work of understanding non-abelian reducible representations of $\pi_1(M)$, certifying that the tautological Azumaya algebra extends over a point $p$ comes directly from Conditions $(\star_1)$, $(\star_2)$, whereas certifying that the tautological Azumaya algebra does not extend over a point $p$ is more delicate and requires some local understanding of the geometry of $\frak{C}$ at $p$.
For this we will require work of Heusener and Porti which shows that characters of non-abelian reducible representations of $\pi_1(M)$ are reduced points of $\frak{C}$.

After proving Theorems \ref{thm:main1} and \ref{thm:main2}, in Section \ref{sec:examples} we will give an infinite family of examples for which they apply.
Specifically, we will show that the once punctured torus bundles arising as $j$-fold cyclic covers of the figure eight knot complement satisfy both $(\star_1)$, $(\star_2)$ for any canonical component if and only if $j$ is odd.

\begin{restatable}{thm}{mainexamples}\label{thm:main1examples}
Suppose that $\{N_j\}$ is the family of hyperbolic once punctured torus bundles arising as $j$-fold cyclic covers of the figure eight knot complement.
Then if $j$ is odd, the hypothesis of Theorem \ref{thm:main1} and Corollary \ref{cor:main1} hold for any canonical component $\frak{C}_j$ of $\frak{X}(N_j)_{\bbQ}$.
If $j$ is even, then the hypothesis of Theorem \ref{thm:main2} hold for any canonical component $\frak{C}_j$ of $\frak{X}(N_j)_{\bbQ}$.
In particular, the tautological Azumaya algebra extends if and only if $j$ is odd and in that case $S_{N_j}=\{2\}$.
\end{restatable}

In the case of hyperbolic knot complements, after showing finiteness of the set $S_M$ from Theorem \ref{thm:main1}, Chinburg--Reid--Stover go on to study which primes appear in this set \cite[Thm 1.5]{CRS}.
Specifically, they are able to show that the primes contained in $S_M$ come from studying the similar condition to $(\star)$ over not just number field but also over extensions of the finite field $\bbF_\ell$ for rational primes $\ell$.
Exhibiting such a result uses machinery from integral models and we adapt their machinery to our setting to exhibit the following theorem (see Section \ref{sec:intmod} for more details).

\begin{restatable}{thm}{intmod}\label{thm:intmod}
Suppose that $\frak{C}$ is a curve component of $\frak{X}(M)_k$ which contains the character of an irreducible representation and for which the tautological Azumaya algebra $\mathcal{A}_{k(\widetilde{C})}$ extends over $\widetilde{C}$.
Let $S$ denote the set of odd rational primes which divide the least common multiple of the orders of $x$, $y$ in $H_1(M,\bbZ)_{tors}$ union $\{2\}$.
Then the tautological Azumaya algebra extends over $\widetilde{\frak{C}}_S$, where $\widetilde{\frak{C}}_S$ is a relatively minimal regular model of $\widetilde{C}$ over $\mathcal{O}_{k,S}$.
\end{restatable}

\noindent In particular, this theorem shows that the finite set $S_M$ from Corollary \ref{cor:main1} is contained in $S$ and so Theorem \ref{thm:intmod} gives an effective way to compute a candidate set of primes for the set $S_M$.

Finally we end the paper by discussing several examples in Section \ref{sec:examples} and some remaining open question in Section \ref{sec:questions}.\\

\noindent \textbf{Acknowledgements:} A special debt is owed to Matthew Stover for not only spending the time to explain his work with Chinburg and Reid but also for spending time helping the author develop several ideas at the genesis of this project.
The author was partially supported by NSF Grant DMS-2005438/2300370.

%~~~~~~~~~~~~~~~~~~~~~~~~~~~~~~~~~~~~~~~~~~~~~~~~~~~~~~~~~~~
%~~~~~~~~~~~~~~~~~~~~~~~~~~~~~~~~~~~~~~~~~~~~~~~~~~~~~~~~~~~
%~~~~~~~~~~~~~~~~~~~~~~~~~~~~~~~~~~~~~~~~~~~~~~~~~~~~~~~~~~~

\section{Background}\label{sec:background}

\subsection{Once punctured torus bundles}\label{sec:1ptb}
Throughout we use $\mathring{T}$ to denote a once punctured torus with fundamental group $\pi_1(\mathring{T})=F_2$, the free group on two letters, and we always use the symbols $x$, $y$ to denote generators of $F_2$.

The main focus of this paper will be oriented once punctured torus bundles, which are bundles fibering over the circle $S^1$ with fiber $\mathring{T}$.
Such bundles can be written as $M_\phi=\mathring{T}\times[0,1]/(t,0)\sim(\phi(t),1)$, where $\phi\in\Aut(F_2)$ is the monodromy of the fiber.
The fundamental group of $M_\phi$ is therefore given by 
\begin{equation}\label{1ptbpresent}
\pi_1(M) = \langle t,x,y\ |\ txt^{-1} = \phi(x),\ tyt^{-1} = \phi(y)\rangle.
\end{equation}
By a theorem of Murasugi, the homeomorphism type of $M_\phi$ is completely determined by the conjugacy class of $\phi$ in the mapping class group, $\Mod^+(\mathring{T})=\SL_2(\bbZ)$, and therefore we may write $\phi$ as
\[
\phi = i^\ep R^{n_1} L^{n_2} \cdots L^{n_{2k}},
\]
where $\ep \in \{0, 1\}$, $n_j \neq 0$ for all $j$ except possibly $n_1$ or $n_{2k}$, and
\begin{align*}
i(x) &= x^{-1}, & i(y) &= y^{-1}, \\
R(x) &= x, & R(y) &= yx, \\
L(x) &= xy, & L(y) &= y.
\end{align*}
The image of these maps in $\Mod^+(\mathring{T})$ have the following form, which we abusively denote using the same letter
\begin{align*}
i &= \begin{pmatrix} -1 & 0 \\ 0 & -1 \end{pmatrix}, \\
R &= \begin{pmatrix} 1 & 1 \\ 0 & 1 \end{pmatrix}, \\
L &= \begin{pmatrix} 1 & 0 \\ 1 & 1 \end{pmatrix}.
\end{align*}
Therefore the associated representative of $\phi$ in $\Mod^+(\mathring{T})$ is given by the matrix
\begin{equation}\label{eqn:monmatrix}
\Phi = [\phi]= \begin{pmatrix} a & b \\ c & d \end{pmatrix} \in \SL_2(\bbZ).
\end{equation}
The once punctured torus bundle $M_\phi$ will carry a hyperbolic metric if and only if $\Phi$ is a hyperbolic matrix in $\SL_2(\bbZ)$, which is precisely when $|\Tr(\Phi)|>2$.
We will always assume throughout the paper that $M_\phi$ is a hyperbolic once punctured torus bundle and hence we always assume the condition on the trace.

As $M_\phi$ has one cusp, the boundary $\partial M_\phi$ is a torus and hence $\pi_1(\partial M_\phi)\cong\bbZ^2$.
The fundamental group $\pi_1(\partial M_\phi)$ is generated by the \emph{stable letter} $t$, which is also frequently called the \emph{meridian}, and the \emph{longitude} $l=[x,y]$.
Useful for us in Section \ref{sec:nonabclassification} will be that these two elements commute in $\pi_1(M_\phi)$.
From Equation \eqref{1ptbpresent}, it is straightforward to check that
\begin{equation}\label{eqn:homology}
H_1(M_\phi,\bbZ)\cong\pi_1(M)^{ab} \cong \bbZ \oplus F,
\end{equation}
where the $\bbZ$ factor is generated by (the class of) $t$ and
\[
F = \langle X, Y\ |\ X^{(-1)^\ep-a} = Y^c,\ Y^{(-1)^\ep-d} = X^b,\ [X, Y] = 1 \rangle,
\]
is a (possibly trivial) finite abelian group of order $|F|=|\Tr(\Phi)-2|$.
From this point on, we drop the subscript $\phi$ and always assume that $M$ is a hyperbolic once punctured torus bundle.

%~~~~~~~~~~~~~~~~~~~~~~~~~~~~~~~~~~~~~~~~~~~~~~~~~~~~~~~~~~~
%~~~~~~~~~~~~~~~~~~~~~~~~~~~~~~~~~~~~~~~~~~~~~~~~~~~~~~~~~~~
%~~~~~~~~~~~~~~~~~~~~~~~~~~~~~~~~~~~~~~~~~~~~~~~~~~~~~~~~~~~

\subsection{Some notions from algebraic geometry}\label{sec:alggeoprelim}

Throughout this paper, we will use the language of schemes and varieties over number fields for which we will need some preliminaries, all of which can be found in \cite{Hartshorne,Liu}. 
We will always use $k$ to denote a fixed number field and, throughout this paper we consider schemes $X$ of finite type over $k$, so that $X$ admits a finite cover by affine open sets $U$ of the form $U=\Spec(k[x_1,\dots,x_n]/\mathcal{I})$.
Given $p\in X$, we denote by $\mathcal{O}_{X,p}$ the local ring at $p$, by $\frak{m}_p$ the maximal ideal of $\mathcal{O}_{X,p}$, and by $k(p)=\mathcal{O}_{X,p}/\frak{m}_p$ the residue field.
We call a point $p\in X$ a \emph{closed point} if $\overline{\{p\}}=\{p\}$ and we call $p\in X$ a \emph{codimension $1$} point if $\overline{\{p\}}$ is codimension $1$ in $X$.
We use the notation $X^{(1)}$ to denote the set of codimension $1$ points.
We call $X$ irreducible over $k$ if $X$ cannot be written as the union of two proper Zariski closed subschemes.
When $X$ is irreducible, it has a unique generic point $\xi$ such that $\overline{\{\xi\}}=X$ and we define the \emph{function field of $X$} as $k(X)=\mathcal{O}_{X,\xi}$.

As $k$ is a number field, if we moreover assume that $k\subset \bbC$ under a fixed embedding then there is a base change of $X$ to $\bbC$ which we always denote by $X_\bbC=X\times_{\Spec(k)}\Spec(\bbC)$.
We call $X$ \emph{geometrically irreducible} if $X_\bbC$ is irreducible over $\bbC$ and we call $X$ \emph{geometrically integral} if $X_\bbC$ is moreover integral.

Throughout the paper we will be analyzing the specific case when $X$ is a curve.
In this setting, if $p\in X$ is a closed point such that $\mathcal{O}_{X,p}$ is a regular local ring, then $\mathcal{O}_{X,p}$ is in fact a discrete valuation ring.
When this is the case, we call any choice of generator $\pi$ of the maximal ideal $\frak{m}_p$ a \emph{uniformizing parameter}.
Moreover, $\pi$ induces a discrete valuation on the fraction field $\mathrm{Frac}(\mathcal{O}_{X,p})$, and in the sequel we will use the notation $\ord_p(-)$ to denote this valuation.

%~~~~~~~~~~~~~~~~~~~~~~~~~~~~~~~~~~~~~~~~~~~~~~~~~~~~~~~~~~~
%~~~~~~~~~~~~~~~~~~~~~~~~~~~~~~~~~~~~~~~~~~~~~~~~~~~~~~~~~~~
%~~~~~~~~~~~~~~~~~~~~~~~~~~~~~~~~~~~~~~~~~~~~~~~~~~~~~~~~~~~

\subsection{Spaces of representations, fields of definition, and tangent spaces}\label{sec:backrep}

This paper will be focused on studying the properties of specific schemes $\frak{R}(M)$ and $\frak{X}(M)$ over $\bbQ$ and we refer the reader to \cite{LubotzkyMagid} or \cite{Sikora} for the information below.
Define the \emph{representation scheme} and the \emph{character scheme}, respectively, by the formulas
\begin{equation}\label{eqn:schemetime}
\frak{R}(M)=\frak{Hom}(\pi_1(M),\SL_2), \quad \frak{X}(M)=\frak{Hom}(\pi_1(M),\SL_2)/\!/\SL_2,
\end{equation}
where the double bars in the latter denote the GIT quotient of $\frak{R}(M)$ by the action of conjugation by $\SL_2$.
These are affine schemes defined over $\bbQ$ with the property that, if $\bbQ[\frak{R}(M)]$ denotes the coordinate ring for $\frak{R}(M)$, then the coordinate ring of $\frak{X}(M)$ is given by the $\SL_2$-invariant functions $\bbQ[\frak{X}(M)]=\bbQ[\frak{R}(M)]^{\SL_2}$.

The ring $\bbQ[\frak{R}(M)]$ is, in general, non-reduced and we define the ring $\bbQ[R(M)]$ to be the quotient by its nilradical
$$\bbQ[R(M)]=\bbQ[\frak{R}(M)]/\frak{N}_{\bbQ[\frak{R}(M)]}=\bbQ[\Hom(\pi_1(M),\SL_2)].$$
The surjective map $\bbQ[\frak{R}(M)]\to \bbQ[R(M)]$similarly induces a surjective map
$$\bbQ[\frak{X}(M)]=\bbQ[\frak{R}(M)]^{\SL_2}\to \bbQ[R(M)]^{\SL_2},$$
whose kernel is precisely the nilradical of $\bbQ[\frak{X}(M)]$.
In particular, taking the spectrum of these coordinate rings gives the classical \emph{representation variety} and \emph{character variety} over $\bbQ$, defined by 
\begin{equation}\label{eqn:repvar}
R(M)=\Spec(\bbQ[R(M)])=\Hom(\pi_1(M),\SL_2),
\end{equation}
\begin{equation}\label{eqn:charvar}
X(M)=\Spec(\bbQ[R(M)]^{\SL_2})=\Hom(\pi_1(M),\SL_2)/\!/\SL_2.
\end{equation}
Importantly, the surjections on the level of coordinate rings induce closed embeddings $R(M)\hookrightarrow \frak{R}(M)$, $X(M)\hookrightarrow \frak{X}(M)$ and the images under these maps are precisely the set of closed points.
In the sequel we use the notation, $\Upsilon:\frak{R}(M)\to \frak{X}(M)$ to denote the associated projection and, given a representation $\rho\in R(M)$, we associate to $\rho$ its image under $R(M)\hookrightarrow \frak{R}(M)$ and write $\chi_\rho=\Upsilon(\rho)$.
Note that $\bbQ[\frak{X}(M)]$ is generated by the trace functions $I_\gamma$.
That is, given a word $\gamma\in\pi_1(M)$, $\bbQ[\frak{X}(M)]$ is generated by functions of the form $I_\gamma(\chi_\rho)=\Tr(\rho(\gamma))$.

In the specific case that $M$ is a once punctured torus bundle, $\mathfrak{X}(M)$ can be explicitly embedded in $\mathbb{A}^7_\bbQ$ via the map which associates to $I_\gamma$ its values on the subset $\{x,y,t,xy,xt,yt,xyt\}$ of $\pi_1(M)$ \cite{CS1}.
Therefore, throughout we always assume that $\frak{X}(M)$ is realized as an affine subscheme of $\mathbb{A}^7_\mathbb{Q}$ via this embedding.
In general, $\frak{X}(M)$ is not geometrically irreducible and moreover the irreducible components of $\frak{X}(M)$ are at most $1$-dimensional since $M$ has $1$ cusp.
Therefore assume that $\frak{C}_\bbC$ is an irreducible curve component of $\frak{X}(M)_\bbC\subset\mathbb{A}^7_\bbC$.
We point out that the notation $\frak{X}(M)_\bbC$ is unambiguous in this case because there is only one embedding of $\bbQ$ in $\bbC$.

From the data of $\frak{C}_\bbC$ and the affine embedding in $\mathbb{A}^7_\bbC$, Long--Reid \cite{LongReid2} extract a field $k$ called the \emph{field of definition of $\frak{C}_\bbC$} which is the unique minimal subfield $k$ such that the ideal $\mathcal{I}$ defining the embedding $\frak{C}_\bbC$ in $\mathbb{A}^7_\bbC$ is defined over $k$.
In particular, choosing generators for the ideal $\mathcal{I}$ with coefficients in $k$ determines a geometrically irreducible curve $\frak{C}$ for which $\frak{C}_\bbC=\frak{C}\times_k\bbC$.
Note importantly that, in our setup, the field of definition will always be a number field.
Moreover, note that $k$ depends explicitly on the affine embedding of $\frak{X}(M)_\bbC$ and need not be the minimal field of definition in the sense of algebraic geometry (e.g. \cite[IV$_2$ \S 4.8]{EGA}).
Indeed, there are examples where these two notions differ (see Section \ref{sec:dunfield} and Remark \ref{rem:dunfield} for instance).

Recall that the Zariski tangent space $T_p X$ to a scheme $X$ at a point $p$ is defined as $(\frak{m}_p/\frak{m}_p^2)^*$, which is the dual of $\frak{m}_p/\frak{m}_p^2$ considered as a $k(p)$-vector space.
In the future we will need to identify the Zariski tangent space to $\frak{R}(M)_\bbC$ at certain representations $\rho$ with the space of $1$-cocycles $Z^1(\pi(M),\frak{sl}_2(\bbC))$, which we now recall how to do.
Given any representation $\rho:\pi_1(M)\to \SL_2(\bbC)$, then $\rho$ defines the structure of a $\pi_1(M)$-module on the Lie algebra $\frak{sl}_2(\bbC)$ via $\gamma\cdot X=\mathrm{Ad}_{\rho(\gamma)}X$.
We use the notation $\frak{sl}_2(\bbC)^\rho$ to record this structure on $\frak{sl}_2(\bbC)$.
Recall that a $1$-cocycle $d$ is a map $d:\pi_1(M)\to \frak{sl}_2(\bbC)$ such that 
$$d(\gamma\delta)=d(\gamma)+\gamma\cdot d(\delta),$$
for all $\gamma,\delta\in\pi_1(M)$.
Then an observation of Weil \cite{Weil} (see also \cite[Thm 35]{Sikora}) shows that for a closed point $\rho\in \frak{R}(M)_\bbC$, we have an identification $T_\rho \frak{R}(M)_\bbC\cong Z^1(\pi(M),\frak{sl}_2(\bbC)^\rho)$.

Recall that the embedding $R(M)\hookrightarrow \frak{R}(M)$ has image the closed points of $\frak{R}(M)$ and the same is true of the basechanged morphism $R(M)_\bbC\hookrightarrow \frak{R}(M)_\bbC$.
Therefore, identifying $R(M)_\bbC$ with its image, it makes sense to also consider the Zariski tangent space $T_\rho R(M)_\bbC$.
From this and the observation of Weil, the following chain of inequalities holds
\begin{equation}\label{eqn:tangentineq}
\dim_\rho R(M)_\bbC\le\dim_\bbC T_\rho R(M)_\bbC\le\dim_\bbC T_\rho \frak{R}(M)_\bbC=\dim_\bbC Z^1(\pi_1(M),\frak{sl}_2(\bbC)^\rho),
\end{equation}
where $\dim_\rho R(M)_\bbC$ is the local dimension at $\rho$, i.e., the Krull dimension of the local ring $\mathcal{O}_{R(M)_\bbC,\rho}$ (see \cite[Pg 31]{LubotzkyMagid}).
When the second inequality is an equality, it will follow that the local ring $\mathcal{O}_{\frak{R}(M)_{\bbC},\rho}$ at $\rho$ is a reduced ring.
In this situation we call $\rho$ a \emph{reduced point of $\frak{R}(M)$}.
We point out that when both inequalities are equality, then $\rho$ is a regular point and in particular is a smooth point of $R(M)_\bbC$ which is contained in a unique component of dimension equal to $\dim_\bbC Z^1(\pi_1(M),\frak{sl}_2(\bbC)^\rho)$.

%~~~~~~~~~~~~~~~~~~~~~~~~~~~~~~~~~~~~~~~~~~~~~~~~~~~~~~~~~~~
%~~~~~~~~~~~~~~~~~~~~~~~~~~~~~~~~~~~~~~~~~~~~~~~~~~~~~~~~~~~
%~~~~~~~~~~~~~~~~~~~~~~~~~~~~~~~~~~~~~~~~~~~~~~~~~~~~~~~~~~~

\subsection{Curve components in character schemes and the canonical component}\label{sec:curveback}

In general, the character scheme $\frak{X}(M)$ will not be geometrically irreducible so, as in Section \ref{sec:backrep}, we decompose it into irreducible components $\frak{X}(M)_\bbC=\cup_{i=1}^n (V_i)_\bbC$.
In the sequel we will only care about components $(V_i)_\bbC$ of dimension $1$, which we call \emph{curve components}, and moreover we will only care about components $(V_i)_\bbC$ which contain a character $\chi_\rho$ such that $\rho$ is an irreducible representation.

As in Section \ref{sec:backrep}, each such $(V_i)_\bbC$ has field of definition some number field $k_i$ so we fix once and for all the standing notation that $\frak{C}$ denotes a fixed choice of geometrically irreducible curve component of $\frak{X}(M)_k$, where $k$ is the field of definition of $\frak{C}$, which contains the character of an irreducible representation.
Note that $k$ will in general depend on $\frak{C}$, however for the reader's convenience we do not build this dependence into the notation.
As usual, we also write $k[\frak{C}]$ for the corresponding coordinate ring.

The component $\frak{C}$ can, in general, have non-reduced points and therefore we use the notation $C^{red}=\Spec(k[\frak{C}]^{red})$, where $k[\frak{\frak{C}}]^{red}=k[\frak{C}]/\frak{N}_{k[\frak{C}]}$ is the quotient of $k[\frak{C}]$ by its nilradical.
As in Section \ref{sec:backrep}, we have an induced map $C^{red}\to \frak{C}$.
The scheme $C^{red}$ need not be non-singular, however there is always a unique smooth projective model $\widetilde{C}$ of $C^{red}$ up to isomorphism.
Specifically, we define $\widetilde{C}$ as the smooth projective completion of the normalization of $C^{red}$.
This smooth projective curve comes equipped with a birational map $\widetilde{C}\dashrightarrow C^{red}$ which is an isomorphism on regular points of $C^{red}$.
We call the points $\mathcal{I}(\widetilde{C})\subset\widetilde{C}$ for which this map is not defined the \emph{ideal points}.
Moreover composition of this map with the map to $\frak{C}$ yields a birational map $\Pi_k:\widetilde{C}\dashrightarrow \frak{C}$.
Note that the points where $\Pi_k$ is undefined are also precisely the ideal points $\mathcal{I}(\widetilde{C})$.

The map $\Pi_k$ takes the generic point of $\widetilde{C}$ to that of $\frak{C}$ and hence the birationality induces an isomorphism of function fields $\Pi_k^*:k(\frak{C})\to k(\widetilde{C})$ defined by $\Pi_k^*(f)= f\circ\Pi_k$.
In the remainder of the paper, we will identify $k(\frak{C})$ with $k(\widetilde{C})$ using $\Pi_k^*$.
Importantly, for a fixed trace function $I_\gamma\in k(\frak{C})$ from Section \ref{sec:backrep}, we write $\widetilde{I}_\gamma=\Pi_k^*(I_\gamma)$ to be the corresponding function in $k(\widetilde{C})$.

Continue fixing a geometrically irreducible curve component $\frak{C}\subset X(M)_k$ and let $\Upsilon_k$ be the natural map $\Upsilon_k:\frak{R}(M)_k\to \frak{X}(M)_k$.
Then we need the following lemma of Chinburg, Reid, and Stover.

\begin{lem}[\cite{CRS}, Lemma 2.3]
There exists an irreducible curve $\frak{D}\subset \frak{R}(M)_k$ such that $\Upsilon_k(\frak{D})\subset \frak{C}$, the extension of function fields $k(\frak{D})/k(\frak{C})$ is a finite extension, and such that there exists a representation $\rho_{\frak{C}}:\pi_1(M)\to\SL_2(k(\frak{D}))$ with the property that $\chi_{\rho_\frak{C}}(\gamma)(\rho)=\chi_\rho(\gamma)=\Tr(\rho(\gamma))$.
\end{lem}

The representation $\rho_{\frak{C}}:\pi_1(M)\to\SL_2(k(\frak{D}))$ from the above lemma is called the \emph{tautological representation} and is defined by the formula
$$\rho_{\frak{C}}(\gamma)=\begin{pmatrix}
a_\gamma&b_\gamma\\
c_\gamma&d_\gamma\end{pmatrix},$$
where $a_\gamma,\dots,d_\gamma$ are the coordinate functions such that $a_\gamma(\rho),\dots,d_\gamma(\rho)$ give the value of the corresponding entry of the matrix $\rho(\gamma)$.

Chinburg--Reid--Stover are also able to describe the relationship between the trace field of a representation $\rho$ and the residue field of $\frak{C}$ at the character associated to $\chi_\rho$.
This will be useful for us throughout the paper, especially in the case that $\chi_\rho$ corresponds to a regular point.

\begin{lem}[\cite{CRS}, Lemma 2.5]\label{lem:residuefield}
Let $p\in \widetilde{C}\setminus \mathcal{I}(\widetilde{C})$, $\chi_\rho=\Pi_k(p)$, and let $k_\rho$ be the trace field of $\rho$. 
Then $k_\rho\subset k(\chi_\rho)\subset k(p)$ where the latter inclusion is equality when $p$ is a regular point of $\widetilde{C}$.
Moreover when $p$ is regular, $k(p)$ is generated by $k$ and $k_\rho$.
\end{lem}

%~~~~~~~~~~~~~~~~~~~~~~~~~~~~~~~~~~~~~~~~~~~~~~~~~~~~~~~~~~~
%~~~~~~~~~~~~~~~~~~~~~~~~~~~~~~~~~~~~~~~~~~~~~~~~~~~~~~~~~~~
%~~~~~~~~~~~~~~~~~~~~~~~~~~~~~~~~~~~~~~~~~~~~~~~~~~~~~~~~~~~

\subsection{Azumaya algebras and tame symbols}\label{sec:azalg}

Given a ring $R$, an Azumaya algebra $A$ over $R$ is a finitely generated, faithful, projective $R$-algebra such that $A\otimes_R A^{opp}\cong\End_R(A)$, where $A^{opp}$ is the opposite algebra of $A$.
In \cite{Grothendieck}, Grothendieck defined an \emph{Azumaya algebra} $\mathcal{A}_X$ over a scheme $X$ as a sheaf of $\mathcal{O}_X$-algebras such that for each $p\in X$, the stalk $\mathcal{A}_p$ is an Azumaya algebra over the local ring $\mathcal{O}_{X,p}$.
Azumaya algebras over a fixed $X$ modulo a suitable equivalence relation define a group $\mathrm{Br}(X)$ called the Brauer group.
Specifically, we say that two Azumaya algebras $\mathcal{A}_X$, $\mathcal{A}'_X$ are equivalent if there exist locally free, finite rank sheaves of $\mathcal{O}_X$-modules $\mathcal{E}$, $\mathcal{E}'$ such that $\mathcal{A}_X\otimes_{\mathcal{O}_X}\End(\mathcal{E})\cong \mathcal{A}'_X\otimes_{\mathcal{O}_X}\End(\mathcal{E}')$.
The \emph{Brauer group} $\mathrm{Br}(X)$ is then the group of equivalence classes of Azumaya algebras over $X$ with group structure given by tensor product over $\mathcal{O}_X$.
Note that when $X=\mathrm{Spec}(L)$ where $L$ is a field, $\mathrm{Br}(X)=\mathrm{Br}(L)$ is precisely the classical Brauer group.

Given an Azumaya algebra $\mathcal{A}_X$ over a regular integral noetherian scheme $X$ with generic point $\xi$, we denote by $\mathcal{A}_\xi$ the stalk over $\xi$ which by definition is an Azumaya algebra over the function field $k(X)$.
Therefore there is a natural map 
\begin{align}
\Psi:\mathrm{Br}(X)&\to \mathrm{Br}(k(X)),\label{eqn:azumayamap}\\
[\mathcal{A}_X]&\mapsto[\mathcal{A}_\xi]\nonumber
\end{align}
which is injective.
By results of Grothendieck \cite[Prop 2.1]{GroBrauer3} and Gabber \cite{Gabber}, one can understand the image of $\Psi$.
\begin{thm}[\cite{Poonen}, Theorem 6.8.3]
With $X$ as above, the following sequence is exact
$$\xymatrix{
0\ar[r]&\mathrm{Br}(X)\ar[r]&\mathrm{Br}(k(X))\ar[r]^-{\oplus r_p}&\bigoplus_{p\in X^{(1)}}H^1(k(p),\mathbb{Q}/\bbZ)\ar[r]&},$$
where $r_p$ is a certain map $\Br(k(X))\to H^1(k(p),\bbQ/\bbZ)$.
The above comes with the caveats that one must exclude the $p$-primary part of all of the groups in the case that either 1) $X$ has dimension less than or equal to $1$ and some $k(p)$ is imperfect with characteristic $p$ or 2) the dimension of $X$ is at least 2 and some $k(p)$ is of characteristic $p$.
\end{thm}

In the sequel, we will only study the classes in $\mathrm{Br}(k(X))$ which are represented by a quaternion algebra.
Restricting to these classes of algebras, one sees that $r_p$ agrees with the map
\begin{align}
r_p:\mathrm{Br}(k(X))&\to k(p)^*/k(p)^{*2}\label{eqn:tamesymbol},\\
\{\alpha,\beta\}&\mapsto \{\alpha,\beta\}_p=\left[(-1)^{\ord_p(\alpha)\ord_p(\beta)}\frac{\alpha^{\ord_p(\beta)}}{\beta^{\ord_p(\alpha)}}\right]\nonumber
\end{align}
where $\{\alpha,\beta\}$ is any choice of Hilbert symbol.
The image of $r_p$ is independent of this choice.
The map $r_p$ is called the \emph{tame symbol} and we say that the tame symbol vanishes at a point $p$ if $r_p(\mathcal{A})= [1]$.
Triviality of the tame symbol will be the key tool in analyzing when a class of Azumaya algebra over $k(X)$ lies in the image of the map $\Psi$ from Equation \eqref{eqn:azumayamap}.

Specifying to our setting, we assume that $\frak{C}$ is a fixed geometrically irreducible curve component of $\frak{X}(M)_k$, where $k$ is the field of definition of $\frak{C}$.
Suppose that $\frak{C}$ contains the character of an irreducible representation and let $\rho_{\frak{C}}$ be the associated tautological representation, defined over some finite extension $k(\frak{D})$ of the function field $k(\frak{C})\cong k(\widetilde{C})$.
Then, as in \cite[\S 2.5]{CRS}, define a $k(\widetilde{C})$-quaternion subalgebra of $\mathrm{Mat}_2(k(\frak{D}))$ by
$$\mathcal{A}_{k(\widetilde{C})}=\left\{\sum_i a_i\rho_{\frak{C}}(\gamma_i)\mid a_i\in k(\widetilde{C}),~~\gamma_i\in\pi_1(M)\right\},$$
where the sums above are assumed to be finite.
We call this algebra the \emph{tautological Azumaya algebra associated to $\frak{C}$}.
This algebra gives a well defined class in the Brauer group $\mathrm{Br}(k(\widetilde{C}))$.

Whether or not the tautological Azumaya algebra lies in the image of $\Psi$ from Equation \eqref{eqn:azumayamap} is intimately related to when the set $S$ defined in the introduction exists.
Indeed, Chinburg, Reid, and Stover show the following where we recall the notation from Equation \eqref{eqn:quatalg}:
\begin{thm}[\cite{CRS}, Theorem 1.1(3)--(4)]
Let $\Psi:\mathrm{Br}(\widetilde{C})\to\mathrm{Br}(k(\widetilde{C}))$ be as in Equation \eqref{eqn:azumayamap} and let $C^\#$ denote the normalization of $C^{red}$.
Then the following hold:
\begin{enumerate}
\item Suppose that $\mathcal{A}_{k(\widetilde{C})}\notin\im(\Psi)$, then there is no finite set of rational primes $S$ such that the following holds: For all but finitely many reduced points $\chi_\rho\in \frak{C}(\overline\bbQ)$ with $\rho$ absolutely irreducible, the $k(\chi_\rho)$-quaternion algebra $\mathcal{A}_\rho\otimes_{k_\rho}k(\chi_\rho)$ is unramified outside of finite places of $k(\chi_\rho)$ lying over primes in $S$.
\item Suppose that $\mathcal{A}_{k(\widetilde{C})}\in\im(\Psi)$, then there exists a finite set of rational primes $S$ with the following property: given any $\chi_\rho\in C^\#(\overline\bbQ)$ such that $\rho$ is absolutely irreducible, any finite ramified prime $\frak{p}$ of the $k(p)$-quaternion algebra $\mathcal{A}_\rho\otimes_{k_\rho}k(p)$ lies over a rational prime in $S$.
\end{enumerate}
\end{thm}
\noindent In particular, this shows that $S$ exists and is finite if and only if $\mathcal{A}_{k(\widetilde{C})}\in\im(\Psi)$.
Note moreover that when $\mathcal{A}_{\widetilde{C}}$ is such that $\Psi(\mathcal{A}_{\widetilde{C}})=\mathcal{A}_{k(\widetilde{C})}$, then the injectivity of $\Psi$ shows that its class is completely determined by $\mathcal{A}_{k(\widetilde{C})}$.
When the algebra $\mathcal{A}_{\widetilde{C}}$ exists, we say that \emph{the tautological Azumaya algebra $\mathcal{A}_{k(\widetilde{C})}$ extends over $\widetilde{C}$}.

To get a tangible grasp on the extension problem, we will want to work with explicit Hilbert symbols for $\mathcal{A}_{k(\widetilde{C})}$ in the sequel and for this we will frequently make use of technology built by Chinburg--Reid--Stover.
In particular, the following two lemmas will be important for us, the first of which shows that one has some control over the functions one chooses in a Hilbert symbol for $\mathcal{A}_{k(\widetilde{C})}$ provided that $\frak{C}$ satisfies some mild conditions.
In the following we continue the previous assumption on $\frak{C}$.

\begin{lem}[\cite{CRS}, Lemma 2.8]\label{lem:choosehilbert}
Assume that $g\in\pi_1(M)$ is such that $\widetilde{I}_g^2-4$ is not identically $0$ on $\widetilde{C}$.
Then there exists some $h\in\pi_1(M)$ and $\chi_\rho=\Upsilon(\rho)\in \frak{C}$ such that $\widetilde{I}_{[g,h]}-2$ is not identically zero on $\widetilde{C}$ and $\rho\vert_{\langle g,h\rangle}$ is irreducible.
Therefore 
\begin{equation}\label{eqn:azhilbsymb}
\mathcal{A}_{k(\widetilde{C})}=\left(\frac{\widetilde{I}_g^2-4,\widetilde{I}_{[g,h]}-2}{k(\widetilde{C})}\right),
\end{equation}
gives a Hilbert symbol for $\mathcal{A}_{k(\widetilde{C})}$.
\end{lem}

The second of these two lemmas will allow us to not only specify our choice of $g$ but also that of $h$.

\begin{lem}[\cite{CRS}, Corollary 2.9]\label{lem:CRShilb}
Fix $g,h\in\pi_1(M)$ such that $\widetilde{I}^2_g-4$ is not identically zero on $\widetilde{C}$ and such that there exists some $\chi_\rho=\Upsilon(\rho)\in \frak{C}$ for which $\rho\vert_{\langle g,h\rangle}$ is an irreducible representation.
Then $\mathcal{A}_{k(\widetilde{C})}$ can be described by the Hilbert symbol in Equation \eqref{eqn:azhilbsymb} with this choice of $g,h$.
\end{lem}

We briefly record a simple observation that will be useful in the sequel when we choose the functions in Equation \eqref{eqn:azhilbsymb}, especially for a canonical component of $\frak{X}(M)_k$.

\begin{cor}\label{cor:canoncomp}
If $\rho_0$ is any lift of the discrete faithful representation of $\pi_1(M)$ to $\SL_2(\bbC)$, then $\rho_0\vert_{\langle x,y\rangle}$, $\rho_0\vert_{\langle x,t\rangle}$, and $\rho_0\vert_{\langle y,t\rangle}$ are all irreducible representations.
\end{cor}
\begin{proof}
That $\rho_0\vert_{\langle x,y\rangle}$ is irreducible is clear since $F_2=\langle x,y\rangle$ is not solvable.
We therefore must show that $\rho_0\vert_{\langle x,t\rangle}, \rho_0\vert_{\langle y,t\rangle}$ is irreducible.

Note that the longitude $l=[x,y]$ has the property that $\chi_{\rho_0}=\Tr(\rho_0(l))=-2$.
Using J\o rgensen's normalization \cite[\S 2]{Jorgensen} (see also \cite[Lem 6]{Bowditch}), up to conjugation we may assume that
\begin{align*}
\rho_0(x)&=\begin{pmatrix}
z_x-z_yz_{xy}^{-1}&z_xz_{xy}^{-2}\\
z_x&z_yz_{xy}^{-1}
\end{pmatrix},\\
\rho_0(y)&=\begin{pmatrix}
z_y-z_xz_{xy}^{-1}&-z_yz_{xy}^{-2}\\
-z_y&z_xz_{xy}^{-1}
\end{pmatrix},\\
\rho_0(l)&=\begin{pmatrix}
-1&-2\\
0&-1
\end{pmatrix},
\end{align*}
where $z_g=\chi_{\rho_0}(g)=\Tr(\rho_0(g))$.
Note that the condition on the trace of the longitude implies that $z_x,z_y,z_{xy}$ satisfy the relationship $z_{x}^2+z_y^2+z_{xy}^2=z_xz_yz_{xy}$ so this gives a genuine $\SL_2(\bbC)$-representation of $\langle x,y\rangle$.
Moreover since $\rho_0$ is the discrete faithful representation, it follows that $z_x,z_y,z_{xy}\neq 0$ since $x,y,xy$ each generate infinite cyclic subgroups.

The stable letter $t$ must commute with $l$ and from this it is straightforward to verify that
\begin{equation}\label{eqn:discretestable}
\rho_0(t)=\begin{pmatrix}
\pm 1&z_0\\
0&\pm 1\end{pmatrix},
\end{equation}
for some $z_0\in\bbC^*$.
The conclusion that $z_0\neq 0$ again follows from the fact that $t$ generates an infinite cyclic subgroup.
As $z_x,z_y\neq 0$, it is then clear from Equation \eqref{eqn:discretestable} that $\rho\vert_{\langle x,t\rangle}$, $\rho\vert_{\langle y,t\rangle}$ are irreducible representations.
\end{proof}

In \cite{CRS}, Chinburg--Reid--Stover also show that to understand when the tautological Azumaya algebra extends over $\widetilde{C}$ it is enough to study this problem at characters of non-abelian reducible representations.

\begin{prop}[\cite{CRS}, Proposition 4.1]\label{prop:idealirrepextend}
The tautological Azumaya algebra extends over characters of irreducible representations and ideal points.
Precisely, let $p\in \widetilde{C}$, such that either $p\in\mathcal{I}(\widetilde{C})$ or $\Pi_k(p)=\chi_\rho$ for an irreducible representation $\rho$, then $r_p(\mathcal{A}_{k(\widetilde{C})})=[1]$, where $r_p$ is as in Equation \eqref{eqn:tamesymbol}.
\end{prop}

\noindent To be precise, the proof of \cite[Prop 4.1]{CRS} is stated in the case of hyperbolic knot complement in $S^3$, however its proof goes through identically in more generality since the only property they use of $\frak{C}$ is that it contains the character of an irreducible representation.
Note however that Proposition \ref{prop:idealirrepextend} gives no control over the explicit form of the functions in the Hilbert symbol in Equation \eqref{eqn:azhilbsymb} and so one really wants to know both the information from Lemma \ref{lem:CRShilb} and Proposition \ref{prop:idealirrepextend}.
To the author's understanding, this is why \cite[Prop 4.1]{CRS} is stated for hyperbolic knot complements.

At all points for which the tame symbol does not automatically vanish, we have the following.

\begin{lem}[\cite{CRS}, Lemma 2.10, Proposition 2.11(1)]\label{lem:CRSnonab}
There exists finitely many characters $\chi_\rho$ of reducible representations on $\frak{C}$.
Moreover, for every such character there exists a non-abelian reducible representation $\rho'$ such that $\chi_{\rho'}=\chi_\rho$.
\end{lem}

\noindent In particular, the problem of studying the existence of the set $S$ reduces to that of studying the tame symbol at characters of non-abelian reducible representations on $\frak{C}$.
The remainder of the paper will be spent analyzing such characters for hyperbolic once punctured torus bundles.

%~~~~~~~~~~~~~~~~~~~~~~~~~~~~~~~~~~~~~~~~~~~~~~~~~~~~~~~~~~~
%~~~~~~~~~~~~~~~~~~~~~~~~~~~~~~~~~~~~~~~~~~~~~~~~~~~~~~~~~~~
%~~~~~~~~~~~~~~~~~~~~~~~~~~~~~~~~~~~~~~~~~~~~~~~~~~~~~~~~~~~

\section{Classification of non-abelian reducible representations of once punctured torus bundles}\label{sec:nonabclassification}

As remarked at the end of Section \ref{sec:azalg}, in order to understand the extension problem for the tautological Azumaya algebra we must understand the non-abelian reducible representations of $\pi_1(M)$.
Therefore in this section, we classify non-abelian reducible representations of the fundamental groups of once hyperbolic punctured torus bundles into $\SL_2(\bbC)$.
We will show that these representations bifurcate into three distinct types, depending on whether the image of the stable letter $t$ is plus or minus the identity matrix, is infinite order quasi-unipotent, or neither.
When $t$ has parabolic but not quasi-unipotent image, we will find an explicit form of these representations up to conjugation.
When $t$ does have quasi-unipotent image, we will give a qualitative description of these representations and give infinite families of examples.

Throughout the remainder of this section, we suppose that
\[
\rho : \pi_1(M) \to \SL_2(\bbC),
\]
is a non-abelian reducible representation.
Up to conjugation we can, and do, assume that $\rho$ has upper-triangular image. We will fix
\begin{align*}
\rho(x) &= \begin{pmatrix} \al & r \\ 0 & \al^{-1} \end{pmatrix}, \\
\rho(y) &= \begin{pmatrix} \beta & s \\ 0 & \beta^{-1} \end{pmatrix},\\
\rho(t) &= \begin{pmatrix} w & u \\ 0 & w^{-1} \end{pmatrix},
\end{align*}
for $\al, \beta, w \in \bbC^*$ and $r,s,u \in \bbC$.
Moreover, by using a further conjugation we will assume that either
$$\rho(t) = \begin{pmatrix} w & 1 \\ 0 & w^{-1} \end{pmatrix},$$
or $\rho(t)$ is plus or minus the identity.
To be precise, we can always conjugate $\rho(t)$ to this form unless $u=0$ and $w=\pm 1$, hence the above condition.
The following simple lemma will be useful in the classification of all such $\rho$.

%%%%%%%%%%%%%%%%%%%%%
\begin{lem}\label{lem:RootsOf1}
With $\rho$ as above, $\al$ and $\beta$ are roots of unity.
\end{lem}
%%%%%%%%%%%%%%%%%%%%%

%%%%%%%%%%%%%%%%%%%%%
\begin{pf}
The map
\[
\begin{pmatrix} a & b \\ 0 & a^{-1} \end{pmatrix} \mapsto a,
\]
is a homomorphism from the group of upper-triangular matrices to $\bbC^*$. Composition with $\rho$ defines a homomorphism $\pi_1(M) \to \bbC^*$. Since $\bbC^*$ is abelian and $x,y$ have finite order in the abelianization of $\pi_1(M)$, it follows that their images in $\bbC^*$ must be roots of unity. This proves the lemma.
\end{pf}
%%%%%%%%%%%%%%%%%%%%%

%~~~~~~~~~~~~~~~~~~~~~~~~~~~~~~~~~~~~~~~~~~~~~~~~~~~~~~~~~~~
%~~~~~~~~~~~~~~~~~~~~~~~~~~~~~~~~~~~~~~~~~~~~~~~~~~~~~~~~~~~
%~~~~~~~~~~~~~~~~~~~~~~~~~~~~~~~~~~~~~~~~~~~~~~~~~~~~~~~~~~~

\subsection{Type A representations}\label{sec:typea}

We now classify a particular subclass of representations which we call \emph{Type A representations}.
These are representations where the image of the stable letter is parabolic but not quasi-unipotent.
More precisely, we will analyze the setting where
$$\rho(t) = \begin{pmatrix} w & 1 \\ 0 & w^{-1} \end{pmatrix},$$
for $w$ not a root of unity.
In fact, the analysis below only needs $w\neq\pm 1$.

%%%%%%%%%%%%%%%%%%%%%
\begin{lem}\label{lem:abeliansetup}
If $w\neq \pm1$ then $\rho(x)$ and $\rho(y)$ commute and one of the following holds:
\begin{enumerate}
\item If $r,s\neq 0$, then $\al=\pm 1$ if and only if $\beta=\pm 1$.
\item If $r=0$, then either $\al=\pm 1$ or $s=0$.
\item If $s=0$, then either $\beta=\pm 1$ or $r=0$.
\end{enumerate}
\end{lem}
%%%%%%%%%%%%%%%%%%%%%

%%%%%%%%%%%%%%%%%%%%%
\begin{pf}
Write $l=[x,y]$ and notice that $l$ is the longitude, which commutes with $t$.
Then
\begin{align}
\rho(l) &= \begin{pmatrix} 1 & \al \beta( (\al - \al^{-1})s + (\beta^{-1} - \beta)r) \\ 0 & 1 \end{pmatrix} \label{eqn:upperright},\\
&= \begin{pmatrix} 1 & v \\ 0 & 1 \end{pmatrix},\nonumber
\end{align}
so we have
\begin{equation}\label{eq:MeridianLongitudeCommutator}
[\rho(t), \rho(l)] = \begin{pmatrix} 1 & (w^2 - 1)v \\ 0 & 1 \end{pmatrix},
\end{equation}
which must be the identity matrix.
Since $w^2 - 1 \neq 0$, we must have $v = 0$ and in particular $\rho(x)$ and $\rho(y)$ commute.
Therefore Equation \eqref{eqn:upperright} yields that
$$\al \beta( (\al - \al^{-1})s + (\beta^{-1} - \beta)r) =0,$$
from which the conclusions of the lemma are immediate.
\end{pf}
%%%%%%%%%%%%%%%%%%%%%

%%%%%%%%%%%%%%%%%%%%%
\begin{prop}\label{prop:norootrep}
If $w\neq \pm1$ and either $\al\neq\pm 1$ or $\beta\neq \pm 1$, then $\rho(\pi_1(M))$ is abelian.
\end{prop}
%%%%%%%%%%%%%%%%%%%%%

%%%%%%%%%%%%%%%%%%%%%
\begin{pf}
Without loss of generality we will assume that $\al\neq\pm 1$.
Recall that by Lemma \ref{lem:abeliansetup}, $\rho(x)$ and $\rho(y)$ commute.

First assume that $\beta\neq \pm1$, then we first claim that $\rho(\langle x,y\rangle)$ is a finite abelian group.
Indeed, if $\eta\neq\pm 1$, then a simple induction argument shows that
\begin{equation}\label{eqn:calculation}
\begin{pmatrix} \eta & z \\ 0 & \eta^{-1} \end{pmatrix}^m = \begin{pmatrix} \eta^m & \displaystyle{z \eta^{-(m-1)} \sum_{i = 0}^{m-1} \eta^{2 i}} \\ 0 & \eta^{-m} \end{pmatrix},
\end{equation}
for any natural number $m$ and any $z\in\bbC$. If $\eta$ is an $m^{th}$ root of unity, then $\eta^2$ is also an $m^{th}$ root of unity, and the above expression then reduces to the identity matrix. It follows from Lemma \ref{lem:RootsOf1} that there exists $m,n\in\bbN$ such that $\rho(x)$ (resp.\ $\rho(y)$) has order dividing $m$ (resp.\ $n$), and hence $\rho(\langle x,y\rangle)$ is a finite abelian group as claimed.

It then follows that $\rho(\langle x,y \rangle)$ is simultaneously diagonalizable, and hence we can replace $\rho$ with a conjugate and assume that
\begin{align*}
\rho(x) &= \begin{pmatrix} \al & 0 \\ 0 & \al^{-1} \end{pmatrix}, \\
\rho(y) &= \begin{pmatrix} \beta & 0 \\ 0 & \beta^{-1} \end{pmatrix}, \\
\rho(t) &= \begin{pmatrix} w & u \\ 0 & w^{-1} \end{pmatrix},
\end{align*}
for some $u \in \bbC$ and the same $w$ as above. However, then
\[
\rho(t x t^{-1}) = \begin{pmatrix} \al & (\al^{-1} - \al) u w \\ 0 & \al^{-1} \end{pmatrix} \in \rho(\langle x, y \rangle).
\]
Since $(\al^{-1} - \al)$ and $w$ are nonzero, we must have $u = 0$, i.e., $\rho(t)$ is also diagonal.
This completes the proof when $\al, \beta\neq \pm 1$.

Now assume that $\beta=\pm 1$, then Lemma \ref{lem:abeliansetup}(1) shows that one of $r$ or $s$ is zero.
Moreover, Lemma \ref{lem:abeliansetup}(2) shows therefore that $s=0$ in either case and hence $\rho(y)$ is the identity or its negative.
Moreover the calculation in Equation \eqref{eqn:calculation} shows that $\rho(x)$ again has finite order and hence we may further conjugate $\rho$ so that $\rho(x)$ is diagonal.
As $\rho(y)=\pm \mathrm{Id}$ this conjugation is such that $\rho(\langle x,y\rangle)$ is diagonal and therefore we have that $\rho(t)$ must be diagonal as well, identically as above.
This completes the proof.
\end{pf}

%%%%%%%%%%%%%%%%%%%%%

We are now in a position to classify non-abelian reducible representations for which $w\neq\pm 1$.

%%%%%%%%%%%%%%%%%%%%%
\begin{thm}\label{thm:typearep}
Suppose that $\rho: \pi_1(M) \to \SL_2(\bbC)$ is a non-abelian reducible representation with
\[
\rho(t) = \begin{pmatrix} w & 1 \\ 0 & w^{-1} \end{pmatrix},
\]
with $w \neq \pm 1$.
Then up to conjugation $\rho$ is of the form
$$\rho(x) = (-1)^{\ep_1} \begin{pmatrix} 1 & r \\ 0 & 1 \end{pmatrix},\quad\rho(y) = (-1)^{\ep_2} \begin{pmatrix} 1 & s \\ 0 & 1 \end{pmatrix},\quad\rho(t) = \begin{pmatrix} w & 0 \\ 0 & w^{-1} \end{pmatrix},$$
and $\ep_1, \ep_2 \in \{0, 1\}$.
Moreover $w^2$ is an eigenvalue of $ \Phi$, with notation as in Equation \eqref{eqn:monmatrix}, and $(r,s)$ is a left eigenvector of $ \Phi$ with eigenvalue $w^2$.
Conversely, if $w^2$ is an eigenvalue of $ \Phi$ then the above defines a non-abelian reducible representation of $\pi_1(M)$.
\end{thm}
%%%%%%%%%%%%%%%%%%%%%

\begin{rem}
Before beginning the proof we mention that the fact that the top right entry of $\rho(t)$ changed from $1$ to $0$ is not a typo.
At the end of the proof of Theorem \ref{thm:typearep}, we will show that further conjugation can ensure that this holds.
\end{rem}

%%%%%%%%%%%%%%%%%%%%%
\begin{pf}
By Proposition \ref{prop:norootrep} it is clear that we are reduced to the case where $\al=\pm 1$ and $\beta=\pm 1$.
Moreover if $r=s=0$ then $\rho(x)$ and $\rho(y)$ would be the identity or its negative and hence $\rho(\langle x,y\rangle)$ is contained in the center of $\SL_2(\bbC)$.
Thus $\rho$ would be abelian, so we may assume that at least one of $r,s$ is non-zero.

By Lemma \ref{lem:abeliansetup}, $\rho(\langle x, y \rangle)$ is abelian and we compute that
\begin{align*}
\rho(t) \rho(x) \rho(t)^{-1} &= (-1)^{\ep_1} \begin{pmatrix} 1 & w^2 r \\ 0 & 1 \end{pmatrix}, \\
\rho(\phi(x)) &= (-1)^{\ep_3} \begin{pmatrix} 1 & a r + c s \\ 0 & 1 \end{pmatrix}, \\
\rho(t) \rho(y) \rho(t)^{-1} &= (-1)^{\ep_2} \begin{pmatrix} 1 & w^2 s \\ 0 & 1 \end{pmatrix}, \\
\rho(\phi(y)) &= (-1)^{\ep_4} \begin{pmatrix} 1 & b r + d s \\ 0 & 1 \end{pmatrix},
\end{align*}
for some $\epsilon_3,\epsilon_4\in\{0,1\}$.
Note that for this to give a genuine representation into $\SL_2(\bbC)$ we must have that $\epsilon_1=\epsilon_3$, $\epsilon_2=\epsilon_4$ and in particular this implies that
\begin{align*}
(a - w^2) r + c s &= 0, \\
b r + (d - w^2) s &= 0.
\end{align*}
This means precisely that the vector $(r, s)^T$ is in the cokernel of the matrix
\[
\Phi-w^2\Id=\begin{pmatrix} a - w^2 & b \\ c & d - w^2 \end{pmatrix}.
\]
Since at least one of the entires in $(r,s)$ is non-zero, this matrix is necessarily singular.
Equivalently, we see that $w^2$ is an eigenvalue of $\Phi$.

We have now shown that, up to conjugation, $\rho$ is of the form
$$\rho(x) = (-1)^{\ep_1} \begin{pmatrix} 1 & r \\ 0 & 1 \end{pmatrix},\quad\rho(y) = (-1)^{\ep_2} \begin{pmatrix} 1 & s \\ 0 & 1 \end{pmatrix},\quad\rho(t) = \begin{pmatrix} w & 1 \\ 0 & w^{-1} \end{pmatrix}.$$
Since $|\Tr(\Phi)|>2$, it follows that $w^2$ is not a root of unity.
In particular, $w\neq \pm1$, and hence we can further conjugate this representation by the matrix
$$\begin{pmatrix}
1&(w-w^{-1})^{-1}\\
0&1
\end{pmatrix},$$
to obtain the representation 
$$\rho(x) = (-1)^{\ep_1} \begin{pmatrix} 1 & r \\ 0 & 1 \end{pmatrix},\quad\rho(y) = (-1)^{\ep_2} \begin{pmatrix} 1 & s \\ 0 & 1 \end{pmatrix},\quad\rho(t) = \begin{pmatrix} w & 0 \\ 0 & w^{-1} \end{pmatrix},$$
which is of the form required in the statement of the theorem.

For the backward direction, if $w^2$ is an eigenvalue of $\Phi$, the above suffices to show that $\rho$ is indeed a homomorphism.
\end{pf}

%~~~~~~~~~~~~~~~~~~~~~~~~~~~~~~~~~~~~~~~~~~~~~~~~~~~~~~~~~~~
%~~~~~~~~~~~~~~~~~~~~~~~~~~~~~~~~~~~~~~~~~~~~~~~~~~~~~~~~~~~
%~~~~~~~~~~~~~~~~~~~~~~~~~~~~~~~~~~~~~~~~~~~~~~~~~~~~~~~~~~~

\subsection{Type B representations}\label{sec:typeb}
We now analyze non-abelian representations where $\rho(t)$ is infinite order and has root of unity eigenvalues.
In particular, $\rho(t)$ is quasi-unipotent.
As we saw in Section \ref{sec:typea}, this ensures that $w=\pm 1$ as otherwise $w$ would satisfy the characteristic polynomial for $\Phi$.
In particular, these are representations of the form
\begin{align*}
\rho(x) &= \begin{pmatrix} \al & r \\ 0 & \al^{-1} \end{pmatrix}, \\
\rho(y) &= \begin{pmatrix} \beta & s \\ 0 & \beta^{-1} \end{pmatrix}, \\
\rho(t) &= \begin{pmatrix} \pm 1 & 1 \\ 0 & \pm 1 \end{pmatrix},
\end{align*}
with $\al, \beta$ roots of unity and $r, s \in \bbC$.
We can always replace $\rho(t)$ with $-\rho(t)$ and still have a representation of $\pi_1(M)$ for which $\rho(x)$, $\rho(y)$ have the same form. Therefore, until the statement of the theorem, we instead assume throughout that
\[
\rho(t) = \begin{pmatrix} 1 & \de \\ 0 & 1 \end{pmatrix},
\]
for $\de \in \{\pm 1\}$. Moreover, we can conjugate by a diagonal matrix with eigenvalues $\pm i$ to ensure that $\de = 1$.
Therefore for the time being we work with representations for which
$$\rho(t) = \begin{pmatrix}  1 & 1 \\ 0 &  1 \end{pmatrix}.$$
In this case it is automatic that $\rho(l)$ commutes with $\rho(t)$, where $l$ is again the longitude $[x,y]$, hence we cannot assume that $\rho(l)$ is trivial.
Furthermore, notice that $\rho(\pi_1(M))$ is abelian if $\al, \beta \in \{\pm 1\}$.
Therefore, we assume without loss of generality that $\al$ is a primitive $m^{th}$ root of unity for $m > 2$ and that $\beta$ is a (possibly trivial) primitive root of unity.
As usual, the arguments in the reverse case will be exactly the same.

Conjugating $\rho$ further, as in the end of the proof of Theorem \ref{thm:typearep}, we can moreover assume that
\begin{align*}
\rho(x) &= \begin{pmatrix} \al & 0 \\ 0 & \al^{-1} \end{pmatrix}, \\
\rho(y) &= \begin{pmatrix} \beta & s \\ 0 & \beta^{-1} \end{pmatrix}, \\
\rho(t) &= \begin{pmatrix} 1 & 1 \\ 0 & 1 \end{pmatrix}.
\end{align*}
Now a straightforward computation shows that
\begin{align*}
\rho(t x t^{-1}) &= \begin{pmatrix} \al & \al^{-1} - \al \\ 0 & \al^{-1} \end{pmatrix}, \\
\rho(t y t^{-1}) &= \begin{pmatrix} \beta & s + \beta^{-1} - \beta \\ 0 & \beta^{-1} \end{pmatrix}.
\end{align*}
Moreover, we can write
\begin{align*}\rho(\phi(x)) &= \begin{pmatrix} \alpha^{a}\beta^{c} & s\cdot f(\alpha,\beta) \\ 0 & \alpha^{-a}\beta^{-c} \end{pmatrix},\\
\rho(\phi(y)) &= \begin{pmatrix} \alpha^{b}\beta^{d} & s\cdot g(\alpha,\beta) \\ 0 & \alpha^{-b}\beta^{-d} \end{pmatrix},\end{align*}
for some rational functions $f(\alpha,\beta)$, $g(\alpha,\beta)$ in $\alpha$, $\beta$ which depend only on $\phi$. 
We therefore have the following relations
\begin{align}
\rho(\phi(x)) &= \begin{pmatrix} \alpha^{a}\beta^{c} & s\cdot f(\alpha,\beta) \\ 0 & \alpha^{-a}\beta^{-c} \end{pmatrix} =\begin{pmatrix} \al & \al^{-1} - \al \\ 0 & \al^{-1} \end{pmatrix}\label{eqn:Commutator},\\
\rho(\phi(y)) &= \begin{pmatrix} \alpha^{b}\beta^{d} & s\cdot g(\alpha,\beta) \\ 0 & \alpha^{-b}\beta^{-d} \end{pmatrix} = \begin{pmatrix} \beta & s + \beta^{-1} - \beta \\ 0 & \beta^{-1} \end{pmatrix}\label{eqn:Commutator2}.
\end{align}
Combining the results of this section and bearing in mind that we may have replaced $\rho(t)$ by $-\rho(t)$, we have the following theorems which are stated using the notation of Equations \eqref{eqn:Commutator} and \eqref{eqn:Commutator2}.

\begin{thm}\label{thm:wtfrepspm1}
Suppose there exists a non-abelian reducible representation $\rho:\pi_1(M)\to\SL_2(\bbC)$ in which $\rho(y)$ and $\rho(t)$ have eigenvalues contained in $\{-1,1\}$.
Then up to conjugation $\rho$ is of the form
$$\rho(x)= \begin{pmatrix} \al & 0 \\ 0 & \al^{-1} \end{pmatrix},\quad \rho(y) = \begin{pmatrix} \pm 1 & s \\ 0 & \pm 1 \end{pmatrix}, \quad \rho(t) = \begin{pmatrix}  \pm 1 & 1 \\ 0 &  \pm 1 \end{pmatrix},$$
where $\alpha$ is a primitive $m^{th}$ root of unity for $m>2$ dividing $|a+d-2|$, $s\neq 0$, and the equations
\begin{equation}\label{eqn:wtfpm1}
\alpha^{a-1}(\pm 1)^c=1,\quad \alpha^{b}(\pm 1)^{c-1}=1,\quad g(\alpha,\pm 1)=1,
\end{equation}
all hold where $g(\alpha,\pm 1)$ is a unique rational function in $\alpha$ associated to $\phi$.
Similarly, if $\rho(x)$ and $\rho(t)$ have eigenvalues contained in $\{-1,1\}$ then the identical statement holds with the places of $\rho(x)$ and $\rho(y)$ interchanged.
\end{thm}

\begin{thm}\label{thm:wtfreps}
Suppose there exists a non-abelian reducible representation $\rho:\pi_1(M)\to\SL_2(\bbC)$ in which $\rho(t)$ has eigenvalues contained in $\{-1,1\}$ and neither $\rho(x)$ nor $\rho(y)$ do.
Then up to conjugation $\rho$ is of the form
$$\rho(x)= \begin{pmatrix} \al & 0 \\ 0 & \al^{-1} \end{pmatrix},\quad \rho(y) = \begin{pmatrix} \beta & s \\ 0 & \beta^{-1} \end{pmatrix}, \quad \rho(t) = \begin{pmatrix}  \pm 1 & 1 \\ 0 & \pm  1 \end{pmatrix},$$
where $\alpha,\beta\neq \pm 1$ are $m^{th}$ roots of unity for $m$ dividing $|a+d-2|$, $s\neq 0$, and the equations
\begin{equation}\label{eqn:wtf}
\alpha^{a-1}(\beta)^c=1,\quad \alpha^{b}(\beta)^{c-1}=1,\quad s=\frac{\al^{-1}-\al}{f(\alpha,\beta)}=\frac{\beta^{-1}-\beta}{g(\alpha,\beta)-1},
\end{equation}
all hold where $f(\alpha,\beta)$, $g(\alpha,\beta)$ are unique rational functions in $\alpha$, $\beta$ associated to $\phi$.
\end{thm}

\begin{rem}
One can write down explicitly what the functions $f(\alpha,\beta)$, $g(\alpha,\beta)$ are based on the spelling of the words $\phi(x)$, $\phi(y)$ however the statement is complicated and, since we will not need this in this paper, we choose to leave Theorem \ref{thm:wtfrepspm1} and \ref{thm:wtfreps} at the current level of precision.
\end{rem}

We point out that, though the latter conditions of Equations \eqref{eqn:wtfpm1} and \eqref{eqn:wtf} look complicated in general, there are infinite families of manifolds with non-abelian reducible representations fitting the criteria of Theorem \ref{thm:wtfrepspm1} and the criteria of Theorem \ref{thm:wtfreps}.
We now describe these families and their representations in detail.

\begin{ex}\label{ex:tunnelone}
Let $M_j$ be a once punctured torus bundle of tunnel number one.
Then $\phi_j=RL^j$ and therefore
$$\phi_j(x)=x(yx)^j,\quad\phi_j(y)=yx.$$
Notice in particular that when $j=1$, $M_1$ is the figure eight knot complement.
The monodromy matrix of $\phi_j$ is given by 
$$\Phi_j=\begin{pmatrix}
j+1&1\\
j&1
\end{pmatrix},$$
and therefore $\alpha$ and $\beta$ are always $j^{th}$ roots of unity (in fact, it always follows that $\alpha=1$).
Provided $j>1$, we claim that the non-abelian representations
\begin{align*}
\rho_j(x) &=\begin{pmatrix} \al & s \\ 0 & \al^{-1}  \end{pmatrix}= \begin{pmatrix} 1 & \eta_j^{-2}-1 \\ 0 & 1  \end{pmatrix}, \\
\rho_j(y) &= \begin{pmatrix} \beta & 0 \\ 0 & \beta^{-1} \end{pmatrix}= \begin{pmatrix} \eta_j & 0 \\ 0 & \eta_j^{-1} \end{pmatrix}, \\
\rho_j(t) &= \begin{pmatrix} 1 & 1 \\ 0 & 1 \end{pmatrix},
\end{align*}
give a family of representations of $\pi_1(M_j)$ fitting the criteria of Theorem \ref{thm:wtfrepspm1}.
We remark that in the notation of Theorem \ref{thm:wtfrepspm1} the positions of $\alpha$, $\beta$ are swapped so in particular $\alpha=1$, $\beta=\eta_j$, $s=\eta_j^{-2}-1$, and $g(1,\beta)=\sum_{i=0}^j\beta^{j-2i}$.
Indeed it is straightforward to check that
$$\rho_j(tyt^{-1})=\rho_j(\phi_j(y))=\rho_j(yx),$$
and that
$$\rho_j(\phi_j(x))=\begin{pmatrix}
1&s\sum_{i=0}^j\eta_j^{j-2i}\\
0&1
\end{pmatrix}=\begin{pmatrix}
1&s\\
0&1
\end{pmatrix}=\rho_j(txt^{-1}),$$
where the middle inequality follows since $\eta_j^2$ is a $j^{th}$ root of unity.
\end{ex}

\begin{ex}\label{ex:nfold}
Let $N_j$ be the once punctured torus bundle with monodromy given by $\phi_j=(RL)^j$. 
When $j=1$, $N_{1}$ is the figure eight knot complement and all other $N_j$ are $j$-fold cyclic covers of $N_1$.
One can compute that the monodromy matrix of $\phi_j$ is given by 
 $$\Phi_j=\begin{pmatrix}
2&1\\
 1&1
 \end{pmatrix}^j=\begin{pmatrix}
F_{2j+1}&F_{2j}\\
 F_{2j}&F_{2j-1}
 \end{pmatrix},$$
 where $F_j$ is the $j^{th}$ Fibonacci number, indexed so that $F_1=F_2=1$.
The abelianization is therefore given by
\begin{equation}\label{eqn:fibabel}
\pi_1(N_j)^{ab}=H_1(N_j,\bbZ)=\begin{cases}
\bbZ\oplus\bbZ/F_j\bbZ\oplus\bbZ/5F_j\bbZ,&\text{j even},\\
\bbZ\oplus\bbZ/G_j\bbZ\oplus\bbZ/G_j\bbZ,&\text{j odd},
\end{cases}\end{equation}
where $G_j$ is the $j^{th}$ Lucas number $G_j=F_{j-1}+F_{j+1}$ with the convention that $G_1=1$ (see \cite[\S 3]{MR2}).
We restrict ourselves to even $j$, in which case the image of $x,y$ in $\pi_1(N_j)^{ab}$ have order $5F_j$.
Let $\alpha=\eta_5$, a primitive $5^{th}$ root of unity, and let $\beta=\alpha^{2}$.
Each $\pi_1(N_j)$ can be realized as the subgroup of $\pi_1(N_1)$ generated by $\langle x,y,t^j\rangle$. 
We then obtain a non-abelian reducible representation $\rho_j$ with the following images of the generators
 \begin{align*}
 \rho(x)&=\begin{pmatrix} \al & 0 \\ 0 &  \al^{-1} \end{pmatrix},\\
 \rho(y)&= \begin{pmatrix}  \beta & \frac{1-\al^{-1}}{m} \\ 0 &  \beta^{-1} \end{pmatrix},\\
 \rho(t_j)&=\begin{pmatrix} 1 & 1 \\ 0 & 1 \end{pmatrix},
 \end{align*}
where $m$ is such that $j=2m$.
The subscript on the element $t_j$ above is meant to indicate that this is the stable letter of $\pi_1(N_j)$, which is the same element as $t^j$ when $\pi_1(N_j)$ is considered as a subgroup of $\pi_1(N_1)$.
Notice in particular that this produces an infinite family of representations for which the criteria of Theorem \ref{thm:wtfreps} is satisfied.
\end{ex}

%~~~~~~~~~~~~~~~~~~~~~~~~~~~~~~~~~~~~~~~~~~~~~~~~~~~~~~~~~~~
%~~~~~~~~~~~~~~~~~~~~~~~~~~~~~~~~~~~~~~~~~~~~~~~~~~~~~~~~~~~
%~~~~~~~~~~~~~~~~~~~~~~~~~~~~~~~~~~~~~~~~~~~~~~~~~~~~~~~~~~~

\subsection{Type C representations}\label{sec:typec}

The final type of non-abelian reducible representation of $\pi_1(M)$ are those where $\rho(t)=\pm \Id$, which we assume for the remainder of the subsection.
As in Section \ref{sec:typeb}, this implies we can write
\begin{align*}
\rho(x) &= \begin{pmatrix} \al & r \\ 0 & \al^{-1} \end{pmatrix}, \\
\rho(y) &= \begin{pmatrix} \beta & s \\ 0 & \beta^{-1} \end{pmatrix}, \\
\rho(t) &= \pm\begin{pmatrix} 1 & 0 \\ 0 & 1 \end{pmatrix}.
\end{align*}
It is straightforward to see that if $\rho(x)$, $\rho(y)$ have eigenvalues contained in $\{-1,1\}$ then $\rho(\pi_1(M))$ is abelian hence we assume that at least one of $\al,\beta\neq \pm1$.
For the time being assume that $\al\neq \pm 1$.
Then $\rho$ is conjugate to a representation where $\rho(x)$ is similarly diagonal and by a further conjugation we may assume that the upper right entry of $\rho(y)$ is 1, provided $\rho$ is non-abelian.
Therefore, up to conjugation, $\rho$ is of the form
\begin{align*}
\rho(x) &= \begin{pmatrix} \al & 0 \\ 0 & \al^{-1} \end{pmatrix}, \\
\rho(y) &= \begin{pmatrix} \beta & 1 \\ 0 & \beta^{-1} \end{pmatrix}, \\
\rho(t) &= \pm\begin{pmatrix} 1 & 0 \\ 0 & 1 \end{pmatrix}.
\end{align*}
Applying the similar analysis to Section \ref{sec:typeb} we obtain the following.

\begin{thm}\label{thm:typecrep}
Suppose there exists a non-abelian reducible representation $\rho:\pi_1(M)\to\SL_2(\bbC)$ for which $\rho(t)=\pm \Id$. Then up to conjugation $\rho$ is either of the form
$$\rho(x)  = \begin{pmatrix} \al & 0 \\ 0 & \al^{-1} \end{pmatrix},\quad
\rho(y) = \begin{pmatrix} \beta & 1 \\ 0 & \beta^{-1} \end{pmatrix},\quad
\rho(t) = \pm\begin{pmatrix} 1 & 0 \\ 0 & 1 \end{pmatrix},$$
where $\al$ is a primitive $m^{th}$ root of unity for $m>2$ dividing $|a+b-2|$ and $\beta$ is a (possibly trivial) primitive root of unity with order dividing $|a+b-2|$ or of the form
$$\rho(x)  = \begin{pmatrix} \pm 1 & 1 \\ 0 & \pm 1 \end{pmatrix},\quad
\rho(y) = \begin{pmatrix} \beta & 0 \\ 0 & \beta^{-1} \end{pmatrix},\quad
\rho(t) = \pm\begin{pmatrix} 1 & 0 \\ 0 & 1 \end{pmatrix},
$$
where $\beta$ is a primitive $m^{th}$ root of unity for $m>2$ dividing $|a+b-2|$..
\end{thm}

Note that examples of such representations exist, for instance one can see that the following gives a representation of Type C.

\begin{ex}\label{ex:typec}
Continuing the notation from Example \ref{ex:nfold}, we let $N_{3j}$ be the once punctured torus bundle with monodromy given by $\phi_{3j}=(RL)^{3j}$ for any $j\in\bbN$.
Following Equation \eqref{eqn:fibabel}, one can check that these are precisely the cyclic covers of the figure eight knot for which the images of $x$ and $y$ in the abelianization have order divisible by $4$.
This follows from the sixfold periodicity of $F_j$ and $G_j$ modulo $4$.
Taking $\alpha=\beta=i$, one can check that there is a non-abelian reducible representation of $\pi_1(N_{3j})$ obtained by setting
\begin{align*}
\rho(x)&=\begin{pmatrix} i & 0 \\ 0 &  -i \end{pmatrix},\\
\rho(y)&= \begin{pmatrix}  i & 1 \\ 0 &  -i \end{pmatrix},\\
\rho(t)&=\begin{pmatrix} 1 & 0 \\ 0 & 1 \end{pmatrix},
\end{align*}
which one checks satisfies the relations of $\pi_1(N_{3j})$.
This gives an infinite family of representations satisfying the conditions of Theorem \ref{thm:typecrep}.
One can also check that characters of the above representations are not characters of Type B representations, so in particular this is a distinct family of non-abelian reducible representation.
\end{ex}

As a corollary of Lemma \ref{lem:CRSnonab}, Theorems \ref{thm:typearep}, and Theorems \ref{thm:wtfrepspm1}--\ref{thm:typecrep} we have the following, where the proof follows from examining the possible images of $t$ under any non-abelian reducible representation.

\begin{cor}\label{cor:typeaorb}
Suppose that $\frak{C}\subset \frak{X}(M)_k$ is a fixed geometrically irreducible component such that $\frak{C}_\bbC$ contains the character of an irreducible representation.
If $\chi_\rho\in \frak{C}$ is a character of a reducible representation, then there exists a representation $\rho'$ of Type A, Type B, or Type C such that $\chi_{\rho}=\chi_{\rho'}$.
\end{cor}

%~~~~~~~~~~~~~~~~~~~~~~~~~~~~~~~~~~~~~~~~~~~~~~~~~~~~~~~~~~~
%~~~~~~~~~~~~~~~~~~~~~~~~~~~~~~~~~~~~~~~~~~~~~~~~~~~~~~~~~~~
%~~~~~~~~~~~~~~~~~~~~~~~~~~~~~~~~~~~~~~~~~~~~~~~~~~~~~~~~~~~

\section{The proof of Theorem \ref{thm:main1}}

In this section, we prove Theorem \ref{thm:main1} which we now recall for the reader's convenience.

\mainone*
Before proceeding to the proof of this theorem, we will first need to prove Proposition \ref{prop:galois} which, via Corollary \ref{cor:goodtraceextend}, will help connect square roots of eigenvalues of the monodromy matrix $\Phi$ from Equation \eqref{eqn:monmatrix} with a certain field extension property related to the tame symbol.

\begin{prop}\label{prop:galois}
Let $w$ be a root of the polynomial $f(x)=x^4-c x^2+1$ for some $c\in\bbZ$ such that $|c|> 2$.
Then $\bbQ(w)=\bbQ(w+1/w)$ if and only if $c=a^2+2$ for some $a\in\bbN$.
\end{prop}
\begin{proof}
By the rational roots test and the assumption on $c$, there are visibly no roots of $f(x)$ in $\bbQ$.
Therefore if $f(x)$ reduces over $\bbQ$ then it must factor as one of the following
\begin{equation}\label{eqn:polyfactor}
f(x)=\begin{cases}
(x^2+ax+1)(x^2-ax+1),&c=a^2-2,\\
(x^2+ax-1)(x^2-ax-1),&c=a^2+2,
\end{cases}
\end{equation}
for some $a\in\bbZ$.
The condition that $|c|>2$ implies that $a\neq 0$.

Assume first that $c\neq a^2\pm 2$ and therefore that $f(x)$ is irreducible.
Then for any root $w$ of $f(x)$ we conclude that $[\bbQ(w):\bbQ]=4$ and from the equation for $f(x)$ it follows that
$$w^2+w^{-2}=\frac{c\pm\sqrt{c^2-4}}{2}+\frac{c\mp\sqrt{c^2-4}}{2}=c.$$
Therefore $\bbQ(w^2+w^{-2})=\bbQ$ and moreover as
$$(w+w^{-1})^2-2=w^2+w^{-2},$$
it is evident that $\bbQ(w+w^{-1})$ is at most quadratic over $\bbQ$.
By degree considerations it follows that $\bbQ(w+w^{-1})\subsetneq\bbQ(w)$.

Now assume that $f(x)$ is reducible and factors as
$$f(x)=(x^2+ax+1)(x^2-ax+1),\quad c=a^2-2.$$
From the assumption that $|c|>2$, it follows that $|a|\ge 3$ and therefore the rational roots test again shows that both factors of $f(x)$ are irreducible.
A straightforward computation shows in this case that $w+w^{-1}=\pm a\in\bbQ$.
As each factor of $f(x)$ is irreducible it follows that $\bbQ(w+w^{-1})\subsetneq\bbQ(w)$

Finally, assume that $f(x)$ is reducible and factors as
$$f(x)=(x^2+ax-1)(x^2-ax-1),\quad c=a^2+2.$$
Now the condition that $|c|>2$ implies that $a\neq 0$ and we conclude again that each factor is irreducible.
In this case, we obtain that
$$w+w^{-1}=\sqrt{a^2+4}\in\bbQ(w)\setminus\bbQ,$$
and degree considerations therefore show that $\bbQ(w)=\bbQ(w+w^{-1})$.
Combining this with the above analysis we see that $\bbQ(w)=\bbQ(w+w^{-1})$ if and only if $c=a^2+2$ for some $a\in\bbN$, as required.
\end{proof}

Using the condition of the previous proposition, we can now analyze the tame symbol at any Type A representation. 
In what follows, we remind the reader that $\Pi_k:\widetilde{C}\dashrightarrow \frak{C}$ is the birational morphism as in Section \ref{sec:curveback}.

\begin{cor}\label{cor:goodtraceextend}
Let $\frak{C}$ be a geometrically irreducible curve component of $\frak{X}(M)_k$ with field of definition $k$ such that $\frak{C}$ contains the character of an irreducible representation and let $\widetilde{C}$ be the corresponding smooth projective completion.
Let $p\in\widetilde{C}$ be such that $\chi_\rho=\Pi_k(p)\in \frak{C}$ where $\rho$ is a representation of Type A.
If $\Tr(\Phi)=a^2+2$ for some $a\in\bbN$, then the tautological Azumaya algebra $\mathcal{A}_{k(\widetilde{C})}$ extends over $p$.
\end{cor}
\begin{proof}
Let $\rho_\frak{C}$ be the associated tautological representation over $\frak{C}$ as in Section \ref{sec:curveback}.
Recall by Theorem \ref{thm:typearep} that at a Type A representation, $\rho(t)$ has eigenvalues $ w^{\pm 1}$ where $w^2$ is an eigenvalue of $\Phi$.
Since $|\Tr(\Phi)|>2$ we know that $w\neq\pm 1$ and it therefore follows that $\chi_\rho(t)\neq \pm 2$, thus $\rho_\frak{C}(t)\neq \pm\mathrm{Id}$.
Hence $\widetilde{I}_t^2-4$ is not identically zero on $\widetilde{C}$ so by Lemma \ref{lem:choosehilbert} there exists some $h\in\pi_1(M)$ such that $\mathcal{A}_{k(\widetilde{C})}$ has Hilbert symbol
$$\mathcal{A}_{k(\widetilde{C})}=\left(\frac{\widetilde\kappa,\widetilde\nu}{k(\widetilde{C})}\right)=\left(\frac{\widetilde{I}_t^2-4,\widetilde{I}_{[t,h]}-2}{k(\widetilde{C})}\right).$$
We will show that the tame symbol is trivial at $p$.

Since $I_t(\chi_\rho)\neq\pm 2$, it follows that $I_t(\chi_\rho)^2-4\neq 0$ and since
$$\widetilde{I}_t(p)= I_t(\Pi_k(p))=I_t(\chi_\rho),$$
we find that $\ord_{p}(\widetilde\kappa)=0$.
Similarly, as the representation $\rho$ is reducible, $I_{[t,h]}(\chi_\rho)=2$ and hence $\ord_{p}(\widetilde\nu)>0$.
When $\ord_{p}(\widetilde\nu)$ is even, Equation \eqref{eqn:tamesymbol} shows that the tame symbol $\{\widetilde\kappa,\widetilde\nu\}_{p}$ is trivial.
When $\ord_{p}(\widetilde\nu)$ is odd, we compute that
$$\{\widetilde\kappa,\widetilde\nu\}_{p}=\left[\widetilde\kappa(p)\right]=\left[\widetilde{I}_t(p)^2-4\right]\in k(p)^*/k(p)^{*2}.$$
The condition that $\bbQ(w+1/w)=\bbQ(w)$ implies that
$$k(w+1/w)=k(w)=k(w-1/w),$$
and since by Lemma \ref{lem:residuefield} we have the inclusion of fields, $k(w+1/w)\subset k(p)$, it follows that
\begin{equation}\label{eqn:tamestar1}
\{\widetilde\kappa,\widetilde\nu\}_{p}=\left[(w+1/w)^2-4\right]=\left[(w-1/w)^2\right]=[1]\in k(p)^*/k(p)^{*2}.
\end{equation}
We have therefore shown in both cases that $\{\widetilde\kappa,\widetilde\nu\}_{p}=1$ and hence $\mathcal{A}_{k(\widetilde{C})}$ extends over $p$ as required.
\end{proof}

\begin{rem}
In fact, we will show that Corollary \ref{cor:goodtraceextend} is an if and only if provided that $k$ satisfies some additional hypotheses. This is the content of Lemma \ref{prop:noextendA}.
\end{rem}

We are now in a position to complete the proof of Theorem \ref{thm:main1}.

\begin{proof}[Proof of Theorem \ref{thm:main1}]
Let $\frak{C}$ be a fixed geometrically irreducible component of $\frak{X}(M)_k$ containing the character of an irreducible representation of $\pi_1(M)$, where $k$ is the field of definition of $\frak{C}$.
Let $\mathcal{A}_{k(\widetilde{C})}$ be the associated tautological Azumaya algebra.
By Proposition \ref{prop:idealirrepextend}, $\mathcal{A}_{k(\widetilde{C})}$ extends over all points $p\in\widetilde{C}$ with the exception of perhaps those for which $\Pi_k(p)$ is the character of a non-abelian reducible representation, so it suffices to study the behavior of $\mathcal{A}_{k(\widetilde{C})}$ at the latter.
Corollary \ref{cor:typeaorb} implies that for such points, $\Pi_k(p)=\chi_\rho$ is the character of a Type A, B, or C representation.

When $\Pi_k(p)=\chi_\rho$ and $\rho$ is a Type A representation, Corollary \ref{cor:goodtraceextend} shows that $\mathcal{A}_{k(\widetilde{C})}$ extends over $p$ in the case that $\Tr(\Phi)=a^2+2$ and similarly Equation \eqref{eqn:tamestar1} shows that $\mathcal{A}_{k(\widetilde{C})}$ extends over $p$ in the case that $(w-1/w)\in k$.
When $\rho$ is a Type B or Type C representation, let $g$ be the choice of $x$ or $y$ such that $\rho(g)$ has primitive $m^{th}$ root of unity eigenvalues which are not contained in $\{-1,1\}$.
At least one of $x,y$ must have this property.
Note that $m\mid n$ for $n$ as in Condition $(\star_2)$.
Since $\chi_\rho(g)\neq\pm 1$, $\widetilde{I}_g^2-4$ is not identically zero on $\widetilde{C}$ and Lemma \ref{lem:choosehilbert} shows that there is some $h\in\pi_1(M)$ such that $\mathcal{A}_{k(\widetilde{C})}$ has a Hilbert symbol of the form
$$\mathcal{A}_{k(\widetilde{C})}=\left(\frac{\widetilde\kappa,\widetilde\nu}{k(\widetilde{C})}\right)=\left(\frac{\widetilde{I}_g^2-4,\widetilde{I}_{[g,h]}-2}{k(\widetilde{C})}\right).$$
Again since $I_g(\chi_\rho)\neq\pm 2$, $\ord_p(\widetilde\kappa)=0$ and since $\rho$ is reducible, $\ord_p(\widetilde\nu)>0$.
Since $k(\eta_m)\subset k(\eta_n)\subset k(p)$, it follows that the tame symbol
$$\{\widetilde\kappa,\widetilde\nu\}_{p}=\left[\widetilde\kappa(p)\right]=\left[(\eta_m-1/\eta_m)^2\right]=[1]\in k(p)^*/k(p)^{*2},$$
is trivial by Lemma \ref{lem:residuefield}.
Therefore $\mathcal{A}_{k(\widetilde{C})}$ extends over all points $p\in \widetilde{C}$ and hence there exists an extension of $\mathcal{A}_{k(\widetilde{C})}$ to the Azumaya algebra $\mathcal{A}_{\widetilde{C}}$ over the smooth projective completion $\widetilde{C}$.  
\end{proof}

%~~~~~~~~~~~~~~~~~~~~~~~~~~~~~~~~~~~~~~~~~~~~~~~~~~~~~~~~~~~
%~~~~~~~~~~~~~~~~~~~~~~~~~~~~~~~~~~~~~~~~~~~~~~~~~~~~~~~~~~~
%~~~~~~~~~~~~~~~~~~~~~~~~~~~~~~~~~~~~~~~~~~~~~~~~~~~~~~~~~~~

\section{Abelian representations and the Alexander invariant}\label{sec:convthm}

In this section we will prove Theorem \ref{thm:main2}, the partial converse to Theorem \ref{thm:main1}.
The key difficulty in proving Theorem \ref{thm:main2} is that, in our setting, establishing non-triviality of a tame symbol is more subtle than establishing its triviality.
Specifically, one needs to know not only the values of the functions defining a Hilbert symbol for $\mathcal{A}_{k(\widetilde{C})}$ at a point $p\in\widetilde{C}$ but one also needs to understand their valuations in the corresponding local ring, the latter of which presented no issue in the proof of Theorem \ref{thm:main1}.
Moreover, the discrepancy between $\frak{C}$ and $\widetilde{C}$ will also appear and we will need to additionally know that Type A, B, and C representations are reduced points of $\frak{C}$ so that the corresponding local rings are isomorphic.

To this end, Section \ref{sec:HPrecap} will be devoted to work of Heusener--Porti which shows that Type A, B, and C representations are reduced points of the $\PSL_2(\bbC)$-character schemes, Section \ref{sec:replift} will discuss how to lift this result to the $\SL_2(\bbC)$-character scheme $\frak{X}(M)_\bbC$, and finally Section \ref{sec:localcoord} will connect this work to tame symbols of the tautological Azumaya algebra.

%~~~~~~~~~~~~~~~~~~~~~~~~~~~~~~~~~~~~~~~~~~~~~~~~~~~~~~~~~~~
%~~~~~~~~~~~~~~~~~~~~~~~~~~~~~~~~~~~~~~~~~~~~~~~~~~~~~~~~~~~
%~~~~~~~~~~~~~~~~~~~~~~~~~~~~~~~~~~~~~~~~~~~~~~~~~~~~~~~~~~~

\subsection{The Alexander invariant and deformations of non-abelian representations}\label{sec:HPrecap}

Given a non-abelian reducible representation $\overline{\rho}:\pi_1(M)\to \PSL_2(\bbC)$ we can always conjugate $\overline{\rho}$ so that it is in the standard form
$$\overline{\rho}(\gamma)=\begin{pmatrix}
a_\gamma&b_\gamma\\
0&a^{-1}_\gamma\end{pmatrix},$$
where $b_\gamma\neq 0$ for some $\gamma\in\pi_1(M)$.
Such a representation always gives rise to an abelian representation $\overline{\rho}^{ab}:\pi_1(M)\to\PSL_2(\bbC)$ with $\overline{\chi_{\overline{\rho}}}=\overline{\chi_{\overline{\rho}^{ab}}}$ by forgetting the top right entry, that is, by defining
$$\overline{\rho}^{ab}(\gamma)=\begin{pmatrix}
a_\gamma&0\\
0&a^{-1}_\gamma\end{pmatrix}.$$
In \cite{HeusenerPorti}, Heusener--Porti study the inverse question of when such an abelian representation has the same character as a non-abelian reducible representation.
In this subsection we recall their results, focusing on the case of once punctured torus bundles but remark that they hold more generally for one cusped 3-manifolds with torus cusp.
We will state all of their results for the $\PSL_2(\bbC)$-character scheme, as that is how they are written in \cite{HeusenerPorti}, and use the notation $\overline{\rho}$ to indicate that this is a $\PSL_2(\bbC)$-representation.

Suppose that we have a non-trivial homomorphism $f:\pi_1(M)\to\bbC^*$ and let $f^{\frac{1}{2}}:\pi_1(M)\to\bbC^*$ be a map (not necessarily a homomorphism) which satisfies the equation $(f^{\frac{1}{2}}(\gamma))^2=f(\gamma)$ for all $\gamma\in\pi_1(M)$.
Then we can define an abelian representation $\overline{\rho}_f:\pi_1(M)\to\PSL_2(\bbC)$ via the formula
\begin{equation}\label{eqn:abrep}
\overline{\rho}_f(\gamma)=\pm\begin{pmatrix}
f^{\frac{1}{2}}(\gamma)&0\\
0&f^{-\frac{1}{2}}(\gamma)\end{pmatrix},
\end{equation}
where $f^{-\frac{1}{2}}(\gamma)=(f^{\frac{1}{2}}(\gamma))^{-1}$ for all $\gamma\in\pi_1(M)$, provided that $ f^{\frac{1}{2}}(\gamma\delta)=\pm f^{\frac{1}{2}}(\gamma)f^{\frac{1}{2}}(\delta)$ for all $\gamma,\delta\in\pi_1(M)$.
The inclusion of the $\pm$ above owes to the fact that $\pm\mathrm{Id}$ are the same element of $\PSL_2(\bbC)$.

Note that the non-triviality assumption on $f$ implies that $f^{\pm\frac{1}{2}}(\gamma)\neq \pm 1$ for some $\gamma$, hence the question of when $\overline{\rho}_f$ has the same character as a non-abelian reducible is equivalent to the existence of a map $d:\pi_1(M)\to\mathbb{C}$ such that 
\begin{equation}\label{eqn:nonabrep}
\overline{\rho}^d_f(\gamma)=\pm\begin{pmatrix}
1&d(\gamma)\\
0&1\end{pmatrix}\overline{\rho}_f(\gamma)=\pm\begin{pmatrix}
f^{\frac{1}{2}}(\gamma)&f^{-\frac{1}{2}}(\gamma)d(\gamma)\\
0&f^{-\frac{1}{2}}(\gamma)\end{pmatrix},
\end{equation}
is a non-abelian representation.
The calculations
\begin{equation}\label{eqn:dcomputation}\overline{\rho}^d_f(\gamma\delta)=\pm\begin{pmatrix}
f^{\frac{1}{2}}(\gamma\delta)&f^{-\frac{1}{2}}(\gamma\delta)d(\gamma\delta)\\
0&f^{-\frac{1}{2}}(\gamma\delta)\end{pmatrix}
,\end{equation}
$$
\overline{\rho}^d_f(\gamma)\overline{\rho}^d_f(\delta)=\pm\begin{pmatrix}
f^{\frac{1}{2}}(\gamma)f^{\frac{1}{2}}(\delta)&f^{-\frac{1}{2}}(\delta)(f^{\frac{1}{2}}(\gamma)d(\delta)+f^{-\frac{1}{2}}(\gamma)d(\gamma))\\
0&f^{-\frac{1}{2}}(\gamma)f^{-\frac{1}{2}}(\delta)\end{pmatrix},
$$
show that $\overline{\rho}_f^d$ is a homomorphism if and only if 
$$ f^{-\frac{1}{2}}(\gamma\delta)d(\gamma\delta)=\pm f^{-\frac{1}{2}}(\delta)(f^{\frac{1}{2}}(\gamma)d(\delta)+f^{-\frac{1}{2}}(\gamma)d(\gamma)),$$
where $ f^{\frac{1}{2}}(\gamma\delta)=\pm f^{\frac{1}{2}}(\gamma)f^{\frac{1}{2}}(\delta)$.
This is equivalent to
$$ d(\gamma\delta)=f^{\frac{1}{2}}(\gamma)(f^{\frac{1}{2}}(\gamma)d(\delta)+f^{-\frac{1}{2}}(\gamma)d(\gamma))=d(\gamma)+f(\gamma)d(\delta).$$
Let $\bbC_f$ denote $\bbC$ regarded as a $\pi_1(M)$-module via left multiplication under the homomorphism $f:\pi_1(M)\to \mathbb{C}^*$, then we conclude that the condition that $\overline{\rho}_f^d$ is a homomorphism is equivalent to the condition that $d\in Z^1(\pi_1(M),\bbC_f)$, where the latter set denotes the $1$-cocycles.

From Equation \eqref{eqn:nonabrep} and a straightforward linear algebra argument, one also deduces that $\overline{\rho}^d_f$ is abelian if and only if there exists $\tau\in\bbC^*$ and a matrix
$$A=\begin{pmatrix}
1&\tau\\
0&1\end{pmatrix},$$
such that $A \overline{\rho}_f^d(\gamma)A^{-1}$ is a diagonal matrix with eigenvalues $f^{\pm\frac{1}{2}}(\gamma)$.
Examining the upper right entry of this conjugation, this condition is equivalent to requiring that
$$f^{-\frac{1}{2}}(\gamma)d(\gamma)=\tau(f^{\frac{1}{2}}(\gamma)-f^{-\frac{1}{2}}(\gamma)),$$
which is to say that there exists $\tau\in\bbC^*$ such that
$$d(\gamma)=-(1-f(\gamma))\tau,$$
for each $\gamma\in\pi_1(M)$.
Hence we see that $\overline{\rho}_f^d$ is abelian if and only if $d\in B^1(\pi_1(M),\bbC_f)$, the $1$-coboundaries.

Now suppose that $d\in Z^1(\pi_1(M),\bbC_f)$ and let $d'=cd+b\in Z^1(\pi_1(M),\bbC_f)$ for any $c\in \bbC^*$, $b\in B^1(\pi_1(M),\bbC_f)$.
We claim that the representations $\overline{\rho}_f^{d}$, $\overline{\rho}_f^{d'}$ are conjugate.
Indeed, we can assume that $d,d'\notin B^1(\pi_1(M),\bbC_f)$ else we are done by the previous paragraph.
Write $b(\gamma)=(1-f(\gamma))\tau$ for some $\tau\in\bbC^*$ and let $\omega\in\bbC^*$ be such that $\omega^2=c$.
Then
\begin{align*}
\begin{pmatrix}
1&\tau\\
0&1\end{pmatrix}&\begin{pmatrix}
\omega&0\\
0&\frac{1}{\omega}\end{pmatrix}\overline{\rho}_f^{d}(\gamma)\begin{pmatrix}
\frac{1}{\omega}&0\\
0&\omega\end{pmatrix}
\begin{pmatrix}
1&-\tau\\
0&1\end{pmatrix}
=\begin{pmatrix}
1&\tau\\
0&1\end{pmatrix}\overline{\rho}_f^{cd}(\gamma)\begin{pmatrix}
1&-\tau\\
0&1\end{pmatrix},\\
&=\begin{pmatrix}
1&\tau\\
0&1\end{pmatrix}\begin{pmatrix}
f^{\frac{1}{2}}(\gamma)&f^{-\frac{1}{2}}(\gamma)cd(\gamma)\\
0&f^{-\frac{1}{2}}(\gamma)\end{pmatrix}\begin{pmatrix}
1&-\tau\\
0&1\end{pmatrix},\\
&=\begin{pmatrix}
f^{\frac{1}{2}}(\gamma)&f^{-\frac{1}{2}}(\gamma)cd(\gamma)+\tau(f^{-\frac{1}{2}}(\gamma)-f^{\frac{1}{2}}(\gamma))\\
0&f^{-\frac{1}{2}}(\gamma)\end{pmatrix},\\
&=\overline{\rho}_f^{cd+b}(\gamma),\\
&=\overline{\rho}_f^{d'}(\gamma).
\end{align*}
Therefore, $d$ and $d'$ give rise to conjugate representations.
Combining all of the above facts shows the following lemma, which is implicit in \cite{HeusenerPorti}.

\begin{lem}\label{lem:nonabeliancocycle}
Let $f:\pi_1(M)\to\bbC^*$ be a non-trivial homomorphism and define $\overline{\rho}_f$ as in Equation \eqref{eqn:abrep}, then there exists a non-abelian reducible representation $\overline{\rho}_f^d:\pi_1(M)\to\PSL_2(\bbC)$ of the form in Equation \eqref{eqn:nonabrep} such that $\chi_{\overline{\rho}_f}=\chi_{\overline{\rho}_f^d}$ if and only if $H^1(\pi_1(M),\bbC_f)>0$.
Moreover, up to conjugation the representation $\overline{\rho}_f^d$ depends only on the $\bbC$-span of the class of $d$ in $H^1(\pi_1(M),\bbC_f)$.
\end{lem}

\noindent It is this latter condition that Heusener--Porti analyze.
In order to state their results we must define a family of invariants of $\pi_1(M)$ which they call the Alexander invariants.

We again remind the reader that we are working specifically in the case where $M$ is a hyperbolic once-punctured torus bundle and therefore we can choose an isomorphism $H_1(M,\bbZ)\cong\bbZ\oplus F$, where $F$ is a finite group and $\bbZ$ is generated by the class of the stable letter $t$.
Let $\phi_1:\pi_1(M)\to\bbZ$ denote the projection onto the first factor, which takes an element $\gamma\in\pi_1(M)$ to $t^{s_\gamma}$, where $s_\gamma$ is the exponent sum of the letters $t$, $t^{-1}$ in $\gamma$.
Similarly, let $\phi_2:\pi_1(M)\to F$ denote projection onto the second factor.
Then given any fixed homomorphism $\sigma:F\to \Un(1)$, we get an induced representation
\begin{align*}
\phi_\sigma:\pi_1(M)&\to \bbC[z,z^{-1}]^*,\\
\gamma&\mapsto \sigma(\phi_2(\gamma))t^{\phi_1(\gamma)}
\end{align*}
where we are considering $\Un(1)$ embedded in $\mathbb{C}^*$ in the standard way.
Hence we may define the twisted chain complex $C_*(\widetilde{M},\bbZ)\otimes_{\pi_1(M)}\bbC[z,z^{-1}]$, where $\widetilde{M}$ denotes the universal cover of $M$, and taking homology of this complex gives the twisted homology groups $H^{\phi_\sigma}_*(M)$.
As $\bbC[z,z^{-1}]$ is a PID, finite generation of $H^{\phi_\sigma}_1(M)$ yields a decomposition 
$$H^{\phi_\sigma}_1(M)=\bigoplus_{i=0}^s \frac{\bbC[z,z^{-1}]}{r_i\bbC[z,z^{-1}]},$$
where $r_i\in\bbC[z,z^{-1}]$ and $r_{i+1}\mid r_i$ for each $i$.
Using this, we define the \emph{$m^{th}$ Alexander invariant} by the formula $\Delta_m^{\phi_\sigma}(z)=r_mr_{m+1}\dots r_s$.
The invariant $\Delta_m^{\phi_\sigma}$ is well defined up to a unit in $\bbC[z,z^{-1}]$, that is, well defined up to multiplication by $at^j$ for $a\in\bbC^*$ and $j\in \bbZ$.
To avoid having to constantly remind the reader about the fact that $\Delta_m^{\phi_\sigma}$ is only defined up to a unit, in the sequel we will always consider $\Delta_m^{\phi_\sigma}$ to be normalized so that either $\Delta_m^{\phi_\sigma}$ is $1$ (which occurs when $\Delta_m^{\phi_\sigma}$ is itself a unit) or so that $\Delta_m^{\phi_\sigma}$ is a monic polynomial which does not have $0$ as a root.
There is a unique such choice of representative for $\Delta_m^{\phi_\sigma}$.

Despite the formulation of the definition above, there is a standard way to compute $\Delta_m^{\phi_\sigma}$ using Fox derivatives.
Let $\Pi:F_3\to \pi_1(M)$ be the surjective map from the free group on 3 letters given by $x_1\mapsto x$, $x_2\mapsto y$ and $x_3\mapsto t$, where we rewrite the presentation for $\pi_1(M)$ as $\pi_1(M)=\langle x,y,t\mid R_1, R_2\rangle$ with
$$R_1=txt^{-1}(\phi(x))^{-1},\quad R_2=tyt^{-1}(\phi(y))^{-1}.$$
Using $\frac{\partial R_j}{\partial x_i}$ to denote the Fox derivative with respect to $x_i$, we obtain a matrix
\begin{equation}\label{eqn:jac}
J^{\phi_\sigma}=\begin{pmatrix}
\phi_\sigma\left(\Pi\left(\frac{\partial R_1}{\partial x_1}\right)\right)&\phi_\sigma\left(\Pi\left(\frac{\partial R_1}{\partial x_2}\right)\right)&\phi_\sigma\left(\Pi\left(\frac{\partial R_1}{\partial x_3}\right)\right)\\
\phi_\sigma\left(\Pi\left(\frac{\partial R_2}{\partial x_1}\right)\right)&\phi_\sigma\left(\Pi\left(\frac{\partial R_2}{\partial x_2}\right)\right)&\phi_\sigma\left(\Pi\left(\frac{\partial R_2}{\partial x_3}\right)\right)\\
\end{pmatrix},\end{equation}
with entries in $\bbC[z]$ from which \cite[\S 3]{HeusenerPorti} shows that $\Delta_m^{\phi_\sigma}(z)$ is precisely the greatest common divisor of the $(2-m)$-minors of $J^{\phi_\sigma}$.

\begin{rem}\label{rem:fieldindep}
This definition was written for $\bbC$ to follow the work of \cite{HeusenerPorti}.
However it holds more generally when $\bbC$ is replaced with any field $\bbK$, provided that we replace the $\sigma$ from above by a homomorphism $\sigma:F\to \bbK^*$.
Indeed, $\bbK[z,z^{-1}]$ is still a PID and $H^{\phi_\sigma}_1(M)$ is still a finitely generated torsion module.
\end{rem}

\begin{ex}
As in Example \ref{ex:tunnelone}, let $M_j$ be the once-punctured torus bundle with monodromy $RL^j$ so that
$$\pi_1(M_j)=\langle x,y,t\mid txt^{-1}=x(yx)^j,~tyt^{-1}=yx\rangle.$$
It follows that $F=\mathbb{Z}/j\mathbb{Z}$ with notation from Equation \eqref{eqn:homology} and every $\sigma_i:F\to\Un(1)$ is given by $\sigma_i(x)=1$ and $\sigma_i(y)=\eta_j^i$, where $\eta_j$ is a primitive $j^{th}$ root of unity and $0\le i<j$.
Therefore
\begin{align*}
\phi_{\sigma_i}:\pi_1(M_j)&\to\bbC[z,z^{-1}],\\
x&\mapsto 1,\\
y&\mapsto \eta_j^i,\\
t&\mapsto z,
\end{align*}
and one computes from Equation \eqref{eqn:jac} that
$$J^{\phi_{\sigma_i}}=\begin{pmatrix}
z-(1+\eta_j^{-i}+\dots+\eta_j^{-ji})&-(\eta_j^{-i}+\eta_j^{-2i}+\dots+\eta_j^{-ji})&0\\
-\eta_j^i&z-1&1-\eta_j^i
\end{pmatrix}.$$
Evaluating this when $i=0$ gives $\Delta_0^{\phi_{\sigma_0}}(z)=z^2-(j+2)z+1$ and evaluating when $i\neq 0$ gives $\Delta_0^{\phi_{\sigma_i}}(z)=1$ under our normalization.
\end{ex}

The observant reader will notice that $\Delta_0^{\phi_{\sigma_0}}(z)$ in the previous example is precisely the characteristic polynomial of the monodromy matrix $\Phi$.
We now prove this in general and also record a few other properties of the $\Delta_m^{\phi_\sigma}$ which will be useful later.
Recall in what follows that we are still considering a specific normalization of $\Delta_m^{\phi_\sigma}$ and not their values up to multiplication by a unit.

\begin{lem}\label{lem:foxderiv}
Let $M$ be a hyperbolic once-punctured torus bundle with homology $H_1(M,\mathbb{Z})=\bbZ\oplus F$ and with associated monodromy $\Phi$ as in Equation \eqref{eqn:monmatrix}.
Fix some $\sigma:F\to \Un(1)$ and recall that $\phi_2$ is the projection to the second factor of $H_1(M,\bbZ)$.
Then the following hold:
\begin{enumerate}
\item $\phi_2(\sigma(\langle x,y\rangle))=1$ if and only if $\Delta_0^{\phi_\sigma}(z)$ is the characteristic polynomial of $\Phi$.
\item $\phi_2(\sigma(\langle x,y\rangle))$ is non-trivial if and only if $\Delta_0^{\phi_\sigma}(z)$ is a polynomial of degree at most $1$ in $z$.
\item $\Delta_1^{\phi_\sigma}(z)$ is $1$.
\end{enumerate}
\end{lem}
\begin{proof}

Following the computation in Equation \eqref{eqn:jac} and using basic properties of Fox derivatives, we see that $J^{\phi_\sigma}$ is given by
$$\begin{pmatrix}
z+\phi_\sigma(x)\phi_\sigma\left(\Pi\left(\frac{\partial (\phi(x))^{-1}}{\partial x_1}\right)\right)&\phi_\sigma(x)\phi_\sigma\left(\Pi\left(\frac{\partial (\phi(x))^{-1}}{\partial x_2}\right)\right)&1-\phi_\sigma(x)\\
\phi_\sigma(y)\phi_\sigma\left(\Pi\left(\frac{\partial (\phi(y))^{-1}}{\partial x_1}\right)\right)&z+\phi_\sigma(y)\phi_\sigma\left(\Pi\left(\frac{\partial (\phi(y))^{-1}}{\partial x_2}\right)\right)&1-\phi_\sigma(y)\\
\end{pmatrix},$$
where in the above we are using that $\bbC[z,z^{-1}]$ is abelian and hence $\phi_\sigma(txt^{-1})=\phi_\sigma(x)$, $\phi_\sigma(tyt^{-1})=\phi_\sigma(y)$.
Notice that from this it is straightforward to see that $\Delta_0^{\phi_\sigma}(z)$ is a polynomial of degree at most $2$ and a necessary condition for its degree to equal $2$ is $1-\phi_\sigma(x)=1-\phi_\sigma(y)=0$, which is to say that $\phi_2(\sigma(\langle x,y\rangle))=1$.
We now exhibit (1)--(3).

For (1), suppose that $\phi_2(\langle x,y\rangle)=1$, so that $J^{\phi_\sigma}$ simplifies to
\begin{multline}\label{eqn:foxderivcomp}
\begin{pmatrix}
z-\phi_\sigma(\phi(x))\phi_\sigma\left(\Pi\left(\frac{\partial (\phi(x))}{\partial x_1}\right)\right)&-\phi_\sigma(\phi(x))^{-1}\phi_\sigma\left(\Pi\left(\frac{\partial (\phi(x))}{\partial x_2}\right)\right)&0\\
-\phi_\sigma(\phi(y))^{-1}\phi_\sigma\left(\Pi\left(\frac{\partial (\phi(y))}{\partial x_1}\right)\right)&z-\phi_\sigma(\phi(y))\phi_\sigma\left(\Pi\left(\frac{\partial (\phi(y))}{\partial x_2}\right)\right)&0\\
\end{pmatrix}\\
=\begin{pmatrix}
z-\phi_\sigma\left(\Pi\left(\frac{\partial (\phi(x))}{\partial x_1}\right)\right)&-\phi_\sigma\left(\Pi\left(\frac{\partial (\phi(x))}{\partial x_2}\right)\right)&0\\
-\phi_\sigma\left(\Pi\left(\frac{\partial (\phi(y))}{\partial x_1}\right)\right)&z-\phi_\sigma\left(\Pi\left(\frac{\partial (\phi(y))}{\partial x_2}\right)\right)&0\\
\end{pmatrix}.
\end{multline}
The conclusion then follows from the fact that, upon evaluating $\phi_\sigma$, each of the leftover derivatives in the above equation are simply the exponent sum in the corresponding variable (see the next paragraph) and hence $J^{\phi_\sigma}$ becomes
\begin{equation}\label{eqn:foxmatrix}
\begin{pmatrix}
z-a&-b&0\\
-c&z-d&0\\
\end{pmatrix},
\end{equation}
where $a$, $b$, $c$, $d$ are as in Equation \eqref{eqn:monmatrix}.
Therefore $\Delta_0^{\phi_\sigma}(z)=z^2-\Tr(\Phi)z+1$, the characteristic polynomial of $\Phi$, as required.

For the details of the calculation indicated above, first note that given some word $v\in F_2$, we have that
\begin{align*}
\frac{\partial}{\partial x} (x^\alpha y^\beta v)&=\frac{\partial}{\partial x}(x^\alpha y^\beta)+x^\alpha y^\beta\frac{\partial}{\partial x}v=\frac{\partial}{\partial x}x^\alpha+x^\alpha y^\beta\frac{\partial}{\partial x}v,\\
\frac{\partial}{\partial y} (y^\beta x^\alpha v)&=\frac{\partial}{\partial y}(y^\beta x^\alpha)+y^\beta x^\alpha\frac{\partial}{\partial y}v=\frac{\partial}{\partial y}y^\beta+y^\beta x^\alpha\frac{\partial}{\partial y}v.
\end{align*}
Now suppose that we are given a word $w=x^{a_1}y^{b_1}\dots x^{a_n}y^{b_n}$ where $a_i,b_i\in\mathbb{Z}$ are all non-zero except possibly $a_1,b_n$.
Then using the previous computation, it follows that
\begin{align*}
\frac{\partial }{\partial x}w&=\frac{\partial }{\partial x} x^{a_1}+x^{a_1}y^{b_1}\frac{\partial }{\partial x}(x^{a_2}y^{b_2}\dots x^{a_n}y^{b_n}),\\
&=\frac{\partial }{\partial x} x^{a_1}+x^{a_1}y^{b_1}\left(\frac{\partial }{\partial x}x^{a_2}+x^{a_2}y^{b_2}\left(\dots+x^{a_{n-1}}y^{b_{n-1}}\frac{\partial }{\partial x}x^{a_n}\right)\right),
\end{align*}
and similarly 
\begin{align*}
\frac{\partial }{\partial y}w&=x^{a_1}\frac{\partial }{\partial y}\left( y^{b_1}x^{a_2}\dots x^{a_n} y^{b_n}\right),\\
&=x^{a_1}\left[\frac{\partial }{\partial y} y^{b_1}+y^{b_1}x^{a_2}\frac{\partial }{\partial y}(y^{b_2}\dots x^{a_n}y^{b_n})\right],\\
&=x^{a_1}\left[\frac{\partial }{\partial y} y^{b_1}+y^{b_1}x^{a_2}\left(\frac{\partial }{\partial y}y^{b_2}+y^{b_2}x^{a_3}\left(\dots+y^{b_{n-1}}x^{a_{n}}\frac{\partial }{\partial y}y^{b_n}\right)\right)\right].
\end{align*}
Moreover one readily checks that
$$\frac{\partial }{\partial x} x^\alpha=\begin{cases}
1+x+\dots+x^{\alpha-1},&\alpha>0\\
-x^{-1}-x^{-2}-\dots-x^{\alpha},&\alpha<0\end{cases},$$
which notably has $\alpha$ terms in each sum.
The similar computation holds for $y$.
In particular, if we evaluate $\frac{\partial w}{\partial x}, \frac{\partial w}{\partial y}$ at $x=y=1$, it follows that
\begin{equation}\label{eqn:foxeval}
\left.\frac{\partial w}{\partial x}\right\vert_{x=1,y=1}=\sum_{i=1}^n\frac{\partial}{\partial x} x^{a_i}=\sum_{i=1}^na_i,\quad\left.\frac{\partial w}{\partial y}\right\vert_{x=1,y=1}=\sum_{i=1}^n\frac{\partial}{\partial y} y^{b_i}=\sum_{i=1}^nb_i,
\end{equation}
which is precisely the exponent sum of $x$ or $y$ in $w$.
Returning to the situation at hand where we evaluate Equation \eqref{eqn:foxderivcomp} at $\phi_\sigma$ when $\phi_2(\langle x,y\rangle )=1$, it follows that $\phi_\sigma(\phi(x))=\phi_\sigma(\phi(y))=1$ and from the previous calculation (specifically Equation \eqref{eqn:foxeval}) it follows that the terms $\Pi\left(\frac{\partial (\phi(x))}{\partial x_i}\right)$, $\Pi\left(\frac{\partial (\phi(y))}{\partial x_i}\right)$ are precisely the exponent sums of $\phi(x)$, $\phi(y)$ with respect to $x$ when $i=1$ and with respect to $y$ when $i=2$.
From this one deduces the calculation in Equation \eqref{eqn:foxmatrix}.

For (2), note that if $\sigma$ is such that $\phi_2(\sigma(\langle x,y\rangle))\not\equiv1$ then at least one of $\phi_\sigma(x)$, $\phi_\sigma(y)$ is not equal to $1$.
In particular, at least one of $1-\phi_\sigma(x)$, $1-\phi_\sigma(y)$ is non-zero and therefore two of the three $2$-dimensional minors of $J^{\phi_\sigma}$ are polynomials of degree at most $1$ in $z$.
As $\Delta_0^{\phi_\sigma}(z)$ is the greatest common divisor of these minors, it must also have this property.

For (3), again we notice that if $\phi_2(\sigma(\langle x,y\rangle))\not\equiv1$, then at least one of $1-\phi_\sigma(x)$, $1-\phi_\sigma(y)$ is non-zero.
In this case, $\Delta_1^{\phi_\sigma}(z)$ is the greatest common divisor of all of the entries of $J^{\phi_\sigma}$ and it follows that $\Delta_1^{\phi_\sigma}(z)$ is in $\mathbb{C}^*$.
If $\phi_2(\sigma(\langle x,y\rangle))\equiv 1$, then the computation from part (1) shows the corresponding result since, when $M$ is a hyperbolic once punctured torus bundle, at least one of $b$, $c$ must be non-zero.
In either case $\Delta_1^{\phi_\sigma}(z)$ is in $\bbC^*$ and hence by our normalization above, we have that $\Delta_1^{\phi_\sigma}(z)$ is $1$.
\end{proof}

\begin{defn}
Let $f:\pi_1(M)\to\mathbb{C}^*$ be any homomorphism and let $\sigma=f\vert_F:F\to \Un(1)$.
Then we say that $f$ is a \emph{zero of the $m^{th}$ Alexander invariant of order $r$} if $f(t)\in\mathbb{C}^*$ is a root of order r of $\Delta^{\phi_\sigma}_m(z)$, where $t$ is the stable letter of $\pi_1(M)$.
If $f$ is a zero of order $1$, then we call it a \emph{simple zero}.
\end{defn}

\noindent Note that an immediate corollary of Lemma \ref{lem:foxderiv} is that any zero of $\Delta_0^{\phi_\sigma}(z)$ is simple.

We remark that this definition is independent of choices of isomorphisms that were made previously, as shown in \cite[\S 4]{HeusenerPorti}.
A key lemma for us is the following.

\begin{lem}[\cite{HeusenerPorti}, Lemma 4.6]\label{lem:HP4.6}
With $f$ as above, $H^1(\pi_1(M),\bbC_f)=m+1$ if and only if $f$ is a zero of the $m^{th}$ Alexander invariant and not a zero of the $(m+1)^{st}$ Alexander invariant.
\end{lem}

We point out for the reader that the statement of \cite[Lem 4.6]{HeusenerPorti} has the incorrect indices.
Following their work in the previous sections and the proof of that lemma, one arrives at the statement in Lemma \ref{lem:HP4.6} which is what they use in the remainder of their paper.
With this in mind we get the following immediate corollary.

\begin{cor}
Let $f:\pi_1(M)\to\bbC^*$ be a non-trivial homomorphism and $\overline{\rho}_f$ an abelian representation satisfying Equation \eqref{eqn:abrep}.
Then there exists a non-abelian reducible representation $\overline{\rho}_f^d:\pi_1(M)\to\PSL_2(\bbC)$ such that $\chi_{\overline{\rho}_f}=\chi_{\overline{\rho}_f^d}$ if and only if $f$ is a zero of the $0^{th}$ Alexander invariant.
Moreover, when $f$ is a zero of the $0^{th}$ Alexander invariant, $\overline{\rho}_f^d$ is unique up to conjugation.
\end{cor}
\begin{proof}
Lemma \ref{lem:foxderiv}(3) shows that any such $f$ is never a zero of $\Delta_1^{\phi_\sigma}(z)$ from which the first statement of the lemma follows immediately from Lemma \ref{lem:HP4.6}.
Moreover, since $H^1(\pi_1(M),\bbC_f)=1$ in this case, the second statement follows from Lemma \ref{lem:nonabeliancocycle}.
\end{proof}

We end this subsection by summarizing the main result of \cite{HeusenerPorti}, which is to show that any non-abelian reducible representation $\overline{\rho}^d_f$ is a reduced point of the $\PSL_2(\bbC)$-representation scheme and its associated character is a reduced point of the $\PSL_2(\bbC)$-character scheme.
For this, we define the $\PSL_2(\bbC)$-representation and character schemes (respectively) via
$$\overline{\frak{R}}(M)_\bbC=\frak{hom}(\pi_1(M),\PSL_2)_{\bbC},\quad \overline{\frak{X}}(M)_\bbC=\overline{\frak{R}}(M)_\bbC/\!/\PSL_2(\bbC),$$
where again the latter denotes the GIT quotient.
These are the analogues of the schemes defined in Equation \eqref{eqn:schemetime}.
The main results of Heusener--Porti can then be combined into the following statement.

\begin{thm}[\cite{HeusenerPorti}, Theorems 1.2 and 1.3]\label{thm:HPmain}
If $f$ is a simple zero of the Alexander invariant, then $\overline{\rho}_f$ is contained in exactly two irreducible components of $\overline{\frak{R}}(M)_\bbC$, one component containing irreducible representations and the other containing only abelian representations. Moreover, the representation $\overline{\rho}_f$ is a regular point of each of these irreducible components.

If $\overline{\chi_{\overline{\rho}_f}}$ denotes the associated character in $\overline{\frak{X}}(M)_\bbC$, then $\overline{\chi_{\overline{\rho}_f}}$ is contained in exactly two irreducible curve components of $\overline{\frak{X}}(M)_\bbC$ which are the quotients of those above.
Moreover, $\overline{\chi_{\overline{\rho}_f}}$, is a regular point of each of these components.
\end{thm}

%~~~~~~~~~~~~~~~~~~~~~~~~~~~~~~~~~~~~~~~~~~~~~~~~~~~~~~~~~~~
%~~~~~~~~~~~~~~~~~~~~~~~~~~~~~~~~~~~~~~~~~~~~~~~~~~~~~~~~~~~
%~~~~~~~~~~~~~~~~~~~~~~~~~~~~~~~~~~~~~~~~~~~~~~~~~~~~~~~~~~~

\subsection{Lifting representations to $\SL_2(\bbC)$}\label{sec:replift}

The discussion in Section \ref{sec:HPrecap} was focused strictly on representations into $\PSL_2(\bbC)$.
Therefore in order to transfer these results to $\frak{X}(M)_\bbC$, we will need to lift the representations $\overline{\rho}_f$, $\overline{\rho}_f^d$ to $\SL_2(\bbC)$.
First we make some general comments, then go on to discuss our specific setting.

Let $\overline{V}\subset \overline{\frak{R}}(M)_\bbC$ be the subset of representations into $\PSL_2(\bbC)$ which admit a lift to $\SL_2(\bbC)$.
By \cite[Thm 4.1]{Culler}, $\overline{V}$ is a union of irreducible components of $\overline{\frak{R}}(M)_\bbC$ and hence is a union of Zariski closed subsets.
Moreover $\overline{V}$ is conjugation invariant and hence projects to a closed subset $\overline{W}$ of the $\PSL_2(\bbC)$-character scheme $\overline{\frak{X}}(M)_\bbC$.
There is a natural action of $H^1(\pi_1(M),\mu_2)$ on the $\SL_2(\bbC)$-representation scheme $\frak{R}(M)_\bbC$.
Indeed for each $\epsilon\in H^1(\pi_1(M),\mu_2)$ and $\rho\in \frak{R}(M)_\bbC$, $\epsilon\cdot \rho (\gamma)=\epsilon(\gamma)\rho(\gamma)$ gives another point in $\frak{R}(M)_\bbC$.
This induces an action on $\frak{X}(M)_\bbC$ and it follows from \cite[Prop 4.2]{HeusenerPorti2} that
\begin{equation}\label{eqn:GITquot}
\overline{W}=\frak{X}(M)_\bbC\slash\!\slash H^1(\pi_1(M),\mu_2)\subset\overline{\frak{X}}_\bbC.
\end{equation}
Moreover the map from $\frak{X}(M)_\bbC\to \overline{W}$ is a branched covering with branch locus contained in the set of $\Ad$-reducible representations and is a regular $|H^1(\pi_1(M),\mu_2)|$ cover away from this locus \cite[Pg 513--514]{MorganShalen3}.
Recall that $\Ad$-reducible representations are either reducible in the typical sense or fix an unoriented geodesic under the natural action of $\SL_2(\bbC)$ on $\bbH^3$ (see \cite[Lem 3.12]{HeusenerPorti2}).

We caution the reader that \cite[Rem 4.3]{HeusenerPorti2} describing the branch locus in that same paper is not correct, so we briefly clarify the point following the discussion in \cite{MorganShalen3}.
Given $\epsilon\in H^1(\pi_1(M),\mu_2)$, let $S(\epsilon)$ denote the subset of $\frak{X}(M)_\bbC$ of characters $\chi$ which satisfy the property that $\chi(\gamma)=0$ whenever $\epsilon(\gamma)=-1$.
Then $S(\epsilon)$ is a Zariski closed subset of $\frak{X}(M)_\bbC$ and, as $H^1(\pi_1(M),\mu_2)$ is finite (in fact, is at most order $8$ in our setting), it follows that $\cup_{\epsilon\neq 1} S(\epsilon)$ is also Zariski closed, where we use $\epsilon\neq 1$ to mean that $\epsilon$ is not the trivial map.
It is straightforward to check that the map $\frak{X}(M)_\bbC\to \overline{W}$ branches over $\cup_{\epsilon\neq 1} S(\epsilon)$ and is a regular cover on its complement.
In fact in the cases under consideration, where we restrict this map to components of irreducible representations, $\cup_{\epsilon\neq 1} S(\epsilon)$ is clearly proper and therefore a finite union of points corresponding to characters of non-abelian reducible representations which have the requisite trace condition.

Now fix a Type A, B, or C representation $\rho\in \frak{R}(M)_\bbC$ and let $\sigma_\rho:\pi_1(M)\to\bbC^*$ denote the homomorphism from Lemma \ref{lem:RootsOf1}, which maps an element $\gamma$ to the eigenvalue of $\rho(\gamma)$ in the upper left entry.
From the work done in Section \ref{sec:nonabclassification}, we have seen that in order for a representation $\rho:\pi_1(M)\to\SL_2(\bbC)$ to be non-abelian reducible, one of $x$, $y$, or $t$, has eigenvalues not equal to $\pm 1$ and hence $\sigma_\rho(\gamma)\not\equiv 1$.
Denote by $\overline{\rho}:\pi_1(M)\to\PSL_2(\bbC)$ the representation formed by the composition of $\rho$ with the natural projection $\SL_2(\bbC)\to\PSL_2(\bbC)$.
Taking $f=\sigma_\rho^2$ and $f^{\pm\frac{1}{2}}=\sigma_\rho^{\pm 1}$ in the language of Section \ref{sec:HPrecap}, we see that any non-abelian reducible representation $\rho$ of Type A, B, or C gives rise to a representation $\overline{\rho}:\pi_1(M)\to\PSL_2(\bbC)$ which can be written as $\overline{\rho}^d_{f}$ from Equation \eqref{eqn:nonabrep} for some $d\in H^1(\pi_1(M),\bbC_f)$.
Conversely, any representation $\overline{\rho}^d_f$ which lifts to $\SL_2(\bbC)$ factors through a Type A, B, or C representation and hence is of the form $\overline{\rho}^d_{f}$ for some $\sigma:\pi_1(M)\to\bbC^*$.
Moreover, since $H^1(\pi_1(M),\bbC_f)=1$, there exists precisely one conjugacy class of each of these representations.

Using the reducedness from Theorem \ref{thm:HPmain}, the description of Equation \eqref{eqn:GITquot}, and the fact that $\frak{X}(M)_\bbC\to \overline{W}$ is a branched covering, we obtain the following useful theorem.

\begin{thm}\label{thm:HPsl2c}
Let $\rho$ be a fixed Type A, B, or C representation with associated character $\chi_\rho$ in the $\SL_2(\bbC)$-character scheme $\frak{X}(M)_\bbC$.
If $\chi_\rho$ lies on an irreducible curve $\frak{C}_\bbC$ which contains the character of an irreducible representation, then $\chi_\rho$ is a reduced point of $\frak{C}_\bbC$.
Moreover, there is an irreducible component $\frak{D}_\bbC\subset\frak{R}(M)_\bbC$ of dimension $4$ which projects to $\frak{C}_\bbC$ for which $\rho\in \frak{D}_\bbC$ is a reduced point.
\end{thm}

Recall that the statement that $\chi_\rho$ is regular is the statement that the second inequality in Equation \eqref{eqn:tangentineq} is actually an equality.
Hence $\mathcal{O}_{\frak{C}_\bbC,\chi_\rho}$ is a regular local ring of dimension $1$, that is, a discrete valuation ring.

%~~~~~~~~~~~~~~~~~~~~~~~~~~~~~~~~~~~~~~~~~~~~~~~~~~~~~~~~~~~
%~~~~~~~~~~~~~~~~~~~~~~~~~~~~~~~~~~~~~~~~~~~~~~~~~~~~~~~~~~~
%~~~~~~~~~~~~~~~~~~~~~~~~~~~~~~~~~~~~~~~~~~~~~~~~~~~~~~~~~~~

\subsection{The tame symbol at Type A, B, and C representations.}\label{sec:localcoord}

To prove Theorem \ref{thm:main2} we will need some additional information about the structure of $\frak{X}(M)_\bbC$ at Type A, B, and C representations.
For this, let $\rho$ be a representation of either Type A, B, or C contained in an irreducible curve $\frak{C}_\bbC$ which contains the character of an irreducible representation.
Then, as in Theorem \ref{thm:HPsl2c}, there is a component $\frak{D}_\bbC\subset \frak{R}(M)_\bbC$ such that $\rho\in \frak{D}_\bbC$ and such that $\frak{D}_\bbC$ is regular at $\rho$ and of dimension $4$.

As discussed in Section \ref{sec:backrep}, the Zariski tangent space $T_\rho\frak{D}_\bbC$ can be described using the space of 1-cocycles $Z^1(\pi_1(M),\frak{sl}_2(\bbC)^\rho)$ and we would first like to produce an explicit basis for the latter.
Fix $f:\pi_1(M)\to\bbC^*$ and write $\bbC_-$ to denote $\bbC$ with the structure of a $\pi_1(M)$-module via $\gamma\cdot z=f(\gamma)^{-1}z$, where we remark that this is a well defined left action since the image of $f$ lies in $\bbC^*$ which is abelian.
Write $\frak{b}_+$ for the Borel subalgebra of upper triangular matrices in $\frak{sl}_2(\mathbb{C})$, that is, elements of $\frak{b}_+$ are matrices in the Lie algebra $\frak{sl}_2(\bbC)$ with a $0$ in the lower left entry.
Then we need the following result of Heusener--Porti.

\begin{thm}[\cite{HeusenerPorti}, Corollary 5.4]\label{cor:HPcor54}
If $f:\pi_1(M)\to\bbC^*$ is a simple zero of the Alexander invariant, then the projection $\mathrm{pr}:\frak{sl}_2(\bbC)\to\frak{sl}_2(\bbC)/\frak{b}_+\cong \bbC_-$ induces an isomorphism
$$H^1(\pi_1(M),\frak{sl}_2(\bbC)^\rho)\cong H^1(\pi_1(M),\bbC_-)\cong \bbC.$$
Moreover, $H^1(\pi_1(M),\bbC_-)$ is generated by any $1$-cocycle $d_-(\gamma)\in Z^1(\pi_1(M),\bbC_-)$ such that 
\begin{equation}\label{eqn:rhominus}
\rho_-(\gamma)=\begin{pmatrix}
1&0\\
d_-(\gamma)&1
\end{pmatrix}\rho_f(\gamma)=\pm\begin{pmatrix}
1&0\\
d_-(\gamma)&1
\end{pmatrix}\begin{pmatrix}
f^{\frac{1}{2}}(\gamma)&0\\
0&f^{-\frac{1}{2}}(\gamma)\end{pmatrix},
\end{equation}
is a non-abelian reducible representation. 
\end{thm}

By the same computations as in Section \ref{sec:HPrecap}, the condition that $\rho_-(\gamma)$ is a non-abelian reducible representation is equivalent to the condition that $d_-(\gamma)$ is non-trivial in $H^1(\pi_1(M),\bbC_-)$.
In the next proposition, we use $b_\rho$ and $c_\rho$ to denote the coordinate functions $b_\rho,c_\rho:\pi_1(M)\to\bbC$ defined by the equation
$$\rho(\gamma)=\begin{pmatrix}
a_\rho(\gamma)&b_\rho(\gamma)\\
c_\rho(\gamma)&d_\rho(\gamma)\end{pmatrix}.$$
We then have the following, which is the analogue of \cite[Lem 5.8]{HPS} attuned to our setting.

\begin{prop}\label{prop:localinverse}
Given a fixed representation $\rho$ of Type A, B, or C, then there are fixed explicit and distinct choices $h_1, h_2\in\pi_1(M)$ such that the map
\begin{align*}
\tau:\frak{R}(M)_\bbC&\to \bbC^4,\\
\rho&\mapsto(b_\rho(h_1),b_\rho(h_2),c_\rho(h_1),c_\rho(h_2))\end{align*}
is invertible in a sufficiently small analytic neighborhood of $\rho$ in $\frak{D}_\bbC$.
Specifically, $h_1$, $h_2$ are chosen as follows
\begin{enumerate}
\item if $\rho$ is of Type A, then $h_1=t$ and $h_2$ is the product of $t$ and any choice of $x$, $y$ which is not mapped to the identity via $\rho$,
\item if $\rho$ is of Type B, then $h_1$ is any choice of $x$, $y$ which does not have eigenvalues $\pm 1$ and $h_2=h_1t$,
\item if $\rho$ is of Type C, then $h_1$ is as in Type B and $h_2$ is the remaining choice of $x$, $y$ unless that choice has eigenvalues $\pm 1$ in which case $h_2=xy$.
\end{enumerate}
\end{prop}
\begin{proof}
We will argue by producing an explicit basis $\{v_1,v_2,v_3,v_4\}$ of the $4$ dimensional space $Z^1(\pi_1(M),\frak{sl}_2(\bbC)^\rho)$ and showing that, under the identification of $Z^1(\pi_1(M),\frak{sl}_2(\bbC)^\rho)$ with the Zariski tangent space $T_\rho \frak{D}_\bbC$, the Jacobian is non-zero.
The proposition will then follow from the implicit function theorem.

First assume that $\rho$ is a Type A representation, where we recall that
$$\rho(x) = (-1)^{\ep_1} \begin{pmatrix} 1 & r \\ 0 & 1 \end{pmatrix},\quad\rho(y) = (-1)^{\ep_2} \begin{pmatrix} 1 & s \\ 0 & 1 \end{pmatrix},\quad\rho(t) = \begin{pmatrix} w & 0 \\ 0 & w^{-1} \end{pmatrix},$$
for $w\neq \pm 1$ and $(r,s)\neq (0,0)$.
Let $h_1=t$, $g$ be a fixed choice of $x$, $y$ such that $b_{\rho}(g)\neq 0$, which is possible since one of $r$, $s$ is non-zero, and $h_2=tg$.
Note that then 
$$\rho(h_2)=(-1)^\epsilon\begin{pmatrix} w&z_0\\0&w^{-1}\end{pmatrix},$$
for some $z_0\in\bbC^*$ and some $\epsilon\in\{0,1\}$.

The $3$ dimensional subspace of 1-coboundaries $B^1(\pi_1(M),\frak{sl}_2(\bbC)^\rho)$ therefore has a basis $\{v_1,v_2,v_3\}$ given by
$$v_i(\gamma)=e_i-\Ad_{\rho(\gamma)}e_i,\quad \gamma\in\pi_1(M),$$
where the $e_i$ are defined using the following standard basis
$$e_1=\begin{pmatrix}
0&1\\
0&0
\end{pmatrix},\quad e_2=\begin{pmatrix}
1&0\\
0&-1
\end{pmatrix},\quad e_3=\begin{pmatrix}
0&0\\
1&0
\end{pmatrix},$$
of $\frak{sl}_2(\bbC)$.
By computing the values of each $v_i$ on the generators $x,y,t$, it is straightforward to check that these functions are linearly independent and hence span all of $B^1(\pi_1(M),\frak{sl}_2(\bbC)^\rho)$.

To complete $\{v_1,v_2,v_3\}$ to a basis of $Z^1(\pi_1(M),\frak{sl}_2(\bbC)^\rho)$, it suffices to find a $1$-cocycle $v_4$ which is non-trivial in $H^1(\pi_1(M),\frak{sl}_2(\bbC)^\rho)$.
For this, we use Theorem \ref{cor:HPcor54}.
Recall that any non-abelian reducible representation is conjugate to $\rho=\rho_f^d$ for some cocycle $d\in Z^1(\pi_1(M),\bbC_f)$.
As remarked in Section \ref{sec:HPrecap}, any choice of $d\in Z^1(\pi_1(M),\bbC_f)$ which is non-trivial in $H^1(\pi_1(M),\bbC_f)$ will give rise to the same conjugacy class of representation $\rho$, since $H^1(\pi_1(M),\bbC_f)$ is $1$-dimensional.

All of the computations from Section \ref{sec:HPrecap} hold similarly for $H^1(\pi_1(M),\bbC_-)$ and therefore, by composing with a suitable conjugation, we can assume that $\rho_-$ from Equation \eqref{eqn:rhominus} is simply the transpose inverse of the representation $\rho$.
In particular, using the definition of $d_-(-)$ from Equation \eqref{eqn:rhominus}, we see that $d_-(h_2)\neq 0$ for our chosen $h_2$.
As our transpose inverse representation $\rho_-$ is non-abelian by construction, the class $[d_-]\in H^1(\pi_1(M),\bbC_-)$ is non-trivial and
using the isomorphism
\begin{align*}
\mathrm{pr}^*:H^1(\pi_1(M),\frak{sl}_2(\bbC)^\rho)&\to H^1(\pi_1(M),\bbC_-),\\
u&\mapsto \mathrm{pr}\circ u
\end{align*}
from Theorem \ref{cor:HPcor54}, we conclude that $(\mathrm{pr}^*)^{-1}([d_-])\in H^1(\pi_1(M),\frak{sl}_2(\bbC)^\rho)$ is a non-trivial class.
Since $d_-(h_2)\neq 0$, it follows that we can choose some $v_4\in Z^1(\pi_1(M),\frak{sl}_2(\bbC)^\rho)$ representing $(\mathrm{pr}^*)^{-1}([d_-])$ such that
$$v_4(h_2)=\begin{pmatrix}
*&*\\
b_0&*\end{pmatrix},$$
for some $b_0\in\bbC^*$.
Therefore $\{v_1,\dots,v_4\}$ forms a basis of $Z^1(\pi_1(M),\frak{sl}_2(\bbC)^\rho)$.

Under the identification of $Z^1(\pi_1(M),\frak{sl}_2(\bbC)^\rho)$ with the analytic tangent space, $T_\rho^{an}\frak{R}(M)_\bbC$, $v_i\cdot \rho\in T_\rho \frak{D}_\bbC$ is the vector tangent to the curve
$$\rho_c(\gamma)=\exp(c v_i)\rho(\gamma)=(1+cv_i(\gamma)+\frac{c^2v_i(\gamma)^2}{2!}+\dots)\rho(\gamma),$$
at $c=0$.
Using this expression one computes, in the basis given by the $v_i$, that the Jacobian of $\tau$ is the matrix
$$\begin{pmatrix}
1-w^2&1-w^2&0&0\\
0&\pm2z_0w&0&0\\
0&z_0^2&1-w^{-2}&1-w^{-2}\\
*&*&*&b_0(1-w^{-2})
\end{pmatrix},$$
which has determinant $\pm 2wb_0z_0(1-w^2)(1-w^{-2})^2$.
Since $b_0,z_0\neq 0$ and $w\neq 0,\pm 1$, we see that $\tau$ is locally invertible at $\rho$.

When $\rho$ is a Type B representation, at least one of $x$ or $y$ has eigenvalues which are roots of unity and not equal to $\pm 1$.
Let $h_1$ be a choice of $x$, $y$ for which this holds and let $h_2=h_1t$.
In particular we see that
$$\rho(h_1)=\begin{pmatrix} \al & 0 \\ 0 & \al^{-1} \end{pmatrix},\quad \rho(h_2)=\begin{pmatrix} \pm\al & \al \\ 0 & \pm\al^{-1} \end{pmatrix},$$
and the above arguments go through with this choice of $h_1$, $h_2$.
Again, we compute that the Jacobian of $\tau$ as the matrix
$$\begin{pmatrix}
1-\alpha^2&1-\al^2&0&0\\
0&\pm2\al^2&0&0\\
0&\al^2&1-\alpha^{-2}&1-\al^{-2}\\
*&*&*&b_0(1-\al^{-2})
\end{pmatrix}.$$
This has determinant $\pm 2b_0\al^2(1-\al^2)(1-\al^{-2})^2$ which is again non-zero since $\al^2\neq\pm 0,1$ and $b_0\neq 0$.

When $\rho$ is a representation of Type C, similar to the previous case, we let $h_1$ be a choice of $x$, $y$ which has eigenvalues which are not equal to $\pm 1$.
Let $g$ be the unique element of $\{x,y\}\setminus \{h_1\}$.
If $\rho(g)$ has eigenvalues $\pm 1$ then setting $h_2=h_1g$, we see that 
$$\rho(h_2)=\begin{pmatrix}\pm\al&\al\\0&\pm \al^{-1}\end{pmatrix},$$
and we can argue identically to the previous paragraph.
Otherwise, we let $h_2=g$ so that
$$\rho(h_1)=\begin{pmatrix} \al & 0 \\ 0 & \al^{-1} \end{pmatrix},\quad \rho(g)=\begin{pmatrix} \beta & 1 \\ 0 & \beta^{-1} \end{pmatrix}.$$
The identical logic shows that the Jacobian of $\tau$ is
$$\begin{pmatrix}
1-\alpha^2&1-\beta^2&0&0\\
0&2\beta&0&0\\
0&1&1-\alpha^{-2}&1-\beta^{-2}\\
*&*&*&b_0(1-\beta^{-2})
\end{pmatrix}.$$
This has determinant $ 2\beta b_0(1-\al^2)(1-\al^{-2})(1-\beta^{-2})$ which is again non-zero since $\al,\beta\neq 0,\pm 1$  and $b_0\neq 0$. 
\end{proof}

\begin{lem}\label{lem:generalcoord}
Let $\rho$ be a representation of Type A, B, or C and let $h_1$, $h_2$ be the corresponding choice as in Proposition \ref{prop:localinverse}.
Suppose the character $\chi_{\rho}$ lies on an irreducible curve $\frak{C}_\bbC\subset \frak{X}(M)_\bbC$ which contains the character of an irreducible representation.
Let $p\in\widetilde{C}$ be such that $\Pi_\bbC(p)=\chi_\rho$, then the local ring $\mathcal{O}_{\widetilde{C}_\bbC,p}$ has uniformizing parameter given by
$$\widetilde\lambda=-\frac{\widetilde{I}_{[h_1,h_2]}-2}{\widetilde{I}_{h_1}^2-4}.$$
\end{lem}
\begin{proof}
Since $\chi_\rho$ is a regular point of $\frak{C}_\bbC$, the map $\Pi_\bbC$ induces an isomorphism between the discrete valuation rings $\mathcal{O}_{\widetilde{C}_\bbC,p}$ and $\mathcal{O}_{\frak{C}_\bbC,\chi_\rho}$.
It will therefore suffice to show that 
$$\lambda=-\frac{I_{[h_1,h_2]}-2}{I_{h_1}^2-4},$$
is a uniformizing parameter for $\mathcal{O}_{\frak{C}_\bbC,\chi_\rho}$.
By construction, for all Types A, B, and C we have that
\begin{equation}\label{eqn:generalrep}
\rho(h_1) = \begin{pmatrix} \al & 0 \\ 0 & \al^{-1} \end{pmatrix},\quad\rho(h_2) = \begin{pmatrix} \beta & z_0 \\ 0 & \beta^{-1} \end{pmatrix},
\end{equation}
where $\al\neq \pm 1$ and $z_0\in\bbC^*$.
Though we will not need it in what follows, we always have that $\beta\notin\{-1,0,1\}$.
From Equation \eqref{eqn:generalrep} it follows that $\ord_p(\lambda)>0$ and hence we may write $\lambda=u \frak{t}^k$ in $\mathcal{O}_{\frak{C}_\bbC,\chi_\rho}$ where $k\in\bbN$, $u$ is a unit, and $\frak{t}$ is a generator of the maximal ideal of $\mathcal{O}_{\frak{C}_\bbC,\chi_\rho}$.
To show that $\lambda$ is a uniformizing parameter is to show that $k=\ord_p(\lambda)=1$.

Using the notation of Proposition \ref{prop:localinverse}, the map
$$\tau(\rho)=(b_\rho(h_1),b_\rho(h_2),c_\rho(h_1),c_\rho(h_2)),$$
is locally invertible at $\rho$, where $\rho$ corresponds to the point $(0,z_0,0,0)\in\bbC^4$.
By the local invertibility of $\tau$, there exists a small disk $\mathbb{D}\subset\mathbb{C}$ centered at the origin, such that
\begin{align*}
\sigma:\mathbb{D}&\to \frak{R}(M)_\bbC,\\
z&\mapsto \tau^{-1}(0,z_0,0,z/z_0)
\end{align*}
is well defined.
Explicitly, $\sigma$ locally gives a 1-complex parameter family of representations $\sigma(z)=\rho_z:\pi_1(M)\to\SL_2(\bbC)$ such that
$$\rho_z(h_1)=\begin{pmatrix}
a_2(z)&0\\
0&d_2(z)\end{pmatrix},\quad\rho_z(h_2)=\begin{pmatrix}
a_1(z)&z_0\\
z/z_0&d_1(z)\end{pmatrix}.$$
Moreover, one can see from the determinant condition that $a_1(z)d_1(z)-z=1$ and $a_2(z)d_2(z)=1$.
In particular, a straightforward calculation shows that the equation
$$(4-\Tr(\rho_z(h_1))^2)z=\Tr([\rho_z(h_1),\rho_z(h_2)])-2,$$
holds for all $z\in\mathbb{D}$.

Composing $\sigma$ with the natural projection $\Upsilon_\bbC:\frak{R}(M)_\bbC\to \frak{X}(M)_\bbC$ defines a map $\overline{\sigma}:\mathbb{D}\to \frak{X}(M)_\bbC$.
Given this and our definition of $\lambda$ above, we have that $\lambda\circ\overline\sigma=\id_\mathbb{D}$.
Recall that $\lambda$ can be written as $u \frak{t}^k$ in the local ring $\mathcal{O}_{\frak{C}_\bbC,\chi_\rho}$.
If $k>1$, then the composition $\lambda\circ\overline\sigma$ could not be identity $\id_\mathbb{D}$ and so we conclude that $k=1$.
Therefore $\lambda$ is a uniformizing parameter of $\mathcal{O}_{\frak{C}_\bbC,\chi_\rho}$, completing the proof.
\end{proof}

The previous lemma produced a uniformizing parameter for the local ring at characters of Type A, B, or C representations lying on $\frak{C}_\bbC$, or equivalently $\widetilde{C}_\bbC$. 
However to compute the tame symbol of $\frak{C}$ over its field of definition $k$, we will need a lemma which allows us to conclude that the function $\widetilde{\lambda}$ from Lemma \ref{lem:generalcoord} is actually a uniformizing parameter for the local ring on $\frak{C}$ and $\widetilde{C}$ over $k$.
In what follows, recall that $\frak{C}_\bbC=\frak{C}\times_k \bbC$ is the basechange of a geometrically irreducible curve $\frak{C}\subset \frak{X}(M)_k$ defined over $k$ and write $\Theta:\frak{C}_\bbC\to \frak{C}$ for the corresponding map of schemes.

\begin{lem}\label{lem:transferparam}
Let $p\in \frak{C}_\bbC$, $q\in \frak{C}$ be regular closed points such that $\Theta(p)=q$.
Moreover let $\lambda\in k(\frak{C})$ be a rational function such that $\lambda$ is a uniformizing parameter of the discrete valuation ring $\mathcal{O}_{\frak{C}_\bbC,p}$.
Then $\lambda$ is a uniformizing parameter of $\mathcal{O}_{\frak{C},q}$.
\end{lem}
\begin{proof}
Indeed, that $p$ and $q$ are closed points implies that $\frak{m}_q\subset k[\frak{C}]$ and $\frak{m}_p\subset \bbC[\frak{C}_\bbC]$ are maximal ideals.
Moreover the map $\Theta$ is induced by the ring inclusion $\Theta^\#:k[\frak{C}]\to k[\frak{C}]\otimes_k\bbC\cong\bbC[\frak{C}_\bbC]$ defined by $\Theta^\#(f)=f\otimes 1$.
By hypothesis, $\lambda\otimes 1$ is a uniformizing parameter of $\mathcal{O}_{\bbC,p}$ and therefore generates the maximal ideal $\frak{m}_p\mathcal{O}_{\bbC,p}$.
Since $(\Theta^\#)^{-1}(\frak{m}_p)=\frak{m}_q$ and since $\lambda= (\Theta^\#)^{-1}(\lambda\otimes 1)$, we see that $\lambda$ generates $\frak{m}_q\mathcal{O}_{\frak{C},q}$ in $\mathcal{O}_{\frak{C},q}$.
Thus $\lambda$ is a uniformizing parameter for $\mathcal{O}_{\frak{C},q}$ as required.
\end{proof}

\begin{prop}\label{prop:noextendA}
Suppose that $\rho$ is a Type A representation such that the character $\chi_\rho$ lies on $\frak{C}$ and let $p\in\widetilde{C}$ such that $\Pi_k(p)=\chi_\rho$.
Let $h_1$, $h_2$ be the corresponding choices from Proposition \ref{prop:localinverse} and suppose that there exist $\rho'\in\frak{C}_\bbC$ for which $\rho'\vert_{\langle h_1,h_2\rangle}$ is an irreducible representation.
Then if Condition $(\star_1)$ does not hold for $\rho$, the tautological Azumaya algebra $\mathcal{A}_{k(\widetilde{C})}$ does not extend over $p$.
\end{prop}
\begin{proof}
Let $\rho_\frak{C}$ be the tautological representation associated to $\frak{C}$.
Since $\widetilde{I}_{h_1}(\chi_\rho)\neq\pm 2$, we conclude that $\rho_\frak{C}(h_1)\neq\pm\mathrm{Id}$ and so by the irreducibility assumption and Lemma \ref{lem:CRShilb} it follows that $\mathcal{A}_{k(\widetilde{C})}$ has Hilbert symbol given by
$$\mathcal{A}_{k(\widetilde{C})}=\left(\frac{\widetilde\kappa,\widetilde\nu}{k(\widetilde{C})}\right)=\left(\frac{\widetilde{I}_{h_1}^2-4,\widetilde{I}_{[h_1,h_2]}-2}{k(\widetilde{C})}\right).$$
Recall that $\chi_{\rho}$ is a regular point of $\frak{C}$ and hence the local ring $\mathcal{O}_{\frak{C},\chi_\rho}$ is a discrete valuation ring such that $\mathcal{O}_{\widetilde{C},p}\cong \mathcal{O}_{\frak{C},\chi_\rho}$.
It follows that $k(p)\cong k(\chi_\rho)$ and we therefore instead argue using the local ring $\mathcal{O}_{\frak{C},\chi_\rho}$.
In particular, if $\kappa=I_{h_1}^2-4$, $\nu=I_{[h_1,h_2]}-2$ then $\ord_p(\widetilde\kappa)=\ord_{\chi_\rho}(\kappa)$, $\ord_p(\widetilde\nu)=\ord_{\chi_\rho}(\nu)$, and $\{\widetilde\kappa,\widetilde\nu\}_p=\{\kappa,\nu\}_{\chi_\rho}$ under the identification of residue fields.

Since $I_{h_1}(\chi_\rho)\neq\pm 2$, it follows that $\ord_{\chi_\rho}(\kappa)=0$.
Moreover $\nu=0$ at $\chi_\rho$ and hence $\ord_{\chi_\rho}(\nu)>0$.
By Lemma \ref{lem:generalcoord}, $\ord_{\chi_\rho}(-\nu/\kappa)=\ord_{\chi_{\rho}}(\nu)=1$ and therefore the tame symbol at $\chi_\rho$ is given by
$$\{\kappa,\nu\}_{\chi_\rho}=\kappa=(w+1/w)^2-4=w^2-2+1/w^2\in k(\chi_\rho)^*/k(\chi_\rho)^{*2}.$$
The hypotheses now guarantee that this latter quantity is not in $k(\chi_\rho)^{*2}$ and so the tame symbol $\{\kappa,\nu\}_{\chi_\rho}$ is non-trivial.
This completes the proof.
\end{proof}

\begin{prop}\label{prop:noextendB}
Suppose that $\rho$ is a Type B or C representation such that the character $\chi_\rho$ lies on $\frak{C}$ and let $p\in\widetilde{C}$ such that $\Pi_k(p)=\chi_\rho$.
Let $h_1$, $h_2$ be the corresponding choices from Proposition \ref{prop:localinverse} and suppose that there exist $\rho'\in\frak{C}_\bbC$ for which $\rho'\vert_{\langle h_1,h_2\rangle}$ is an irreducible representation.
Then if Condition $(\star_2)$ does not hold for $\rho$, the tautological Azumaya algebra $\mathcal{A}_{k(\widetilde{C})}$ does not extend over $p$.
\end{prop}
\begin{proof}
The proof is identical to that of Proposition \ref{prop:noextendA} where $w$ is replaced by a root of unity. 
\end{proof}

Note that in each possible choice of $h_1$, $h_2$ from Proposition \ref{prop:localinverse}, the subgroup $\langle h_1, h_2\rangle$ of $\pi_1(M)$ is equal to $\langle x,y\rangle$, $\langle x,t\rangle$, or $\langle y,t\rangle$.
Therefore, combining Propositions \ref{prop:noextendA} and \ref{prop:noextendB} we have now proved Theorem \ref{thm:main2}, which we restate for the reader's convenience.

\maintwo*

As remarked in the introduction, Corollary \ref{cor:canoncomp} shows that this theorem applies to any canonical component and hence yields Corollary \ref{cor:maintwocon}.

\begin{rem}\label{rem:maintwo}
In fact, one can slightly weaken the hypothesis in Theorem \ref{thm:main2}.
Specifically, instead of asking for the existence of a single representation $\rho'$ which is an irreducible representation when restricted to the three of $\langle x,y\rangle$, $\langle x,t\rangle$, $\langle y,t\rangle$ we can consider the conditions separately.
That is, if a character of a Type A representation $\rho_A$ occurs, we can ask for a representation $\rho'_A$ for which $\rho'_A\vert_{\langle h_1,h_2\rangle}$ is irreducible for the choices from Proposition \ref{prop:localinverse} and similarly for Type B and C characters.
This is a priori weaker then the condition stated in Theorem \ref{thm:main2}, however we know of no examples that exhibit the difference and the stronger condition is sufficient to handle the canonical component, as shown in Corollary \ref{cor:canoncomp}.
\end{rem}

%~~~~~~~~~~~~~~~~~~~~~~~~~~~~~~~~~~~~~~~~~~~~~~~~~~~~~~~~~~~
%~~~~~~~~~~~~~~~~~~~~~~~~~~~~~~~~~~~~~~~~~~~~~~~~~~~~~~~~~~~
%~~~~~~~~~~~~~~~~~~~~~~~~~~~~~~~~~~~~~~~~~~~~~~~~~~~~~~~~~~~

\section{Integral models}\label{sec:intmod}

We now turn our attention to effectively computing the set $S_M$ of finite rational primes in Theorem \ref{thm:main1} in the case that the tautological Azumaya algebra extends over $\widetilde{C}$.
In order to compute an appropriate set $S$ containing $S_M$, we must analyze a certain integral model $\frak{C}_S$ of $\frak{C}$ over the ring of $S$-integers $\mathcal{O}_{k,S}$.
Many of the ideas of this section are borrowed from \cite[\S 5]{CRS} and translated into our setting.

Following Chinburg--Reid--Stover, throughout this section we let $S$ be as in the hypothesis of Theorem \ref{thm:intmod} and let $\frak{C}$ be a fixed geometrically integral component of $\frak{X}(M)_k$ containing the character of an irreducible representation with field of definition $k$.
Recall that $\widetilde{C}$ is the unique smooth projective curve birational to $\frak{C}$.
In this setting, $\widetilde{C}$ admits a relatively minimal regular model $\widetilde{\frak{C}}_S$ over $\mathcal{O}_{k,S}$ which is unique up to isomorphism if the genus of $\widetilde{C}$ is at least $1$ (see for instance \cite[\S10, Prop 1.8]{Liu}).
Precisely, $\widetilde{\frak{C}}_S$ is a normal fibered surface over $\Spec(\mathcal{O}_{k,S})$ such that $\widetilde{\frak{C}}_S\to \Spec(\mathcal{O}_{k,S})$ has no exceptional divisors and such that the generic fiber $(\widetilde{\frak{C}}_S)_\xi=\widetilde{\frak{C}}_S\times_{\Spec(k)}\Spec(k_\xi)$ is isomorphic to $\widetilde{C}$.
Since $k(\widetilde{C})\cong k(\widetilde{\frak{C}}_S)$ the tautological Azumaya algebra induces a class in the Brauer group $\Br(k(\widetilde{\frak{C}}_S))$ and this section is devoted to studying the corresponding extension problem for the class of the tautological Azumaya algebra to $\widetilde{\frak{C}}_S$.

The following theorem is a mild generalization of \cite[Thm 5.9]{CRS} to our setting and the proof follows similar lines.
We therefore will try to remain notationally consistent with the proof therein and, for the convenience of the reader familiar with the proof in \cite{CRS}, we give a short list of the necessary modifications in Remark \ref{rem:CRSint}.
In the following statement recall that $F$ is the finite part of $H_1(M,\bbZ)$ as in Equation \eqref{eqn:monmatrix}.

\begin{thm}\label{thm:CRSanalogue}
Suppose that $p$ is a codimension $1$ point of $\frak{C}_S$ over which the tautological Azumaya algebra $\mathcal{A}_{k(\widetilde{C})}$ does not extend.
Then there is some $\alpha\in \overline{k(p)}$, where $\overline{k(p)}$ is the algebraic closure of the residue field $k(p)$, such that the extension $k(p)(\alpha)$ is a separable quadratic extension of $k(p)$ and one of the following holds:
\begin{enumerate}
\item $\alpha$ is an $n^{th}$ root of unity, where $n$ is the least common multiple of the orders of $x$ and $y$ in $F$,
\item There exists a homomorphism $\sigma:F\to \overline{k(p)}^*$, such that $\alpha^2$ is a root of the Alexander invariant $\Delta^{\phi_\sigma}_0(z)$ considered as an element of $\overline{k(p)}[z,z^{-1}]$.
Moreover, $\Delta_0(z)^{\phi_\sigma}$ is the characteristic polynomial of $\Phi$.
\end{enumerate}
\end{thm}
\begin{proof}
We use the notation that $\mathcal{O}_{\frak{C}_S,p}$ denotes the local ring at $p$, $\widehat{\mathcal{O}}_{\frak{C}_S,p}$ denotes the completion of this ring, $\widehat{F}_p=\mathrm{Frac}(\widehat{\mathcal{O}}_{\frak{C}_S,p})$ is the corresponding fraction field, and $\widehat{A}_p=\mathcal{A}_{k(\widetilde{C})}\otimes_{\mathcal{O}_{\frak{C}_S,p}}\widehat{F}_p$.
It follows from \cite[Lem 5.8]{CRS} and the first two paragraphs of the proof of \cite[Thm 5.9]{CRS} that $\widehat{A}_p$ is a quaternion algebra over $\widehat{F}_p$ which cannot be extended to an Azumaya algebra over $\widehat{\mathcal{O}}_{\frak{C}_S,p}$ and moreover that there is a maximal $\widehat{\mathcal{O}}_{\frak{C}_S,p}$-order $\widehat{D}$ of $\widehat{A}_p$ such that $\pi_1(M)$ admits an injective homomorphism
$$\rho_p:\pi_1(M)\to \widehat{D}^1,$$
where $\widehat{D}^1$ is the group of elements of reduced norm $1$.

The maximal order $\widehat{D}$ has a unique maximal $2$-sided ideal $J$ \cite[Prop 5.3]{CRS} and therefore for all $s\ge 0$, we can form the commutative diagram
\begin{equation}\label{eqn:shortexact}
\xymatrix{&&\widehat{D}^1\ar[d]_-{\psi_{s+1}}\ar[dr]^-{\psi_{s}}&&\\
1\ar[r]&\ker(\psi_{s+1})\ar[r]^-\iota&\widehat{D}/J^{s+1}\ar[r]^-{\overline\psi_{s+1}}\ar[r]&\widehat{D}/J^s\ar[r]&1,}
\end{equation}
where $\iota$ is simply inclusion of subgroups.
Any non-trivial element of $E_{s+1}=\ker(\overline{\psi}_{s+1})\cap \psi_{s+1}(\widehat{D}^1)$ has a representative of the form $1+a_sJ^s\pmod{J^{s+1}}$ for some $a_s\notin J$.
Moreover a straightforward calculation shows that $E_{s+1}$ is abelian since if $a=1+a_sJ^s\pmod{J^{s+1}}$ and $a'=1+a_s'J^s\pmod{J^{s+1}}$ are elements of $E_{s+1}$ then
$$aa'=1+(a_s+a_s')J^s=1+(a_s'+a_s)J^s=a'a\pmod{J^{s+1}}.$$

There is a well defined action of $\psi_s(\widehat{D}^1)$ on $E_{s+1}$ by conjugation which factors through the projection $\widehat{D}/J^s\to \widehat{D}/J$.
Indeed, let $d\in \psi_s(\widehat{D}^1)$ and let $\widetilde{d}=d_0+d_1J+\dots +d_{s}J^s+\dots$ be any element in $\widehat{D}^1$ such that $\psi_s(\widetilde{d})=d$.
Then if $a=1+a_sJ^s\pmod{ J^{s+1}}$ is any element of $E_{s+1}$, the computation
$$\widetilde{d}a\widetilde{d}^{-1}=1+d_0a_sJ^s\overline{d}_0\pmod{J^{s+1}},$$
shows that this conjugation is independent of representative $\widetilde{d}$ for $d$ and moreover that this conjugation only depends on the class of $\widetilde{d}$ in $\widehat{D}/J$.
Therefore, the conjugation action of $\psi_s(\widehat{D}^1)$ on $E_{s+1}$ factors through the map $\widehat{D}/J^s\to \widehat{D}/J$ as required.

We can form the similar commutative diagram as Equation \eqref{eqn:shortexact} using the homomorphism $\rho_p$.
In particular, let $W_s=\psi_s(\rho_p(\pi_1(M)))$ then we have the commutative diagram
$$\xymatrix{&&\pi_1(M)\ar[d]_-{\psi_{s+1}\circ\rho_p}\ar[dr]^-{\psi_{s}\circ\rho_p}&&\\
1\ar[r]&Q_s\ar[r]^-\iota&W_{s+1}\ar[r]^-{\overline\psi_{s+1}}\ar[r]&W_s\ar[r]&1,}$$
where we continue to use the notation $\overline\psi_{s+1}$ even though the domain has been restricted and where $Q_s=W_{s+1}\cap E_{s+1}$.
Note that $Q_s$ is abelian as it is a subgroup of $E_{s+1}$.
Also note that $\widehat{D}/J$ is a quadratic field extension of $k(p)$ and hence $W_1$ is abelian.
Consequently we may fix a natural number $s_0$ such that
$$s_0=\max\{s\in\bbN\mid W_j\text{ is abelian for all }j\le s\},$$
where the above set is non-empty since $W_1$ is abelian.
Note that the maximum exists because $\rho_p$ is faithful.
\medskip

\noindent\textbf{Claim: $W_{1}=\psi_{1}(\rho_p(\pi_1(M)))$ is not contained in $k(p)$.}
\smallskip

\noindent\textbf{Proof of claim:}
Suppose for contradiction that $W_{1}=\psi_{1}(\rho_p(\pi_1(M)))$ is contained in $k(p)$ and recall our assumption that the characteristic of $k(p)$ is prime to $n$.
We first show that $\psi_{s_0+1}(\rho_p(x))$, $\psi_{s_0+1}(\rho_p(y))$ are of the form $x_0+x_{s_0}J^{s_0}$, $y_0+y_{s_0}J^{s_0}$ in $W_{s_0+1}$, respectively.
Indeed, write $x_0+x_1J+\dots+ x_{s_0}J^{s_0}$ with $x_i\notin J$.
Then since $W_{s_0}$ is abelian, $x^m\equiv 1\pmod{J^{s_0}}$ for some $m$ dividing $n$, and hence the calculation
\begin{equation}\label{eqn:torsionmods}
x^m\equiv x^m_0+mx_0^{m-1}x_1J+\dots\equiv 1\pmod{J^{s_0}},
\end{equation}
shows that $x_0^m=1$ and $mx_0^{m-1}x_1\in J^{s_0-1}$.
Here we have used the fact that $\psi_{s_0+1}(\rho_p(x))$ is contained in $k(p)$ to write the above equation, since we need that $x_1Jx_0=x_0x_1J$.
As $J$ is maximal, we therefore conclude that $x_1= 0$ by the condition on the characteristic of $k(p)$ since $x_0,x_1\notin J$ by construction.
A straightforward induction argument shows that, in fact,  $\psi_{s_0+1}(\rho_p(x))$ is of the form $x_0+x_{s_0}J^{s_0}$.
The identical argument holds for $\psi_{s_0+1}(\rho_p(y))$.

Writing $\psi_{s_0+1}(\rho_p(t))=t_0+t_1J+\dots+t_{s_0}J^s$ with each $t_i\notin J$, it now follows immediately that $\psi_{s_0+1}(\rho_p(x))$, $\psi_{s_0+1}(\rho_p(y))$, $\psi_{s_0+1}(\rho_p(t))$ all commute mod ${J^{s_0+1}}$ since $x_0,y_0,t_0\in k(p)$ and therefore commute with $J$.
This shows that $W_{s_0+1}$ is abelian, contradicting the maximality of $s_0$.
\medskip

Continuing the proof, the claim above implies that the action by conjugation of elements in $W_1$ on $Q_{s_0}$ is non-trivial.
We then have two cases, depending on whether the image of the stable letter lands in $k(p)$ or not.

\medskip 

\textbf{Case 1: $\psi_1(\rho_p(t))\in k(p)$.}
The composition $\psi_1\circ\rho_p$ factors through $H_1(M,\bbZ)$ since $W_1$ is abelian and hence the confluence of the condition that $\psi_1(\rho_p(t))\in k(p)$ and that $\psi_1(\rho_p(\pi_1(M)))\notin k(p)$ implies that there exists an element of $\pi_1(M)$ of the form $x^ay^b$ such that $\psi_1(\rho_p(x^ay^b))=\alpha\notin k(p)$.
Moreover any element of the form $x^ay^b$ has order dividing $n$ in homology and therefore $\alpha$ is an $n^{th}$ root of unity.
This combined with \cite[Prop 5.3]{CRS} gives conclusion (1) of the theorem.

\medskip

\textbf{Case 2: $\psi_1(\rho_p(t))\notin k(p)$.}
Let $\alpha=\psi_1(\rho_p(t))$ then $k(p)(\alpha)$ is a separable quadratic extension of $k(p)$ \cite[Prop 5.3]{CRS}.
We will show directly that $\alpha^2$ satisfies the characteristic polynomial for $\Phi$.
Denote by $\overline{t}=t_0+t_1J+\dots+ t_{s_0}J^{s_0}$ the image of $t$ in $W_{s_0+1}$, where $t_0=\alpha$ by assumption.
First we show that the image of $\langle x,y\rangle$ in $W_{s_0+1}$ is an abelian group.
Indeed, we know that $t$ and the longitude $l=[ x,y]$ commute in $\pi_1(M)$ and hence so does their image in $W_{s_0+1}$.
Write $\overline{l}$ for the image of $l$ in $W_{s_0+1}$.
As $W_{s_0}$ is abelian, $\overline{l}$ is contained in $E_{s_0+1}=\ker(\overline\psi_{s_0+1})\cap\psi_{s_0+1}(\widehat{D}^1)$ and therefore we conclude that $\overline{l}=1+l_{s_0}J^{s_0}\pmod{J^{s_0+1}}$.
One computes that
\begin{equation}\label{eqn:conjact}
\overline{t}( 1+l_{s_0}J^{s_0})\overline{t}^{-1}=1+t_0 l_{s_0}J^{s_0}t_0^{-1}=1+t_0^2l_{s_0}J^{s_0}\pmod{J^{s_0+1}},
\end{equation}
since $\psi_1(\rho_p(t))=\alpha=t_0\notin k(p)$, $s_0$ is odd \cite[Prop 5.3(iv)]{CRS}, and the reduced norm of $\alpha$ is $1$ so that $J\alpha^{-1}=\alpha J$.
In particular, since $t_0\notin k(p)$ it follows that $t_0^2\neq1$, hence in order for $\overline{t}$ and $\overline{l}$ to commute we necessarily have that $l_{s_0}=0$.
As such, it follows that the image of $x$ and $y$ under $\psi_{s_0+1}$ commute in $W_{s_0+1}$.
From this and the fact that $W_{s_0}$ is abelian, we conclude that
$$1+c_{s_0}J^{s_0}=\psi_{s_0+1}([t,x])=\psi_{s_0+1}(\phi(x)x^{-1})=\psi_{s_0+1}(x^{a-1}y^c)\pmod{J^{s_0+1}},$$
for some $c_{s_0}$.

Let $\overline{x}$, $\overline{y}$ denote the image of $x$, $y$ in $W_{s_0+1}$ under $\psi_{s_0+1}$, respectively.
We claim that $\overline{x}$, $\overline{y}$ are of the form $x_0+x_{s_0}J^{s_0}$, $y_0+y_{s_0}J^{s_0}$, respectively.
Indeed, let $\overline{x}=x_0+x_1J+\dots +x_{s_0}J^{s_0}$ where $x_0\neq 0$, then it follows that
\begin{align}
\overline{x}\psi_{s_0+1}(x^{a-1}y^c)&=\psi_{s_0+1}(x^{a-1}y^c)\overline{x}\pmod{J^{s_0+1}},\nonumber\\
\overline{x}+x_0c_{s_0}J^{s_0}&=\overline{x}+c_{s_0}J^{s_0}x_0\pmod{J^{s_0+1}}\label{eqn:modj},
\end{align}
and hence either $c_{s_0}=0$ or $x_0\in k(p)^*$.
Indeed, recall that $s_0$ is odd and if $c_{s_0}\neq 0$ then the condition that $x_0$ must commute with $J^{s_0}$ is the condition that $x_0\in k(p)^*$.
Similarly, if we write
$$1+d_{s_0}J^{s_0}=\psi_{s_0+1}([t,y])\pmod{J^{s_0+1}},$$
then either $d_{s_0}=0$ or $y_0\in k(p)^*$.
We now show that $x_0, y_0\in k(p)^*$, in which case the argument in the paragraph surrounding Equation \eqref{eqn:torsionmods} shows that $\overline{x}=x_0+x_{s_0}J^{s_0}$, $\overline{y}=y_0+y_{s_0}J^{s_0}$ yielding the above claim.

Note that we cannot have that $c_{s_0}=d_{s_0}=0$ as then $W_{s_0+1}$ would be abelian, contradicting the choice of $s_0$.
If $c_{s_0}\neq 0$ and $d_{s_0}\neq 0$ then we are done by the previous paragraph.
Therefore to exhibit the claim it suffices to show that if $c_{s_0}=0$, $d_{s_0}\neq 0$ then $x_0,y_0\in k(p)^*$, as the symmetric argument will apply in the case that $c_{s_0}\neq 0$, $d_{s_0}=0$.

Suppose that $c_{s_0}=0$.
Then since $d_{s_0}\neq 0$ it follows that $y_0\in k(p)^*$ and hence, as remarked above, we conclude that $\overline{y}=y_0+y_{s_0}J^{s_0}$.
Additionally, $y_{s_0}\neq 0$ since otherwise $t$ and $y$ commute, contradicting that $d_{s_0}\neq 0$.
Since $x$ and $y$ commute mod $J^{s_0+1}$, it follows that
$$\overline{x}~\!\overline{y}=\overline{x}y_0+x_0y_{s_0}J^{s_0}=y_0\overline{x}+y_{s_0}J^{s_0}x_0=\overline{y}~\!\overline{x}\pmod{J^{s_0+1}}.$$
In particular, $x_0y_{s_0}J^{s_0}=y_{s_0}J^{s_0}x_0$ and since $s_0$ is odd it follows, as above, that $x_0\in k(p)^*$ as well.
Therefore we see that $x_0,y_0\in k(p)^*$ as required.

One now readily computes that $\overline{x}^m=x_0+mx_0^{m-1}x_{s_0}J^{s_0}$ for any $m\in\mathbb{N}$ and similarly for $\overline{y}$, where we are allowing for the case that one of $x_{s_0}, y_{s_0}$ is zero.
A straightforward computation therefore shows that
$$x_0+t_0^2x_{s_0}J^{s_0}=\overline{t}\overline{x}\overline{t}^{-1}=\overline{x}^a\overline{y}^c=x_0^ay_0^{c}+(ax_0^{a-1}y_0^cx_{s_0}+cx_0^{a}y_0^{c-1}y_{s_0})J^{s_0}\pmod{J^{s_0+1}},$$
and similarly 
$$y_0+t_0^2y_{s_0}J^{s_0}=\overline{t}\overline{y}\overline{t}^{-1}=\overline{x}^b\overline{y}^d=x_0^by_0^{d}+(bx_0^{b-1}y_0^dx_{s_0}+dx_0^{b}y_0^{d-1}y_{s_0})J^{s_0}\pmod{J^{s_0+1}}.$$
Note that since $W_1$ is abelian it follows that $x_0^{a-1}y_0^c=x_0^by_0^{d-1}=1\in k(p)$ and hence
\begin{align}
x_{s_0}(t_0^2-a)&=cx_0y_0^{-1}y_{s_0},\label{eqn:xs}\\
y_{s_0}(t_0^2-d)&=by_0x_0^{-1}x_{s_0}.\label{eqn:ys}
\end{align}
Recall that since we are considering hyperbolic once punctured torus bundles, we have by definition that $b,c\neq 0$ in $\bbZ$ and since $x_0^{a-1}y_0^c=1$, it follows that $x_0,y_0\neq 0$.

Suppose first that the characteristic of $k(p)$ divides $c$. 
Then $x_{s_0}(t_0^2-a)=0$ and hence either $x_{s_0}=0$ or $t_0^2-a=0$.
In the latter case, we are done since $ad-bc=1$ in $\bbZ$ implies that $ad=1$ in $k(p)$ and therefore
\begin{equation}\label{eqn:t0char}
(t_0^2-a)(t_0^2-d)=t_0^4-(a+d)t_0^2+1=0,
\end{equation}
in $k(p)$.
This is precisely the condition that $t_0^2=\alpha^2$ satisfies the characteristic polynomial of $\Phi$.
Therefore we assume that $x_{s_0}=0$.
From Equation \eqref{eqn:ys} one concludes that either $y_{s_0}=0$ or $t_0^2-d=0$. 
In the former case, it follows that $W_{s_0+1}$ is abelian so it must be that $t_0^2-d=0$.
But then Equation \eqref{eqn:t0char} again holds and we complete the proof.
The symmetric argument holds when the characteristic of $k(p)$ divides $b$.

Therefore assume that the characteristic of $k(p)$ does not divide $b$ or $c$.
Then it follows from Equations \eqref{eqn:xs} and \eqref{eqn:ys} that $x_{s_0}=0$ if and only if $y_{s_0}=0$.
If $x_{s_0}=y_{s_0}=0$, one concludes that $W_{s_0+1}$ is abelian as before and therefore, we must have $x_{s_0},y_{s_0}\neq 0$.
Multiplying Equations \eqref{eqn:xs} and \eqref{eqn:ys} together and using that $ad-bc=1$, we conclude that
$$t_0^4-(a+d)t_0^2+1=0,$$
and hence again $\alpha^2$ satisfies the characteristic polynomial of $\Phi$.
This completes the proof.
\end{proof}

\begin{rem}\label{rem:CRSint}
There are some small but key changes in the adaptation of the above proof from \cite[Thm 5.9]{CRS}, which we now detail.
First, the argument is identical to that in \cite{CRS} until the claim that the image of $W_1$ is not contained in $k(p)$, which is where we imposed the condition that the residue field characteristic is prime to the torsion in $H_1(M,\bbZ)$.
In \cite{CRS}, $M$ is always a hyperbolic knot complement and therefore $H_1(M,\bbZ)\cong \bbZ$, which importantly is not just abelian but cyclic. 
In particular, this implies that $W_{s_0}$ is cyclic since it factors through homology.
To the author's understanding, the argument in \cite{CRS} then proceeds by the logic that if $\alpha\equiv \gamma^m\pmod{J^s}$, $\beta\equiv \gamma^{m'}\pmod{J^{s_0}}$ then
$$\alpha\beta=\gamma^{m+m'}+\beta_0\alpha_{s_0}J^{s_0}+\alpha_0\beta_{s_0}J^{s_0}=\beta\alpha\pmod{J^{s_0+1}},$$
provided that $W_1$, and in particular $\alpha_0$, $\beta_0$, is contained in $k(p)$.
Such an argument is not afforded to us in the case that $H_1(M,\bbZ)$ is merely abelian, so we must impose the condition on characteristic. 
We suspect that this condition is actually unnecessary (see Question \ref{ques:amidumb}).

Second, in \cite[Thm 5.9]{CRS} the proof is concluded using general facts about fitting ideals and the structure of $\Gamma'/[\Gamma',\Gamma']$ as a $\bbZ[\pi_1(M)^{ab}]$, where $\Gamma'=[\pi_1(M),\pi_1(M)]$. 
One can indeed make the similar argument along these lines in our setting, where the Alexander polynomial is replaced by $\Delta_0^\phi(z)$, however we instead opted to give a concrete computation in the proof presented.
\end{rem}

We now give one final simple lemma which analyzes the rational primes for which Condition (2) of Theorem \ref{thm:CRSanalogue} holds.

\begin{lem}\label{lem:finfield}
Suppose that $\Tr(\Phi)=a^2+2$ for some $a\in\bbN$ and let $T$ denote the largest set of rational primes such that for each $\ell\in T$ the following holds: if $z$ is a root of the characteristic polynomial for $\Phi$ in the algebraic closure $\overline{\bbF}_\ell$ and $w\in \overline{\bbF}_\ell$ is such that $w^2=z$ then $\mathbb{F}_\ell(w)\neq\mathbb{F}_\ell(w+1/w)$.
Then $T=\{2\}$ if $\Tr(\Phi)$ is odd and $T=\emptyset$ if $\Tr(\Phi)$ is even.
\end{lem}

\begin{proof}
The characteristic polynomial of $\Phi$ is $x^2-(a^2+2)x+1$ and therefore $w$ satisfies $x^4-(a^2+2)x^2+1$.
As in Proposition \ref{prop:galois}, $x^4-(a^2+2)x^2+1$ factors over $\bbZ$ as $(x^2+ax-1)(x^2-ax-1)$ and hence $w$ satisfies one of these two equations in $\overline{\bbF}_\ell[x]$.

Suppose first that $\ell$ is an odd prime.
If $\ell\nmid a^2+4$ then the argument given for $\bbQ$ in Proposition \ref{prop:galois} goes through in a straightforward manner for $\bbF_\ell$.
Indeed, either $a^2+4$ is a quadratic residue and $\mathbb{F}_\ell[w]=\mathbb{F}_\ell=\mathbb{F}_\ell[w+1/w]$ or $a^2+4$ is not a quadratic residue and $\mathbb{F}_\ell[w]=\mathbb{F}_\ell[w+1/w]$ is a quadratic extension of $\bbF_\ell$.
Now assume that $\ell\mid a^2+4$ and let $\lambda\in\bbF_\ell$ be such that $\lambda\equiv a/2\pmod{\ell}$, where here we have used that $\ell$ is odd to invert $2$. 
Then $2\lambda\equiv a\pmod{\ell}$ and the condition that $\ell\mid a^2+4$ shows that $\lambda^2\equiv-1\pmod{\ell}$.
In particular,
$$x^4-(a^2+2)x^2+1=(x^2+ax-1)(x^2-ax-1)=(x+\lambda)^2(x-\lambda)^2.$$
As $w=\pm \lambda\in\mathbb{F}_\ell$ and hence $\mathbb{F}_\ell[w]=\mathbb{F}_\ell=\bbF_\ell[w+1/w]$, we therefore conclude that no odd prime $\ell$ for which $\ell\mid a^2+4$ is an element of $T$.

From the above analysis, $T\subset\{2\}$ whenever $\Tr(\Phi)=a^2+2$.
If $\ell=2$ then $\Tr(\Phi)$ is even, or equivalently if $a$ is even, then $x^4-(a^2+2)x^2+1=(x+1)^4$ and $\mathbb{F}_\ell[w]=\mathbb{F}_\ell=\bbF_\ell[w+1/w]$.
Hence $T=\emptyset$.
If $\Tr(\Phi)$ is odd, equivalently if $a$ is odd, then $x^2+x+1$ is irreducible in $\bbF_2[x]$.
Consequently $w$ is a root of $x^2+x+1$ and $\bbF_2[w]\cong\bbF_4$.
Since $w^2+1=w$
$$w+\frac{1}{w}=\frac{w^2+1}{w}=1,$$
and therefore $\bbF_4\cong\bbF_2[w]\neq \bbF[w+1/w]\cong\bbF_2$.
From this we see that $T=\{2\}$ when $\Tr(\Phi)$ is odd.
This completes the proof.
\end{proof}

We now prove Theorem \ref{thm:intmod}, which we restate for the reader's convenience.

\intmod*

\begin{proof}[Proof of Theorem \ref{thm:intmod}]
Let $S$ be as in the statement of the theorem and suppose that $p\in \widetilde{\frak{C}}_S$ is a codimension $1$ point.
Note that either $p$ is a horizontal divisor and $k(p)$ has characteristic $0$ or $p$ is a vertical divisor and $k(p)$ has characteristic $\ell$ for some rational prime $\ell\notin S$ \cite[Lem 8.3.4]{Liu}.
In the former case, $p$ lies on the general fiber $\widetilde{C}$ of $\widetilde{\frak{C}}_S$ and the hypothesis of Theorem \ref{thm:intmod} already implies that the tautological Azumaya algebra extends over such points.
It therefore suffices to consider those points for which the residue field $k(p)$ has characteristic $\ell>0$.
If $p$ were such that $\mathcal{A}_{k(\widetilde{\frak{C}_S})}$ did not extend then either (1) or (2) of Theorem \ref{thm:CRSanalogue} holds.
However, (1) cannot hold by assumption on $S$ and, similarly, (2) cannot hold by assumption on $S$ combined with Lemma \ref{lem:finfield}.
Therefore $\mathcal{A}_{k(\widetilde{\frak{C}_S})}$ also extends over all codimension $1$ points with finite residue characteristic.
This completes the proof.
\end{proof}

%~~~~~~~~~~~~~~~~~~~~~~~~~~~~~~~~~~~~~~~~~~~~~~~~~~~~~~~~~~~
%~~~~~~~~~~~~~~~~~~~~~~~~~~~~~~~~~~~~~~~~~~~~~~~~~~~~~~~~~~~
%~~~~~~~~~~~~~~~~~~~~~~~~~~~~~~~~~~~~~~~~~~~~~~~~~~~~~~~~~~~

\section{A cavalcade of examples}\label{sec:examples}

We now apply the results of the previous sections to a host of examples to explore situations where one can concretely compute whether the tautological Azumaya algebra extends over a canonical component.

%~~~~~~~~~~~~~~~~~~~~~~~~~~~~~~~~~~~~~~~~~~~~~~~~~~~~~~~~~~~
%~~~~~~~~~~~~~~~~~~~~~~~~~~~~~~~~~~~~~~~~~~~~~~~~~~~~~~~~~~~
%~~~~~~~~~~~~~~~~~~~~~~~~~~~~~~~~~~~~~~~~~~~~~~~~~~~~~~~~~~~

\subsection{Tunnel number one torus bundles}\label{sec:tunnelnumone}
If $M$ is a once punctured torus bundle of tunnel number one, then $M$ is one of an infinite family $\{M_j\}_{j\in A}$ where $A=\bbZ\setminus\{-4,-3,\dots,0\}$ and $M_j$ has monodromy $RL^j$ (see \cite{BakerPetersen,Sakuma} and Example \ref{ex:tunnelone}).
Following Baker--Petersen \cite[\S 2.1]{BakerPetersen}, we also have a presentation given by
$$\pi_1(M_j)=\langle x,y,t\mid txt^{-1}=x(yx)^j,~tyt^{-1}=yx\rangle=\langle a,y\mid y^{j+2}=a^{-1}ya^2ya^{-1}\rangle,$$
where $a=y^{-1}t$.
For the convenience of the reader familiar with \cite{BakerPetersen}, we point out that in their notation $\mu=t$, $\alpha=a$, $\gamma=x$, and $j=-(n+2)$ so, in particular, our family is indexed slightly differently than theirs.
Abelianizing the latter description of $\pi_1(M_j)$, we see that $H_1(M_j,\bbZ)\cong\bbZ\oplus\bbZ/|j|\bbZ$ where the image of $y$ generates the factor of $\bbZ/|j|\bbZ$ and the image of $x$ is trivial in $H_1(M_j,\bbZ)$.

Baker and Petersen study the $\SL_2(\bbC)$-character variety of the family $\{M_j\}$ and show the following theorem.

\begin{thm}[\cite{BakerPetersen}, Theorem 5.1]\label{thm:BP}
For all $j\in A$, there is a unique canonical component $\frak{X}^{can}(M_j)$ of $\frak{X}(M_j)_\bbQ$ which is birational to a hyperelliptic curve over $\bbQ$.
If $j\not\equiv 0\pmod{4}$ then this is the unique component of $\frak{X}(M_j)_\bbQ$ containing the character of an irreducible representation and if $j\equiv 0\pmod{4}$, then there is an additional component which is isomorphic to $\mathbb{A}^1_\bbQ$ over $\bbQ$.
\end{thm}

\noindent Recall that, though there is a unique canonical component of the $\PSL_2(\bbC)$-character variety, there may be multiple components containing different lifts of the discrete faithful representation in the $\SL_2(\bbC)$-character variety.
The above theorem says that in this particular instance, this phenomenon does not happen.

In \cite[\S 9]{BakerPetersen}, Baker and Petersen go on to also classify when the character of a representation of Type A and B lies on the canonical component $\frak{X}^{can}(M_j)$.
Specifically, using the natural coordinates $(\Tr_a,\Tr_y,\Tr_{ay})$ where $\Tr_{ay}=\Tr_t$, they show that characters of Type A representations correspond to the points $(\pm s,2,\pm s)$ with $s=w+1/w$, and characters of Type B representations correspond to points $(\pm y_m,y_m,2)$ with $y_m=\eta_{|j|}^m+1/\eta_{|j|}^m=2\cos(2\pi m/|j|)$.
Moreover, they show that all of these points lie on $\frak{X}^{can}(M_j)$ unless $y_m\in\{ -2,0,  2\}$ (see the schematics \cite[Figs 2, 3]{BakerPetersen}). 
They also show that when $j\equiv 0\pmod{4}$, there is an additional line of characters of representations given by $L=\{(0,0,z)\}$, which are precisely the characters of representations for which $\rho(y)$ has eigenvalues $\pm i$.
This line contains the character of the Type B representations $(0,0,\pm 2)$, however this point does not lie on $\frak{X}^{can}(M_j)$.
In fact, when $k\equiv 0\pmod{4}$ there are two points $(0,0,\pm z_0)$ in the intersection $\frak{X}^{can}(M_j)\cap L$, however each of these points correspond to characters of irreducible representations.
Note that the two generation of $\pi_1(M_j)$ ensures that there are no non-abelian reducible representations of Type C.

We collate all of this information into a succinct statement.

\begin{thm}[Baker--Petersen]\label{thm:BP2}
Fix $j\in A$ and let $M_j$ be the corresponding once punctured torus bundle with monodromy $RL^j$.
Then the following hold:
\begin{enumerate}
\item There are always characters of Type A representations on the canonical component $\frak{X}^{can}(M_j)$,
\item There are never characters of Type C representations on the canonical component $\frak{X}^{can}(M_j)$,
\item When $j\notin \{1,2,4\}$, there are always characters of Type B representation on the canonical component $\frak{X}^{can}(M_j)$,
\item When $j=1$, there are no characters of Type B representations,
\item When $j=2$, no characters of Type B representations lie on the canonical component $\frak{X}^{can}(M_2)$,
\item When $j=4$, no characters of Type B representations lie on the canonical component $\frak{X}^{can}(M_4)$ but there are characters of Type B representations on the line $L$.
\end{enumerate}
\end{thm}

As a consequence of Theorems \ref{thm:main1}, \ref{thm:main2}, and \ref{thm:BP2} and the fact that the Azumaya algebra extends over a Type A representation if and only if $\Tr(\Phi_j)=j+2=a^2+2$ for some $a\in\bbZ$, we conclude the following.

\begin{thm}
Let $M_j$ be a hyperbolic once punctured torus bundle of tunnel number one and let $\frak{C}=\frak{X}^{can}(M_j)$.
Then the Azumaya algebra $\mathcal{A}_{k(\widetilde{C})}$ extends to all of $\widetilde{C}$ if and only if $j\in\{1,4\}$.
\end{thm}

\noindent We mention that one also has a similar statement for the line of characters $L$ when $j\equiv 0\pmod{4}$, namely that $\mathcal{A}_{k(L)}$ never extends over this component due to the characters of Type B representations at $(0,0,\pm2)$

As in the introduction of the paper, let $M_j^{p/q}$ denote the closed $3$-orbifold obtained by doing $p/q$-Dehn filling on the cusp, provided it is hyperbolic, and let $k_{M_j^{p/q}}$ denote the corresponding trace field and $\rho_j^{p/q}:\pi_1(M_j^{p/q})\to\SL_2(\bbC)$ any lift of the holonomy representation.
Then we form the corresponding quaternion algebra over $k_{M_j^{p/q}}$ by
$$\mathcal{A}_{M_j^{p/q}}=\{\sum a_i\rho_j^{p/q}(\gamma)\mid a_i\in k_{M_j^{p/q}},~\gamma\in\pi_1(M_j^{p/q})\}.$$
We therefore have the following corollary, where the case of $j=1$ was also previously obtained by Chinburg--Reid--Stover \cite[Thm 1.7]{CRS}, which follows directly from Theorem \ref{thm:main2}.

\begin{cor}
Let $S_{M_j}$ denote the set of rational primes such that all of the algebras $\mathcal{A}_{M_j^{p/q}}$ ramify over a prime in $S_{M_j}$ as $p/q$ vary over all hyperbolic Dehn surgeries.
Then $S_{M_j}\subset\{2\}$ when $j=1,4$.
\end{cor}

\begin{rem}
It can be shown that $S_1=\{2\}$, that is, one can actually find a hyperbolic Dehn filling where $\mathcal{A}_{M_j^{p/q}}$ has finite ramification.
However it seems numerically as though $S_4$ should in fact actually be empty.
This is related to Question \ref{ques:amidumb} and the surrounding discussion.
\end{rem}

%~~~~~~~~~~~~~~~~~~~~~~~~~~~~~~~~~~~~~~~~~~~~~~~~~~~~~~~~~~~
%~~~~~~~~~~~~~~~~~~~~~~~~~~~~~~~~~~~~~~~~~~~~~~~~~~~~~~~~~~~
%~~~~~~~~~~~~~~~~~~~~~~~~~~~~~~~~~~~~~~~~~~~~~~~~~~~~~~~~~~~

\subsection{Cyclic covers of the figure eight knot}

A particularly interesting family of examples come from $j$-fold cyclic covers $N_j$ of the figure eight knot complement, as in Example \ref{ex:nfold}.
Recall that
$$\pi_1(N_1)=\langle x,y,t\mid txt^{-1}=xyx,~tyt^{-1}=yx\rangle,$$
and that the manifolds $N_j$ have fundamental groups given by $\pi_1(N_j)=\langle x,y,t^j\rangle<\pi_1(N_1)$.
Moreover, the fiber subgroup $\pi_1(F)=\langle x,y\rangle$ is contained in $\pi_1(N_j)$ for all $j$.
Therefore, via restriction of representations, we have induced algebraic maps $\widehat{i}_1:\frak{X}(N_1)\to \frak{X}(F)$, $\widehat{i}_j:\frak{X}(N_j)\to \frak{X}(F)$, and $\widehat{p}_j:\frak{X}(N_1)\to \frak{X}(N_j)$ for all $j\in\bbN$.
The character scheme $\frak{X}(N_1)$ has a unique curve containing the character of an irreducible representation, which we call $\frak{C}_1$ (see \cite[\S 7]{CRS} for instance).

The image of $\frak{C}_1$ under each $\widehat{p}_j$ is also a curve, which we call $\frak{C}_j$.
By restriction, the curve $\frak{C}_j$ contains the character of a discrete faithful representation of $\pi_1(N_j)$ and hence $\frak{C}_j$ is a canonical component of $\frak{X}(N_j)$.
Therefore we have the following commutative diagram
$$\xymatrix{
\frak{C}_1\ar^-{\widehat{i}_1}[r]\ar_-{\widehat{p}_j}[d]&\frak{X}(F)\ar@{=}[d]\\
\frak{C}_j\ar^-{\widehat{i}_j}[r]&\frak{X}(F).}$$
Let $\phi_j$ be the restriction of the homomorphism $\phi_1:\pi_1(N_1)\to\{-1,1\}$ defined by
$$\phi_1(x)=1,\quad\phi_1(y)=1,\quad\phi_1(t)=-1,$$
to the subgroup $\pi_1(N_j)=\langle x,y,t^j\rangle$.
Then given any representation $\rho$ of $\pi_1(N_j)$, we can construct a new representation $\phi_j\cdot\rho$ via $(\phi_j\cdot\rho)(\gamma)=\phi_j(\gamma)\rho(\gamma)$.
This assignment induces an algebraic isomorphism $\widehat\phi_j:\frak{X}(N_j)\to \frak{X}(N_j)$ by sending $\chi_\rho$ to $\chi_{\phi_j\cdot\rho}$.
In \cite[Prop 5.2]{BLZ} specialized to our setting, Boyer, Luft, and Zhang show that if $\widehat\phi_j(\frak{C}_j)=\frak{C}_j$ then $\widehat{i}_j$ is degree $2$ and otherwise it is degree $1$.

Returning to our above diagram, since $\frak{C}_1$ is the unique component of characters of irreducible representation, it follows that $\widehat{\phi}_1(\frak{C}_1)=\frak{C}_1$ and so $\widehat{i}_1$, and hence the composition $\widehat{i}_j\circ \widehat{p}_j$ is degree $2$ for every $j$.
When $j$ is even, we argue that $\widehat{i}_j$ has degree $1$ as in \cite[Prop 6.1]{BLZ}.
For this, notice that $\phi_1\vert_{\pi_1(N_j)}\equiv 1$ and therefore $\widehat{p}_j\circ\widehat{\phi}_j=\widehat{p}_j$.
Hence $\widehat{p}_j$ is degree $2$ and $\widehat{i}_j$ is degree $1$.

When $j$ is odd, we argue as in \cite[Prop 6.2]{BLZ} that $\widehat{i}_j$ has degree $2$.
Notice that $\widehat{p}_j\circ\widehat\phi_1=\widehat\phi_j\circ\widehat{p}_j$ and recall from before that $\widehat{\phi}_1(\frak{C}_1)=\frak{C}_1$.
Therefore
$$\widehat\phi_j(\frak{C}_j)=\widehat\phi_j(\widehat{p}_j(\frak{C}_1))=\widehat{p}_j(\widehat\phi_1(\frak{C}_1))=\frak{C}_j,$$
and hence $\widehat{i}_j$ has degree $2$ and $\widehat{p}_j$ has degree $1$.

Note that $\widehat{p}_j:\frak{C}_1\to\frak{C}_j$ is a regular map defined over $\bbQ$, as the characters of elements of $\pi_1(N_j)=\langle x,y,t^j\rangle$ are expressible as polynomials with integer coefficients in characters of $x$, $y$, $t$.
Since $\frak{C}_1$ is defined over $\bbQ$ in the sense of Long and Reid \cite{LongReid2}, so too is $\frak{C}_j$.
Moreover the curve $\frak{C}_1$ is a smooth affine curve and hence is isomorphic to the smooth projective completion $\widetilde{C}_1$ at all non-ideal points (\cite[\S 7]{CRS}).
Suppose that $j$ is odd, it follows that the corresponding tautological Azumaya algebras $\mathcal{A}_{\bbQ(\widetilde{C}_j)}$ are all isomorphic.
From the fact that $\mathcal{A}_{\widetilde{C}_1}$ exists and $\widehat{p}_j$ is an isomorphism away from ideal points, it follows that $\mathcal{A}_{\bbQ(\widetilde{C}_j)}$ extends to an Azumaya algebra $\mathcal{A}_{\widetilde{C}_j}$ isomorphic to $\mathcal{A}_{\widetilde{C}_1}$ whenever $j$ is odd.

When $j$ is even, the opposite is true due to the presence of characters of representations of $\pi_1(N_1)$ on $\frak{C}_1$ with dihedral image.
The figure eight knot complement has the following irreducible representation
$$\rho(x) =  \begin{pmatrix} \eta_5 & 0 \\ 0 &  \eta_5^{-1} \end{pmatrix},\quad\rho(y) =  \begin{pmatrix}  \eta_5^{2} & 0 \\ 0 &  \eta_5^{-2} \end{pmatrix},\quad\rho(t) = \begin{pmatrix} 0 & 1 \\ -1 & 0 \end{pmatrix},$$
with dihedral image on the canonical component $\frak{C}_1$, where $\eta_5$ is a primitive $5^{th}$ root of unity.
Its restriction to $\pi_1(N_j)$ for $j$ even is an abelian representation, as $\rho(t^j)=\pm\mathrm{Id}$.
However it follows from the considerations above that $\widehat{p}_j(\chi_\rho)$ agrees with the character of a non-abelian reducible representation of Type B since $\widehat{p}_j(\chi_\rho)$ lies on $\frak{C}_j$.
Explicitly, as in Example \ref{ex:nfold}, let $j=2k$ and define a representation $\widehat\rho_j:\pi_1(N_j)\to \SL_2(\bbC)$ via
$$\widehat\rho_j(x) =  \begin{pmatrix} \eta_5 & 0 \\ 0 &  \eta_5^{-1} \end{pmatrix},\quad\widehat\rho_j(y) =  \begin{pmatrix}  \eta_5^{2} &(-1)^k \frac{1-\eta_5^{-1}}{k} \\ 0 &  \eta_5^{-2} \end{pmatrix},\quad\widehat\rho_j(t^j) = \begin{pmatrix} (-1)^k & 1 \\ 0 & (-1)^k \end{pmatrix}.$$
Perhaps the most straightforward way to see that $\widehat\rho_j$ satisfies the relators of $\pi_1(N_j)$ is to check it for $j=2$ and then note that for all $j$ even, $\widehat\rho_j$ is a conjugate of $\widehat\rho_2\vert_{\pi_1(N_j)}$ by the diagonal matrix with entries $(\sqrt{k})^{\pm 1}$.

A straightforward calculation shows that if $\rho_j=\rho\vert_{\pi_1(N_j)}$, then for $j$ even, $\chi_{\rho_j}=\chi_{\widehat\rho_j}$ and hence $\chi_{\widehat\rho_j}\in \frak{C}_j$.
In particular, we see that $\frak{C}_j$ carries a Type B representation for all $j$ even.
From the discussion above, we moreover know that $\frak{C}_j$ is defined over $\bbQ$.
Theorem \ref{thm:main2} therefore shows that the tautological Azumaya algebra $\mathcal{A}_{\bbQ(\frak{C}_j)}$ does not extend for all $j$ even, since $\eta_5-1/\eta_5\notin \bbQ$.

Combining the above analysis we have the following theorem stated in the introduction.

\mainexamples*

Note that the conclusion that $S_{N_j}=\{2\}$ follows from the fact that this holds for the figure eight knot complement.

\begin{rem}
These calculations provide an interesting counterpart to the results of Long--Reid \cite{LongReid}, where it is shown that the canonical component of the $\PSL_2(\bbC)$-character variety $\overline{X}(M)$ is a birational invariant of any commensurability class when $M$ is a one cusped hyperbolic $3$-manifold. 
Importantly, as remarked therein, these results only hold for the $\PSL_2(\bbC)$-character variety and not for the $\SL_2(\bbC)$-character variety $X(M)$.

In our setting, the smooth projective completion $\widetilde{C}_1$ of the canonical component of $\frak{C}_1$ of the figure eight knot $N_1$ is an elliptic curve and the smooth projective completion of the canonical component of $\overline{\frak{C}}_1$ is $\mathbb{P}^1_{\bbC}$.
When $j$ is even, the above shows that the same is true for all $N_j$ however when $j$ is odd the above shows that both a canonical component for $X(N_j)$ and a canonical component for $\overline{X}(N_j)$ are $\mathbb{P}_{\bbC}^1$.
In particular, the failure within a commensurability class (and even an infinite sequence of finite covers) for a canonical component in $X(N_j)$ to be a birational invariant can be infinite.
\end{rem}

%~~~~~~~~~~~~~~~~~~~~~~~~~~~~~~~~~~~~~~~~~~~~~~~~~~~~~~~~~~~
%~~~~~~~~~~~~~~~~~~~~~~~~~~~~~~~~~~~~~~~~~~~~~~~~~~~~~~~~~~~
%~~~~~~~~~~~~~~~~~~~~~~~~~~~~~~~~~~~~~~~~~~~~~~~~~~~~~~~~~~~

\subsection{An example of Dunfield}\label{sec:dunfield}
We now recall an example of Dunfield \cite{Dunfield}, showing that the field of definition of a canonical component for a once punctured torus need not be $\bbQ$ and discuss its ramifications on the corresponding Azumaya algebra.
We content ourselves to recount the important facets of Dunfield's computations here, as opposed to the entirety of it.

Let $M$ denote the once punctured torus bundle with monodromy $\iota L^2R^2$ so that 
$$\Phi=-\begin{pmatrix}
1&1\\
0&1\end{pmatrix}^2\begin{pmatrix}
1&0\\
1&1\end{pmatrix}^2=\begin{pmatrix}
-1&-2\\
-2&-5\end{pmatrix}.$$
This is the Census manifold $m135$ and has presentation
$$\pi_1(M)=\langle x,y,t\mid txt^{-1}=y^{-1}x^{-1}y^{-1},~tyt^{-1}=yx(y^{-1}x^{-1}y^{-1})^3\rangle,$$
for its fundamental group,
As this group is $3$-generated, we have that $\frak{X}(M)_\bbQ\subset \bbQ^7$ with coordinates given by $\Tr(\rho(x))$, $\Tr(\rho(y))$, $\Tr(\rho(t))$, $\Tr(\rho(xy))$, $\Tr(\rho(xt))$, $\Tr(\rho(yt))$, $\Tr(\rho(xyt))$ (see \cite[Prop 1.4.1]{CS1}).

Let $\alpha=\Tr(\rho(x))$, $\beta=\Tr(\rho(y))$, $\gamma=\Tr(\rho(xy))$, and $\tau=\Tr(\rho(t))$.
In \cite[\S 5]{Dunfield}, Dunfield shows that if $X'$ denotes the subset of $\frak{X}(M)_\bbQ$ such that $\chi_\rho$ lies on a component of $\frak{X}(M)_\bbQ$ corresponding to characters of irreducible representations, then the following relations hold
\begin{align*}\gamma\neq 0,\quad 2\alpha&=\beta\gamma,\quad \gamma^2=2(\alpha^2-2),\\
 \Tr(\rho(xt))=\frac{\beta\tau}{\gamma}, \quad\Tr(\rho(yt))&=\frac{2\beta\tau(\alpha^2-3)}{\gamma^2}, \quad\Tr(\rho(xyt))=\frac{2\tau}{\gamma}.
\end{align*}
From our presentation of $\pi_1(M)$, it follows that $txt^{-1}y=y^{-1}x^{-1}$ and from trace identities one deduces that
$$2\gamma-\alpha\beta=\Tr(\rho(xyt))\Tr(\rho(t))-\Tr(\rho(yt))\Tr(\rho(xt))=\frac{2\tau^2}{\gamma}-\frac{2\beta^2\tau^2(\alpha^2-3)}{\gamma^3}.$$
Using the relations $2\alpha=\beta\gamma$, $\gamma^2=2(\alpha^2-2)$ to write the previous equation strictly in terms of $\tau$ and $\gamma$ one sees that $X'$ is cut out by the equation
$$-4\tau^2(\gamma^2-4)=\gamma^4(\gamma^2-4).$$
Upon basechanging to $k=\bbQ(i)$, one finds that this has $4$ geometrically integral closed subschemes given by the vanishing sets
$$
\frak{C}_1=\{2\tau=-i\gamma^2\},\quad \frak{C}_2=\{2\tau=i\gamma^2\},\quad \frak{C}_3=\{\gamma=2\},\quad \frak{C}_4=\{\gamma=-2\},
$$
where $X_1=\frak{C}_1\cup \frak{C}_2$ in the language of \cite[\S 5]{Dunfield}.
The subschemes $\frak{C}_3$ and $\frak{C}_4$ cannot be canonical components over $\bbC$ since the above relations show that $\gamma^2=4$ implies $\alpha^2=\beta^2=4$ and hence $\Tr(\rho([x,y]))=2$.
This shows that $\rho\vert_{\langle x,y\rangle}$ is reducible on all of $\frak{C}_3$, $\frak{C}_4$ and hence cannot contain the character of a discrete faithful representation.

The subschemes $\frak{C}_1$ and $\frak{C}_2$ are therefore the two canonical components of $\frak{X}(M)$, which differ by a choice of sign for $\rho(t)$ when lifting the discrete faithful representation from the $\PSL_2(\bbC)$-character scheme $\overline{\frak{X}}(M)$.
Let $\frak{C}$ be a choice of either $\frak{C}_1$ or $\frak{C}_2$, then we show that the Azumaya algebra $\mathcal{A}_{k(\widetilde{C})}$ does not extend over $\frak{C}$.
For this, it suffices to see that $\frak{C}$ contains the character of a Type A representation.

The monodromy matrix $\Phi$ has characteristic polynomial $x^2+6x+1$ which has roots given by $\lambda_{\pm}=-3\pm 2\sqrt{2}$.
Recall that if $\rho$ is a Type A representation, then $\rho(t)$ has eigenvalues $w_\pm$, $1/w_{\pm}$, where $w_{\pm}^2=-3\pm 2\sqrt{2}$, and $\rho(xy)$ has eigenvalue $1$ with multiplicity $2$.
Therefore at a Type A representation we have $\gamma=\Tr(\rho(xy))=2$ and $\tau=\Tr(\rho(t))=w_\pm+\frac{1}{w_\pm}$.
It follows from the equation
$$\left(w_\pm+\frac{1}{w_\pm}\right)^2+4=w_\pm^2+\frac{1}{w_\pm^2}+6=-3\pm 2\sqrt{2}+\frac{1}{-3\pm 2\sqrt{2}}+6=0,$$
that for any such choice of $w_\pm$, $\tau=\pm 2 i$.
Each such choice of $\tau$ is possible and hence we see that there are characters of Type A representations on both $\frak{C}_1$ and $\frak{C}_2$.
Proposition \ref{prop:noextendA} then allows us to conclude that, despite the fact that the field of definition of $\frak{C}$ is imaginary quadratic and not $\bbQ$, the Azumaya algebra still does not extend.

\begin{rem}\label{rem:dunfield}
Note that this is an example where the field of definition in the sense of Long--Reid \cite{LongReid2} differs from the field of definition in the sense of algebraic geometry. 
Indeed the canonical component $\frak{C}$ is defined over $\bbQ(i)$, which is the field of definition in the sense of Long--Reid, in this example however the field of definition in the sense of algebraic geometry is simply $\bbQ$.  
\end{rem}

%~~~~~~~~~~~~~~~~~~~~~~~~~~~~~~~~~~~~~~~~~~~~~~~~~~~~~~~~~~~
%~~~~~~~~~~~~~~~~~~~~~~~~~~~~~~~~~~~~~~~~~~~~~~~~~~~~~~~~~~~
%~~~~~~~~~~~~~~~~~~~~~~~~~~~~~~~~~~~~~~~~~~~~~~~~~~~~~~~~~~~

\section{Some final remarks}\label{sec:questions}

The work of \cite{CRS} and this paper leave open several interesting questions and so in this concluding section we make some final remarks and ask some natural follow-up questions.
Throughout we will specialize to the case that $\frak{C}$ is a canonical component of $\frak{X}(M)_k$ with field of definition $k$.

First, assume that the tautological Azumaya algebra extends to $\mathcal{A}_{\widetilde{C}}$ so that there is a finite set of rational primes $S_M$ as in Theorem \ref{thm:main1}.
If $p=\chi_\rho$ is a character associated to a representation $\rho$ of geometric significance (e.g., the holonomy representation of $M$ or a Dehn filling representation) then there may be a discrepancy between the specialization of $\mathcal{A}_{\widetilde{C}}$ at $p$, that is, $\mathcal{A}_{\widetilde{C}}\otimes_{k_\rho} k(p)$ and certain natural commensurability invariants.

Specifically, if $M_{p/q}$ is $p/q$-Dehn filling on $M$ and $\rho$ is the associated holonomy representation then the algebra $\mathcal{A}_{\widetilde{C}}\otimes_{k_\rho} k(p)$ may or may not agree with the invariant quaternion algebra of $M_{p/q}$. 
This owes to a few different issues, including the discrepancy between $\pi_1(M)$ and $\pi_1(M)^{(2)}$ (as shown in \cite[\S 3.3]{MR}) and possibly the difference between the field of definition $k$ and the invariant trace field (though we admittedly don't know of any examples where the latter is actually an issue).
As $p/q$ vary, Theorem \ref{thm:main1} shows that the algebras $\mathcal{A}_{\widetilde{C}}\otimes_{k_\rho} k(p)$ can only have finite ramification at primes lying over the rational primes in $S_M$ and one might hope that the same would be true of the ramification of the corresponding invariant quaternion algebras. 
Unfortunately, this seems not to be the case, for instance one can check numerically using Snappy \cite{snappy} that for Dehn surgeries of the example from Section \ref{sec:tunnelnumone} when $k=4$, there seems to be no such finite set. 
This begs the following open ended question:
\begin{qtn}
Is there a formulation of the Azumaya algebra machinery that encodes the invariant quaternion algebra at each Dehn filling representation?
\end{qtn}
The reader familiar with the discrepancy between the invariant quaternion algebra and non-invariant quaternion algebra might guess that the failure for the example when $k=4$ from Section \ref{sec:tunnelnumone} is due to the presence of $2$-torsion in the homology of the associated once punctured torus bundle.
However, even for hyperbolic knot complements, whose invariant and non-invariant quaternion algebras agree, this phenomenon can occur.
Indeed, the tautological Azumaya algebra extends for the figure-eight knot, where $S_M=\{2\}$, however there are Dehn fillings where the invariant quaternion algebra is ramified at primes not lying over $2$, e.g., the $(6,0)$ and $(8,0)$ filling in Snap \cite{snap}.
This seems to be due to the fact that, whereas the first homology of the figure-eight knot has no $2$-torsion, the same cannot be said of the Dehn filled orbifolds.

We now turn our attention to the non-existence of a finite set $S$ in the case that $\mathcal{A}_{k(\widetilde{C})}$ does not extend to an Azumaya algebra over $\widetilde{C}$.
In this case, we know that there is an infinite set of irreducible representations $\{\rho_i\}\subset\frak{C}$ such that for each $i$, there exists a rational prime $\ell_i$ so that $\mathcal{A}_{\widetilde{C}}\otimes_{k_{\rho_i}} k(p)$ is ramified at a finite place lying over $\ell_i$ and for which $\ell_j\neq \ell_i$ for all $j\neq i$.
The existence of this sequence when $\mathcal{A}_{k(\widetilde{C})}$ does not extend is an existential result due to Harari \cite{Harari} but it does not give any control over which representations lie in the infinite set $\{\rho_i\}$.
In particular, a priori these representations could have no geometric significance.
However numerically it seems to be the case that such representations can be chosen to be Dehn fillings and, in the case that $M$ is a hyperbolic knot complement, this was indeed conjectured by Chinburg--Reid--Stover \cite[Conj 6.7]{CRS}.
The first provable evidence for this conjecture was recently given by Rouse \cite{Rouse} for the knot $7_4$.
We find it natural to repeat this conjecture in our setting and additionally give a sharper form for once punctured torus bundles.
\begin{conjec}
Let $M$ be a hyperbolic once punctured torus bundle and $\mathcal{A}_{k(\widetilde{C})}$ the tautological Azumaya algebra.
If $\mathcal{A}_{k(\widetilde{C})}$ does not extend, then there exists an infinite set of representations $\{\rho_i\}\subset\frak{C}$ where each $\rho_i$ is the holonomy of a distinct $p_i/q_i$ hyperbolic Dehn fillings of $M$ and an infinite sequence of distinct rational primes $\ell_i$ such that the (non-invariant) quaternion algebra associated to $\rho_i$ ramifies at a finite prime lying over $\ell_i$.
Moreover, the $\rho_i$ can be taken to be $n_i/0$-fillings for some strictly monotone sequence of natural numbers $\{n_i\}$.
\end{conjec}
\noindent In the above, the first statement is identical to that of \cite[Conj 6.7]{CRS} and we conjecture the second statement for once punctured torus bundles.
Although we have spent less time experimentally checking this conjecture in the hyperbolic knot complement case, recent work of Rouse suggests that the sharper form may already hold there \cite{Rouse2,Rouse}.

There are also some remarks to make that are more specific to this paper and the field of definition $k$ of any canonical component $\frak{C}$.
As was the case in \cite{CRS}, in all of the theorems in this paper we must be careful about stating our conditions using the field $k$.
For instance, instead of writing $\bbQ(\mu+1/\mu)=\bbQ(\mu)$ we must instead write $k(\mu+1/\mu)=k(\mu)$ and we must frequently worry about whether certain roots of unity lie in $k$.
This is unfortunate since $\mu$ and the roots of unity we must worry about are constructive but the field of definition $k$ is more mysterious. 
It is pointed out in \cite[Pg 6]{CRS} that for hyperbolic knot complements, in all known examples the field of definition $k$ is always $\bbQ$, however from Section \ref{sec:dunfield} we know that this need not be the case for once punctured torus bundles. 
Nonetheless, we wonder aloud whether $k$ can still be replaced by $\bbQ$ in this paper as well.

\begin{qtn}
Is it true that the field of definition $k$ is always a totally imaginary extension of $\bbQ$?
Moreover, is it always true that if $n$ denotes the least common multiple of the order of $x$, $y$ in the finite part of $H_1(M,\bbZ)$, then $\eta_n-1/\eta_n\notin k$ unless $\eta_n\in \bbQ$?
\end{qtn}

\noindent If the answer to the first question is affirmative, then it means that the condition that $w-1/w\in k$ is superfluous in Condition $(\star_1)$ since this never holds and if the answer to the second question is affirmative then in means that Condition $(\star_2)$ is also superfluous in Theorem \ref{thm:main2}.
As in \cite{CRS}, in all examples we know of both of these two properties hold. 
Moreover, if the first question has a negative answer, then we would be particularly interested in an example where $w-1/w\in k$.

We finish this section by giving four idle, unrelated questions.
First, the analysis in Section \ref{sec:intmod} is somewhat incomplete, in that from numerical computations it seems to us that the inclusion to $S$ of all primes which divide the order of the finite part of homology in Theorem \ref{thm:intmod} is unnecessary.
Specifically, there should be a way to handle that case in the analysis of Theorem \ref{thm:CRSanalogue}.
We therefore pose it as a question to resolve this.

\begin{qtn}
Can one remove the hypothesis that $S$ contains all primes which divide the order of $n$ from Theorem \ref{thm:intmod}?
\end{qtn}

Second, when one wants to prove converse theorems such as Theorem \ref{thm:main2}, an issue which enters the picture is understanding certain local coordinates around characters of non-abelian reducible representations, such as in Section \ref{sec:convthm}.
In the case of a once punctured torus bundle, we are fortunately able to lean on previous work of Heusener--Porti showing that such points are scheme reduced in $\frak{X}(M)$.
For the case of hyperbolic knot complements, there is the similar theorem in \cite[Thm 4.6]{CRS} however only when $\mu^2$ is a simple root of the Alexander polynomial, which by work of Heusener--Porti--Suarez \cite{HPS} implies that characters of non-abelian reducible representations are scheme reduced.
In an example where such a character is singular it would be interesting to know how that affects the extension problem for the tautological Azumaya algebra.

\begin{qtn}
To what extent do singularities in the character variety affect the extension problem for the tautological Azumaya algebra? Is it possible to find an example of a character variety for which the tautological Azumaya algebra does not extend over a (singular) projectivization of the canonical component but does extend over its smooth projective completion?
\end{qtn}

\noindent Note that the converse cannot hold, since an extension over the (singular) projectivization will always pullback to an extension on the smooth projective completion.

Third, we are interested in the extent to which Type C representations actually appear. For instance, in all example we know of, Type C representations never exist on the canonical component.
Moreover, in computational examples, we are only able to find Type C representations which have eigenvalues of order dividing $4$ and we are idly curious whether this is a general phenomenon.

\begin{qtn}
Do there exist characters of Type C representations on any canonical component? Can one find Type C representations of hyperbolic once punctured torus bundles where the image of $x$, $y$ do not have order dividing $4$? Equivalently, we are asking if the eigenvalues of the image of $x$, $y$ under such a representation must be a power of $i$. 
\end{qtn}

Finally, it is natural to ask about the story for $k$-cusped hyperbolic manifolds when $k\ge 2$. 
In such a setting, the canonical component has complex dimension $k$ and the tautological Azumaya algebra still makes sense to define.
However, the study of divisors and therefore extensions of Azumaya algebras seems to become much more intricate.
For example in the case that $k=2$, one needs to know that $\mathcal{A}_{k(C)}$ extends over all algebraic curves in the canonical component.
It would be interesting to understand the story in that setting, so we modestly ask the following related question.

\begin{qtn}\label{ques:amidumb}
Does there exists a $k$-cusped hyperbolic manifold $M$ for $k\ge 2$, such that the corresponding $S_M$ from Theorem \ref{thm:main1} exists?
In particular, is there a $k$-cusped hyperbolic manifold for $k\ge 2$ such that for all Dehn fillings of all $k$ cusps, the finite ramification of the (non-invariant) quaternion algebras lie over a finite set of rational primes $S$?
\end{qtn}

\noindent We admittedly have not tried to find any numerically and such an example would be very interesting.

\bibliographystyle{abbrv}
\bibliography{Biblio}

\begin{thebibliography}{10}

\bibitem{BakerPetersen}
K.~L. Baker and K.~L. Petersen.
\newblock Character varieties of once-punctured torus bundles with tunnel
  number one.
\newblock {\em Internat. J. Math.}, 24(6):1350048, 57, 2013.

\bibitem{Bowditch}
B.~H. Bowditch.
\newblock A variation of {M}c{S}hane's identity for once-punctured torus
  bundles.
\newblock {\em Topology}, 36(2):325--334, 1997.

\bibitem{BLZ}
S.~Boyer, E.~Luft, and X.~Zhang.
\newblock On the algebraic components of the {${\rm SL}(2,\Bbb C)$} character
  varieties of knot exteriors.
\newblock {\em Topology}, 41(4):667--694, 2002.

\bibitem{CRS}
T.~Chinburg, A.~W. Reid, and M.~Stover.
\newblock Azumaya algebras and canonical components.
\newblock {\em Int. Math. Res. Not. IMRN}, (7):4969--5036, 2022.

\bibitem{snap}
D.~Coulson, O.~A. Goodman, C.~D. Hodgson, and W.~D. Neumann.
\newblock Computing arithmetic invariants of 3-manifolds.
\newblock {\em Experiment. Math.}, 9(1):127--152, 2000.

\bibitem{Culler}
M.~Culler.
\newblock Lifting representations to covering groups.
\newblock {\em Adv. in Math.}, 59(1):64--70, 1986.

\bibitem{snappy}
M.~Culler, N.~M. Dunfield, M.~Goerner, and J.~R. Weeks.
\newblock Snap{P}y, a computer program for studying the geometry and topology
  of $3$-manifolds.
\newblock Available at \url{http://snappy.computop.org} (08/07/2022).

\bibitem{CS1}
M.~Culler and P.~B. Shalen.
\newblock Varieties of group representations and splittings of {$3$}-manifolds.
\newblock {\em Ann. of Math. (2)}, 117(1):109--146, 1983.

\bibitem{Dunfield}
N.~M. Dunfield.
\newblock Examples of non-trivial roots of unity at ideal points of hyperbolic
  {$3$}-manifolds.
\newblock {\em Topology}, 38(2):457--465, 1999.

\bibitem{Gabber}
K.~Fujiwara.
\newblock A proof of the absolute purity conjecture (after {G}abber).
\newblock In {\em Algebraic geometry 2000, {A}zumino ({H}otaka)}, volume~36 of
  {\em Adv. Stud. Pure Math.}, pages 153--183. Math. Soc. Japan, Tokyo, 2002.

\bibitem{EGA}
A.~Grothendieck.
\newblock \'{E}l\'{e}ments de g\'{e}om\'{e}trie alg\'{e}brique. {IV}. \'{E}tude
  locale des sch\'{e}mas et des morphismes de sch\'{e}mas. {II}.
\newblock {\em Inst. Hautes \'{E}tudes Sci. Publ. Math.}, (24):231, 1965.

\bibitem{GroBrauer3}
A.~Grothendieck.
\newblock Le groupe de {B}rauer. {III}. {E}xemples et compl\'{e}ments.
\newblock In {\em Dix expos\'{e}s sur la cohomologie des sch\'{e}mas}, volume~3
  of {\em Adv. Stud. Pure Math.}, pages 88--188. North-Holland, Amsterdam,
  1968.

\bibitem{Grothendieck}
A.~Grothendieck.
\newblock Le groupe de {B}rauer. {I}. {A}lg\`ebres d'{A}zumaya et
  interpr\'{e}tations diverses [ {MR}0244269 (39 \#5586a)].
\newblock In {\em S\'{e}minaire {B}ourbaki, {V}ol. 9}, pages Exp. No. 290,
  199--219. Soc. Math. France, Paris, 1995.

\bibitem{Harari}
D.~Harari.
\newblock M\'{e}thode des fibrations et obstruction de {M}anin.
\newblock {\em Duke Math. J.}, 75(1):221--260, 1994.

\bibitem{Hartshorne}
R.~Hartshorne.
\newblock {\em Algebraic geometry}.
\newblock Springer-Verlag, New York-Heidelberg, 1977.
\newblock Graduate Texts in Mathematics, No. 52.

\bibitem{HeusenerPorti2}
M.~Heusener and J.~Porti.
\newblock The variety of characters in {${\rm PSL}_2(\Bbb C)$}.
\newblock {\em Bol. Soc. Mat. Mexicana (3)}, 10(Special Issue):221--237, 2004.

\bibitem{HeusenerPorti}
M.~Heusener and J.~Porti.
\newblock Deformations of reducible representations of 3-manifold groups into
  {${\rm PSL}_2(\Bbb C)$}.
\newblock {\em Algebr. Geom. Topol.}, 5:965--997, 2005.

\bibitem{HPS}
M.~Heusener, J.~Porti, and E.~Su\'{a}rez~Peir\'{o}.
\newblock Deformations of reducible representations of 3-manifold groups into
  {${\rm SL}_2(\bold C)$}.
\newblock {\em J. Reine Angew. Math.}, 530:191--227, 2001.

\bibitem{Jorgensen}
T.~J\o~rgensen.
\newblock On pairs of once-punctured tori.
\newblock In {\em Kleinian groups and hyperbolic 3-manifolds ({W}arwick,
  2001)}, volume 299 of {\em London Math. Soc. Lecture Note Ser.}, pages
  183--207. Cambridge Univ. Press, Cambridge, 2003.

\bibitem{Liu}
Q.~Liu.
\newblock {\em Algebraic geometry and arithmetic curves}, volume~6 of {\em
  Oxford Graduate Texts in Mathematics}.
\newblock Oxford University Press, Oxford, 2002.
\newblock Translated from the French by Reinie Ern\'{e}, Oxford Science
  Publications.

\bibitem{LongReid}
D.~D. Long and A.~W. Reid.
\newblock Commensurability and the character variety.
\newblock {\em Math. Res. Lett.}, 6(5-6):581--591, 1999.

\bibitem{LongReid2}
D.~D. Long and A.~W. Reid.
\newblock Fields of definition of canonical curves.
\newblock In {\em Interactions between hyperbolic geometry, quantum topology
  and number theory}, volume 541 of {\em Contemp. Math.}, pages 247--257. Amer.
  Math. Soc., Providence, RI, 2011.

\bibitem{LubotzkyMagid}
A.~Lubotzky and A.~R. Magid.
\newblock Varieties of representations of finitely generated groups.
\newblock {\em Mem. Amer. Math. Soc.}, 58(336):xi+117, 1985.

\bibitem{MR2}
C.~Maclachlan and A.~W. Reid.
\newblock Generalised {F}ibonacci manifolds.
\newblock {\em Transform. Groups}, 2(2):165--182, 1997.

\bibitem{MR}
C.~Maclachlan and A.~W. Reid.
\newblock {\em The arithmetic of hyperbolic 3-manifolds}, volume 219 of {\em
  Graduate Texts in Mathematics}.
\newblock Springer-Verlag, New York, 2003.

\bibitem{MorganShalen3}
J.~W. Morgan and P.~B. Shalen.
\newblock Degenerations of hyperbolic structures. {III}. {A}ctions of
  {$3$}-manifold groups on trees and {T}hurston's compactness theorem.
\newblock {\em Ann. of Math. (2)}, 127(3):457--519, 1988.

\bibitem{Poonen}
B.~Poonen.
\newblock {\em Rational points on varieties}, volume 186 of {\em Graduate
  Studies in Mathematics}.
\newblock American Mathematical Society, Providence, RI, 2017.

\bibitem{Rouse2}
N.~Rouse.
\newblock A conjecture of chinburg-reid-stover for surgeries on twist knots.
\newblock arXiv: 2109.04750.

\bibitem{Rouse}
N.~Rouse.
\newblock Arithmetic of the canonical component of the knot {$7_4$}.
\newblock {\em New York J. Math.}, 27:1494--1523, 2021.

\bibitem{Sakuma}
M.~Sakuma.
\newblock Unknotting tunnels and canonical decompositions of punctured torus
  bundles over a circle.
\newblock Number 967, pages 58--70. 1996.
\newblock Analysis of discrete groups (Kyoto, 1995).

\bibitem{Sikora}
A.~S. Sikora.
\newblock Character varieties.
\newblock {\em Trans. Amer. Math. Soc.}, 364(10):5173--5208, 2012.

\bibitem{Thurston}
W.~P. Thurston.
\newblock {\em Three-dimensional geometry and topology. {V}ol. 1}, volume~35 of
  {\em Princeton Mathematical Series}.
\newblock Princeton University Press, Princeton, NJ, 1997.
\newblock Edited by Silvio Levy.

\bibitem{Weil}
A.~Weil.
\newblock Remarks on the cohomology of groups.
\newblock {\em Ann. of Math. (2)}, 80:149--157, 1964.

\end{thebibliography}
\end{document}